\documentclass[12pt]{amsart} % for a short document
\usepackage[margin=3cm]{geometry}
\usepackage[english]{babel}
\usepackage{xcolor}
\usepackage{hyperref}
\makeatletter
\def\l@subsection{\@tocline{2}{0pt}{2.75pc}{3pc}{}}
\makeatother
\definecolor{deepsea}{RGB}{0, 139, 139}
\hypersetup{
    colorlinks=true,       % Removes the ugly boxes entirely
    linkcolor=blue,    % Table of contents, equation references, etc.
    citecolor=cyan,    % Citation links [1]
    urlcolor=cyan      % Website/URL links
}
\usepackage{amsmath}
\usepackage{amsthm}
\usepackage{amssymb}
\usepackage{bbm}
\usepackage{ucs}
\usepackage{enumerate}
\usepackage{tikz}
\usepackage{tikz-cd}
\usepackage[utf8]{inputenc}
\usepackage[T1]{fontenc}
\usepackage{xcolor}
\usepackage{mathtools}
\usepackage{comment}
\usepackage{dsfont}
\usepackage{faktor}
\usepackage{xcolor}
\usepackage[draft,inline,nomargin,index]{fixme}
\fxsetup{theme=color,mode=multiuser}
\FXRegisterAuthor{pf}{apf}{\color{olive}PF}
\FXRegisterAuthor{flm}{aflm}{\color{red}FLM}
\FXRegisterAuthor{km}{akm}{\color{purple}KM}
\FXRegisterAuthor{ip}{aip}{\color{blue}IP}
\AtBeginDocument{%
   \def\MR#1{}
}
\theoremstyle{plain}
\newtheorem{theorem}{Theorem}[section]
\newtheorem{theoremA}{Theorem}

\newtheorem{lemma}[theorem]{Lemma}
\newtheorem{proposition}[theorem]{Proposition}
\newtheorem{corollary}[theorem]{Corollary}
\theoremstyle{definition}
\newtheorem{definition}[theorem]{Definition}
\newtheorem{example}[theorem]{Example}

\newtheorem{remark}[theorem]{Remark}

\newcommand{\R}{\mathbb{R}}

\newcommand{\C}{\mathbb{C}}
\newcommand{\N}{\mathbb{N}}

\newcommand{\B}{\mathbb{B}}
\newcommand{\Z}{\mathbb{Z}}
\newcommand{\Lcal}{\mathbb{B}}

\newcommand{\ot}{\otimes}
\newcommand{\id}{{\rm id}}

\newcommand{\LL}{{\rm L}}

\newcommand{\F}{\mathbb{F}}
\newcommand{\Aut}{{\rm Aut}}
\newcommand{\Ad}{{\rm Ad}}

\newcommand{\Ucal}{\mathcal{U}}

\DeclareMathOperator{\Sub}{Sub}
\DeclareMathOperator{\Stab}{Stab}

\DeclareMathOperator{\Prob}{Prob}

\DeclareMathOperator{\Hyp}{Hyp}
\DeclareMathOperator{\Homeo}{Homeo}

\DeclareMathOperator{\Proj}{Proj}

\newcommand{\dist}{\mathrm{dist}}

\title[Uniformly recurrent subalgebras]{Uniformly recurrent subalgebras in finite von Neumann algebras}
\date{}
 \author[Amrutam]{Tattwamasi Amrutam}
\address{Tattwamasi Amrutam
\newline
Institute of Mathematics of the Polish Academy of Sciences, ul. Sniadeckich 8, 00-656, Warszawa, Poland}
\email{tattwamasiamrutam@gmail.com}
\thanks{T.A. is supported by the National Science Centre, Poland, Sonata grant number 2025/59/D/ST1/03117.}
\author[Fima]{Pierre Fima}
\address{Pierre Fima
\newline
Universit\'e Paris Cit\'e and Sorbonne Universit\'e, CNRS, IMJ-PRG, F-75013 Paris, France}
\email{pierre.fima@imj-prg.fr}
\thanks{P.F. is partially supported by the CNRS IEA GAOA and ANR-25-CE40-5010 (CroCQG)}
\author[Jiang]{Yongle Jiang}
\address{Yongle Jiang
\newline
School of Mathematical Sciences, Dalian University of Technology, Dalian, 116024, China}
\email{yonglejiang@dlut.edu.cn}
\thanks{Y.J. is supported by National Natural Science Foundation of China, grant number 12471118}
\begin{document}
\begin{abstract}
We introduce the notion of a uniformly recurrent subalgebra (URA) for a trace-preserving action of a countable discrete group $\Gamma$ on a finite von Neumann algebra $M$, providing an operator-algebraic counterpart to the theory of uniformly recurrent subgroups (URS). We also show that the Effros-Mar\'echal space $\Sub(M)$ is compact if and only if $M$ lacks a diffuse direct summand. Leveraging this, we show that URAs can exhibit arbitrary topological complexity and construct exotic URAs homeomorphic to any prescribed minimal Polish space. In the context of crossed products $M \rtimes \Gamma$ with amenable coefficients, we utilize URAs to formulate a new characterization of C*-simplicity, proving that $\Gamma$ is C*-simple if and only if the only amenable URA of the crossed product containing $M$ is $\{M\}$. Finally, to bypass the failure of compactness in $\Sub(M)$, we develop a generalized state-space machinery on the weak-* compact space of trace-extending states. This construction captures compact, discrete, and exotic URAs, while recovering the classical URS framework as a special case.
\end{abstract}
\maketitle
\tableofcontents
%%%%%%%%%%%%%%%%%%%%%%%%%%%%%%%%%%%%%%%%%%%%%%%%%%%%%%%%%%%%%%%%%%%%%%%%%%%%%%%%%%%%%%%%%%%%%%%%%%%%%%%%
\section{Introduction}
\label{sec:intro}
%%%%%%%%%%%%%%%%%%%%%%%%%%%%%%%%%%%%%%%%%%%%%%%%%%%%%%%%%%%%%%%%%%%%%%%%%%%%%%%%%%%%%%%%%%%%%%%%%%%%%%%%
%%%%%%%%%%%%%%%%%%%%%%%%%%%%%%%%%%%%%%%%%%%%%%%%%%%%%%%%%%%%%%%%%%%%%%%%%%%%%%%%%%%%%%%%%%%%%%%%%%%%%%%%
%%%%%%%%%%%%%%%%%%%%%%%%%%%%%%%%%%%%%%%%%%%%%%%%%%%%%%%%%%%%%%%%%%%%%%%%%%%%%%%%%%%%%%%%%%%%%%%%%%%%%%%%
The study of $\text{Sub}(\Gamma)$ associated with a group $\Gamma$ has opened up new directions of research at the interface of geometric group theory, ergodic theory, and operator algebras. The Chabauty space $\Sub(\Gamma)$ serves as a unifying geometric object for a wide range of rigidity questions. Classical rigidity phenomena---such as Margulis's normal subgroup theorem~\cite{margulis1991discrete} and the Nevo--Stuck--Zimmer theorems~\cite{stuck1994stabilizers, nevo1994boundary, nevo1999homogenous} on lattices in higher-rank semisimple Lie groups---can all be formulated as statements about the $\Gamma$-dynamics on $\Sub(\Gamma)$ via the classification of invariant random subgroups (IRS)~\cite{abert2014kesten, Baderetal2016}. On the topological dynamics side, using the notion of Uniformly Recurrent subgroups introduced by Glasner-Weiss~\cite{glasner2015uniformly}, Kennedy's fundamental theorem \cite{Ke18} characterizes the $C^*$-simplicity of $\Gamma$ in terms of the absence of non-trivial amenable URS. The subsequent works of Le Boudec and Matte Bon \cite{le2018subgroup} have placed URS at the heart of the modern theory of $C^*$-simplicity. More recently, this dynamical perspective has been greatly expanded through the study of stationary random subgroups (SRS) by Fraczyk-Gelander~\cite{fraczyk2023infinite}, illustrating the broader scope of stationary measures and random dynamics on $\Sub(\Gamma)$.

The space of von Neumann subalgebras carries a natural Polish topology whose
history is somewhat winding, and which stands in marked contrast to the
Chabauty space $\Sub(\Gamma)$ discussed above. Whereas $\Sub(\Gamma)$ sits
inside $\{0,1\}^\Gamma$ and is thereby a compact metric space, its operator-algebraic counterpart is in general only Polish, and this loss of compactness will shape the entire theory. The relevant topology has a long
lineage. Mar\'echal \cite{Marechal1973}
introduced a topology on the space of von Neumann algebras acting on a
fixed separable infinite-dimensional Hilbert space, refining the standard Borel
structure that Effros had defined on this space \cite{Effros1965}. Maréchal showed that this topology is Polish, although a
full proof only appeared recently in \cite{fima2024michael}.
Although Mar\'echal already observed that this topology yields streamlined
proofs for results on Borel subsets of the space of von Neumann algebras, the
circle of ideas lay largely dormant until Haagerup and Winsl\o w revived it in
two influential papers. In the first \cite{haagerup1998effros}, they established the
continuity of many natural operations on this space, including the commutant
map; in the second \cite{HW2000}, they uncovered a striking link with Connes'
embedding problem, proving it equivalent to the density of the isomorphism
class of the injective type III$_1$ factor in the space of all von Neumann
algebras acting on a separable infinite-dimensional Hilbert space --- a
formulation resting on the Baire category theorem and on Mar\'echal's theorem
that the ambient space is Polish. It is this Effros--Mar\'echal topology, in the
form carried by the space $\Sub(M)$ of von Neumann subalgebras of a finite von
Neumann algebra $(M,\tau)$ with separable predual, that all of our minimality
and continuity statements take place in.

More recently, operator-algebraic counterparts of the study on $\text{Sub}(\Gamma)$ picture have begun to emerge in two parallel directions. On the invariant-random side, \cite{THO25} develops an operator-algebraic analog of the theory of invariant random subgroups, putting the rigidity story of the Invariant Random Subgroups into a non-commutative setting. On the topological-dynamics side, the recent work \cite{AJ} develops an operator-algebraic analogue of uniformly recurrent subgroups by studying the space $\mathrm{PD}_1(\Gamma)$ of normalized positive definite functions on $\Gamma$, recovering both URS and a corresponding $C^*$-simplicity criterion as orbit closures in the pointwise-compact space $\mathrm{PD}_1(\Gamma)$.
In this paper, we take a complementary approach. The natural map associating a positive definite function $\phi_{M} \in \mathrm{PD}_1(\Gamma)$ to a von Neumann subalgebra $M \le L(\Gamma)$ is not injective. Therefore, rather than studying $\mathrm{PD}_1(\Gamma)$, we develop the operator-algebraic analogue of URS directly inside the space of von Neumann subalgebras. Let $(M,\tau)$ be a finite von Neumann algebra with separable predual; the space $\Sub(M)$ of von Neumann subalgebras of $M$ is a Polish space under the Effros--Mar\'echal topology. Given a trace-preserving action $\Gamma \curvearrowright (M,\tau)$, we define a \emph{uniformly recurrent subalgebra} (URA) to be a minimal closed $\Gamma$-invariant subset of $\Sub(M)$.  The basic question driving this paper is to understand the richness of the world of URAs and the tools available for constructing and classifying them.
Before we can pursue this, however, we must address a structural point. In the URS setting, $\Sub(\Gamma)$ is a compact metric space, and Zorn's lemma immediately produces minimal $\Gamma$-invariant subsets. For von Neumann subalgebras, the situation is markedly different. $\Sub(M)$ is in general only a Polish space, and whether or not it is compact turns out to depend on the type structure of $M$. Our first main theorem gives a complete characterization, along with an embedding theorem that quantifies just how badly compactness can fail.
\begin{theoremA}\label{thm:A}
Let $M$ be a finite von Neumann algebra with separable predual. The following are equivalent.
\begin{enumerate}
\item $M$ has a non-zero diffuse direct summand.
\item For every Polish space $X$, there exists a closed topological embedding
$$X \hookrightarrow \Sub(M).$$
\item $\Sub(M)$ is not compact.
\end{enumerate}
\end{theoremA}
This improves \cite[Theorem 2.5]{AJ}, where it was shown that $\Sub(M)$ is non-compact in the diffuse case. The failure of compactness is the structural obstruction that shapes the entire theory developed in this paper. Despite this obstruction, URAs are abundant and exhibit remarkable structural diversity. URAs in $\Sub(L(\Gamma))$ can be wild and exotic in a way that has no counterpart on the subgroup side since the Chabauty space $\Sub(\Gamma)$ is a closed subset of $\{0,1\}^\Gamma$ and is, in particular, compact. In contrast, $\Sub(L(\Gamma))$ supports URAs with any prescribed minimal Polish topology. We make this precise in our second main theorem, which can be regarded as a wildness statement complementing Theorem~\ref{thm:A}. We say that a URA in $\Sub(L(\Gamma))$ is \emph{exotic} if it is not of the form $\{L(\Lambda) : \Lambda \in X\}$ for some URS $X$.
\begin{theoremA}\label{thm:B}
Let $\Lambda$ be a countably infinite abelian group. For every minimal Polish space $X$, there exists an exotic URA $\mathcal{U}_X \subset \Sub(L(\Lambda \wr \mathbb{F}_\infty))$ that is homeomorphic to $X$.
\end{theoremA}
\noindent This says that the topology of URAs can be arbitrarily complicated. For example, $\Sub(L(\Gamma))$ also supports discrete URAs that have no analogue inside $\Sub(\Gamma)$. Discreteness in the Chabauty topology forces a URS to be a finite conjugacy class of subgroups, whereas $\Sub(L(\Gamma))$ contains discrete URAs of infinite cardinality. We exhibit two distinct mechanisms producing them (see Proposition~\ref{prop:finite-action-fixed-point-rigidity} and Corollary~\ref{cor:fixed-point-ura} for the first mechanism and Proposition~\ref{prop:finite-subgroup-corner-ura} for the second).

To distinguish these two families, we attach to every URA $\Ucal$ two numerical invariants. The first is the index $[\Gamma:\Stab_\Gamma(N)]$ of the stabilizer of a member $N\in\Ucal$, which does not depend on $N$ because a finite-index stabilizer forces the action on $\Ucal$ to be transitive (Lemma~\ref{lem:group-index-of-ura}). The second is the Pimsner--Popa index $[M:N]$, which does not depend on $N$ either, this time because the Pimsner--Popa constant is upper semicontinuous on $\Sub(M)$ and invariant under trace-preserving automorphisms (Lemma~\ref{lem:pimsner-popa-index-of-ura}).

Neither discreteness nor compactness is a formal consequence of these constructions.  For the locally finite i.c.c.\ group $\left(\bigoplus_{\N}\Z/3\Z\right)\rtimes S_\infty$ the same fixed-point recipe produces a \emph{compact} exotic URA homeomorphic to the Cantor space, of Pimsner--Popa index $2$ (Proposition~\ref{prop:fixed-point-non-fg-not-discrete}). Nor does finite Pimsner--Popa index by itself produce isolated points (see Proposition~\ref{ex:finitely-generated-nonisolation}). Finally, a URA need not be either compact or discrete; Example~\ref{ex:discontinuous-fiber-map} produces such a URA from a $\Gamma$-invariant continuous field of $C^*$-algebras over an irrational rotation. The full structural diversity of URAs --- compact, exotic, and discrete --- is the object of the subsequent sections (Sections~\ref{sec:URA}, \ref{subsec:compactURAs}, \ref{subsec:discreteURAs}, \ref{subsec:exoticURAs}, and~\ref{subsec:continuousfields}).

We now explain the construction underlying many of our URA examples, which is the content of Theorem~\ref{thm:D}. The natural attempt is to imitate \cite{AJ} and study orbit closures of $\phi_M$ inside the pointwise-compact space $\mathrm{PD}_1(\Gamma)$. This works perfectly at the level of subgroups, but at the level of subalgebras, the program runs into the obstruction quantified by Theorem~\ref{thm:A} since the pointwise limits of $\phi_M$ need not arise from any von Neumann subalgebra, and the natural ambient space --- subalgebras of $L(\Gamma)$ --- is generically not compact. To bypass this obstruction, we develop a highly general state-space centralizer construction. Let $(M,\tau)$ be a finite von Neumann algebra faithfully represented on a Hilbert space $\mathcal{H}$, equipped with a trace-preserving $\Gamma$-action. Let $A \subset \mathbb{B}(\mathcal{H})$ be a chosen $\Gamma$-invariant unital $C^*$-algebra. Rather than working directly in $\Sub(M)$, we enlarge the ambient space to the weak-$*$ compact convex set of states on $\mathbb{B}(\mathcal{H})$ extending the canonical trace:
$$ K = \{\varphi \in S(\mathbb{B}(\mathcal{H})) : \varphi|_{M} = \tau\}.$$
Zorn's lemma produces minimal closed $\Gamma$-invariant subsets $\mathcal{Z} \subset K$. To each $\varphi \in \mathcal{Z}$ we associate its centralizer $M_\varphi \subset M$ relative to $A$, defined as the von Neumann algebra generated by elements $z \in M$ satisfying $\varphi(zy) = \varphi(yz)$ for all $y \in A$. \iffalse The centralizer map $\Phi: \mathcal{Z} \to \Sub(M)$ given by $\Phi(\varphi) = M_\varphi$ is only upper semi-continuous, but a Baire-category argument supplies a dense $G_\delta$ set of continuity points.\fi
The centralizer map $\Phi: \mathcal{Z} \to \Sub(M)$ given by $\Phi(\varphi) = M_\varphi$ is in general not continuous; whenever it is continuous at some point of a minimal state space, the Effros--Mar\'echal orbit closure of the associated subalgebra is a URA. In the compact and discrete realization theorems below, the required continuity is verified directly. 
\begin{theoremA}\label{thm:D}
Let $\mathcal{Z} \subset K$ be a minimal closed $\Gamma$-invariant subset and let $\Phi:\mathcal{Z}\to\Sub(M)$, $\Phi(\varphi)=M_\varphi$, be the centralizer map relative to $A$.
\begin{enumerate}
\item For every continuity point $\varphi_0 \in \mathcal{Z}$ of $\Phi$, the Effros--Mar\'echal orbit closure
$$ \mathcal{U} = \overline{ \{ \alpha_g(M_{\varphi_0}) : g \in \Gamma \} }^{\,\mathrm{EM}} $$
is a minimal closed $\Gamma$-invariant subset of $\Sub(M)$, and thus a URA; moreover $\mathcal{U}=\overline{\Phi(\operatorname{Cont}(\Phi))}^{\,\mathrm{EM}}$, so it does not depend on the chosen continuity point.
\item If $\Phi(\mathcal{Z})$ is relatively compact in $\Sub(M)$, then the set $\operatorname{Cont}(\Phi)$ of continuity points of $\Phi$ is a dense $G_\delta$ subset of $\mathcal{Z}$.
\end{enumerate}
Conversely, every compact URA arises as in {\rm(1)}, for a suitable choice of the representation, of the test algebra $A$ and of the minimal state space $\mathcal{Z}$, and with $\Phi$ continuous on all of $\mathcal{Z}$.
\end{theoremA}
Continuity points are, however, not scarce. As soon as the range $\Phi(\mathcal Z)$ is relatively compact in $\Sub(M)$, they form a dense $G_\delta$ subset of $\mathcal Z$, by a Baire-category argument applied to the coordinate maps $\varphi\mapsto\Vert E_{M_\varphi}(x)\Vert_2$. On the other hand, a single continuity point already forces continuity everywhere when the induced action on the closure of the range is equicontinuous (Proposition~\ref{prop:equicontinuous-centralizer-continuity}).
%The space $K$ is weak-$*$ compact, so Zorn's lemma produces minimal closed $\Gamma$-invariant subsets $\mathcal{Z} \subset K$. 

As mentioned above, every \emph{compact} URA arises from the minimal-state-space construction of Theorem~\ref{thm:D} (see Theorem~\ref{ThmCompactURA}). In particular, every URS arises from this state-space machinery, recovering the Glasner--Weiss classification in our framework (Remark~\ref{rmk:urs-compact-state-space-realization}). Moreover, every \emph{discrete} URA also arises from a broader version of the construction in which the ambient state space is no longer required to be minimal (Theorem~\ref{thm:discrete-uras}). This broadening is necessary, and the exact point at which it becomes necessary can be located: for a discrete URA $\Ucal$, the following four conditions are equivalent (Proposition~\ref{prop:discrete-minimal-state-space-characterization}). $\Ucal$ is the orbit closure of $\Phi(z_0)$ for some continuity point $z_0$ of some $\Gamma$-equivariant map $\Phi$ defined on a compact minimal $\Gamma$-space; $\Ucal$ is finite; $\Ucal$ is compact; $[\Gamma:\Ucal]<\infty$. In particular, an infinite discrete URA is invisible to every construction of the type of Theorem~\ref{thm:D}, and this is precisely what the non-minimal variant of Theorem~\ref{thm:discrete-uras} repairs, by using isolated points of a non-minimal compact $\Gamma$-invariant set of states as continuity points of the centralizer map. The dichotomy is itself a manifestation of the compactness obstruction of Theorem~\ref{thm:A}.

We now turn to applications. Given a trace-preserving action $\Gamma \curvearrowright (M,\tau_M)$ on an amenable finite von Neumann algebra, we show that C*-simplicity of $\Gamma$ is equivalent to the absence of non-trivial amenable URAs in the crossed product $M \rtimes \Gamma$ containing $M$. This extends \cite{AJ} from the trivial coefficient algebra $\C$ to arbitrary amenable coefficients, and at the same time provides a URA-theoretic characterization of C*-simplicity. Our result also generalizes Kawabe's work \cite{kawabe2017uniformly} on the ideal structure of $C(X) \rtimes_r \Gamma$ for minimal $\Gamma$-spaces $X$.
\begin{theoremA}\label{thm:C}
Let $\Gamma \curvearrowright (M,\tau_M)$ be a trace-preserving action of a countable discrete group on an
amenable finite von Neumann algebra with separable predual, and  $X$ be a non-empty compact Hausdorff minimal $\Gamma$-space. The following are equivalent.
\begin{enumerate}
\item The reduced crossed product $C(X) \rtimes_r \Gamma$ is simple.
\item For all $x \in X$, any amenable von Neumann subalgebra $N\subset M\rtimes\Gamma_x$ containing $M$ is not $M$-confined in the crossed product $M\rtimes\Gamma$.
\item There exists $x \in X$ such that any amenable von Neumann subalgebra $N\subset M\rtimes\Gamma_x$ containing $M$ is not $M$-confined in the crossed product $M\rtimes\Gamma$.
\item For all $x\in X$, the only amenable URA $\Ucal$ of $M\rtimes\Gamma$ containing $M$ and such that $\Ucal\cap\Sub(M\rtimes\Gamma_x)\neq\emptyset$  is $\Ucal=\{M\}$.
\item There exists $x\in X$ such that the only amenable URA $\Ucal$ of $M\rtimes\Gamma$ containing $M$ and such that $\Ucal\cap\Sub(M\rtimes\Gamma_x)\neq\emptyset$  is $\Ucal=\{M\}$.
\end{enumerate}
\end{theoremA}
\noindent Taking $M = \C$ and $X$ to be a point, we recover the main result of \cite{AJ}, while statement $(4)$, an intrinsic characterization in terms of URAs, is new even in that special case. For a general minimal compact $\Gamma$-space $X$ (and $M=\C$), it generalizes Kawabe's work \cite{kawabe2017uniformly} on the ideal structure of $C(X) \rtimes_r \Gamma$. While Kawabe characterized the simplicity of these reduced crossed products using amenable URS of the stabilizer subgroups, we show that this phenomenon is governed by the Effros-Mar\'echal dynamics of amenable subalgebras. Theorem~\ref{thm:C} also feeds back into the state-space machinery. When the test algebra is taken as large as possible, $A=\mathbb B(\mathcal H)$, every centralizer $M_\varphi$ is amenable, so every URA produced by Theorem~\ref{thm:D} is amenable; for $M=L(\Gamma)\subset\mathbb B(\ell^2(\Gamma))$ and $\Gamma$ $C^*$-simple, Theorem~\ref{thm:C} then forces all these URAs to be trivial (Remark~\ref{rmk:maximal-test-algebra-amenable-ura}). The choice of the test algebra $A$ is therefore not a technical convenience but the parameter which controls how much of $\Sub(M)$ the construction can see.
\subsection*{Organization of the paper} The paper is organized as follows. Section~\ref{sec:preliminaries} collects preliminaries on finite von Neumann algebras, the Effros--Mar\'echal topology, the Jones and Pimsner--Popa index, fixed-point subalgebras of automorphisms, and minimal Polish spaces. Section~\ref{sec:compactness} characterizes compactness of $\Sub(M)$ and proves the universal embedding theorem (Theorem~\ref{thm:A}). Section~\ref{sec:URA} develops the basic theory of URAs, including the two index invariants and the universality theorem for exotic URAs (Theorem~\ref{thm:B}). Section~\ref{sec:crossedproduct} treats confined subalgebras and proves the simplicity criterion (Theorem~\ref{thm:C}). Section~\ref{sec:statespace} develops the state-space centralizer construction (Theorem~\ref{thm:D}): Subsection~\ref{subsec:generalconstruction} sets up the construction and delimits it by an example, Subsection~\ref{subsec:compactURAs} shows that it captures every compact URA and, in particular, every URS, Subsection~\ref{subsec:discreteURAs} proves the obstruction for infinite discrete URAs and gives the non-minimal variant which captures all of them, Subsection~\ref{subsec:exoticURAs} constructs the fixed-point and corner families of exotic URAs and compares them through their Pimsner--Popa indices, and Subsection~\ref{subsec:continuousfields} gives a geometric realization by continuous fields producing a URA which is neither compact nor discrete.\iffalse and uses it to capture compact URAs, discrete URAs, exotic URAs, the URS revisited via the construction, the continuous-field realization, and amenable URAs via hypertrace dynamics.\fi
\subsection*{Acknowledgments} We thank Prof. Jesse Peterson for helpful comments.
%%%%%%%%%%%%%%%%%%%%%%%%%%%%%%%%%%%%%%%%%%%%%%%%%%%%%%%%%%%%%%%%%%%%%%%%%%%%%%%%%%%%%%%%%%%%%%%%%%%%%%%%
\section{Preliminaries}
\label{sec:preliminaries}
%%%%%%%%%%%%%%%%%%%%%%%%%%%%%%%%%%%%%%%%%%%%%%%%%%%%%%%%%%%%%%%%%%%%%%%%%%%%%%%%%%%%%%%%%%%%%%%%%%%%%%%%
In this section, we present the background on which the rest of the paper rests. We begin with hypertraces and the characterization of amenability for von Neumann subalgebras, then introduce the Effros--Mar\'echal topology on $\Sub(M)$, and finally record several lemmas on fixed-point subalgebras of finite-order automorphisms and on minimal Polish spaces that will be used repeatedly.
%%%%%%%%%%%%%%%%%%%%%%%%%%%%%%%%%%%%%%%%%%%%%%%%%%%%%%%%%%%%%%%%%%%%%%%%%%%%%%%%%%%%%%%%%%%%%%%%%%%%%%%%
\subsection{Von Neumann algebras and hypertraces}
%%%%%%%%%%%%%%%%%%%%%%%%%%%%%%%%%%%%%%%%%%%%%%%%%%%%%%%%%%%%%%%%%%%%%%%%%%%%%%%%%%%%%%%%%%%%%%%%%%%%%%%%
\noindent Let $(M,\tau)$ be a finite von Neumann algebra. We will always assume that $M\subset\Lcal({\rm L}^2(M))$, where $({\rm L}^2(M),\id,\xi_\tau)$ is  the GNS construction of the trace $\tau$. For a von Neumann subalgebra $N\subseteq M$ we denote by $\Hyp_\tau(N)$ the set of $(N,\tau)$-hypertraces:
$$\Hyp_\tau(N):=\left\{\Phi\in S(\Lcal({\rm L}^2(M))):\Phi\vert_M=\tau,\,\Phi(xT)=\Phi(Tx)\,\forall x\in N,\,T\in\Lcal({\rm L}^2(M))\right\}$$
It is clear that $\Hyp_\tau(N)$ is a convex and weak* compact subset of $S(\Lcal({\rm L}^2(M)))$. Hypertraces provide a flexible operator-theoretic certificate of amenability that will be used throughout the paper.
\begin{proposition}\cite[Proposition 2.4]{THO25}\label{prop:amenability-hypertrace}
Let $(M,\tau)$ be a finite von Neumann algebra with separable predual. For a subalgebra $N\subseteq M$, the following are equivalent.
\begin{enumerate}
\item $N$ is amenable.
\item $\Hyp_\tau(N)\neq\emptyset$.
\end{enumerate}
\end{proposition}
\noindent Actually, the previous Proposition can be improved with a similar proof. We include the full argument for completeness.
\vspace{0.2cm}
\noindent By a \textbf{unital operator system} we mean a not necessarily norm-closed self-adjoint linear subspace of $\mathbb{B}(H)$ containing the identity.
\medskip
\noindent The improvement we need is the possibility of prescribing in advance the values of the hypertrace on an operator system which is strictly larger than $N$ itself. This is what will allow us, in Lemma~\ref{lem:pointed-hypertrace}, to produce a hypertrace of an amenable subalgebra of $M\rtimes\Gamma_x$ whose restriction to a copy of $C(X)$ is exactly the Dirac mass at $x$, which is the only property of hypertraces used in the proof of the simplicity criterion.
\medskip

\begin{proposition}\label{lem:N-central-extension}
Let $(N,\tau_N)$ be a finite von Neumann algebra with separable predual. The following are equivalent.
\begin{enumerate}
\item $N$ is amenable.
\item For every faithful representation $N\subset\mathbb{B}(H)$, every unital operator system $\mathcal S\subset \mathbb{B}(H)$ such that $N\mathcal S\subset \mathcal S$ and $\mathcal S N\subset \mathcal S$ and every positive unital linear map $\psi:\mathcal S\to\C$ such that $\psi\vert_N=\tau_N$ and $\psi(as)=\psi(sa)$ for all $a\in N$, $s\in\mathcal S$, there exists a state $\Phi\in\mathbb{B}(H)^*$ such that 
$$\Phi\vert_\mathcal{S}=\psi\quad\text{and}\quad\Phi(aT)=\Phi(Ta)\text{ for all }a\in N,\ T\in\mathbb{B}(H).$$
\end{enumerate}
\end{proposition}
\begin{proof}$(2)\Rightarrow(1)$. Consider the faithful representation $N\subset\mathbb{B}({\rm L}^2(N))$, the operator system $\mathcal{S}=N$ and the state $\psi=\tau_N$ to deduce amenability of $N$.
\medskip
\noindent$(1)\Rightarrow(2)$. Define the weak*-compact convex set $K=\{\Phi\in S(\mathbb{B}(H)):\Phi|_{\mathcal S}=\psi\}$. By the Hahn--Banach theorem, \(K\) is non-empty.
For \(u\in\mathcal U(N)\), define
$$(u\cdot\Phi)(T)=\Phi(u^*Tu)\quad T\in\mathbb{B}(H).$$
Since \(\mathcal S\) is an \(N\)-bimodule and \(\psi\) is \(N\)-central on
\(\mathcal S\), \(K\) is invariant under this action. Indeed, for
\(s\in\mathcal S\) and $\Phi\in K$, $(u\cdot\Phi)(s)=\Phi(u^*su)=\psi(u^*su)=\psi(suu^*)=\psi(s)$. Since \(N\) is a finite amenable von Neumann algebra with a separable predual, it is hyperfinite.
Choose an increasing sequence of finite-dimensional von Neumann subalgebras
$$Q_1\subset Q_2\subset\cdots\subset N$$
whose union is \(\|\cdot\|_2\)-dense in \(N\). Fix \(\Phi_0\in K\). For \(n\in\N\) and $T\in\mathbb{B}(H)$, define $\Phi_n(T):=\int_{\mathcal U(Q_n)}\Phi_0(u^*Tu)\,d\mu_n(u)$,
where \(\mu_n\) denotes the normalized Haar measure on the compact group $\mathcal{U}(Q_n)$. Since the integrand is norm
continuous in \(u\), \(\Phi_n\) is a well-defined state. For \(s\in\mathcal S\), since \(K\) is invariant under \(\mathcal U(Q_n)\), we have, for all $u\in\mathcal U(Q_n)$, $\Phi_0(u^*su)=\psi(s)$. Hence $\Phi_n(s) =\int_{\mathcal U(Q_n)}\Phi_0(u^*su)\,d\mu_n(u)=\psi(s)$ so \(\Phi_n\in K\). Moreover, if \(v\in\mathcal U(Q_n)\), then by the left
invariance of Haar measure,
\begin{eqnarray*}
        \Phi_n(v^*Tv)
        &=&\int_{\mathcal U(Q_n)}\Phi_0(u^*v^*Tvu)\,d\mu_n(u)
         = \int_{\mathcal U(Q_n)}\Phi_0(w^*Tw)\,d\mu_n(w)=\Phi_n(T).
\end{eqnarray*}
Thus \(\Phi_n\) is invariant under conjugation by \(\mathcal U(Q_n)\). Since $Q_n$ is the linear span of \(\mathcal U(Q_n)\), $\Phi_n$ is \(Q_n\)-central. Since $K$ is weak*-compact, the sequence \((\Phi_n)_n\) has a weak-* cluster point $\Phi\in K$. Since
the \(Q_n\)'s are increasing, \(\Phi\) is \(Q_k\)-central for every fixed
\(k\). Hence $\Phi(aT)=\Phi(Ta)$ for all  $a\in \bigcup_k Q_k,\ T\in\mathbb{B}(H)$. Let now \(a\in N\), and choose \(a_j\in\bigcup_k Q_k\) such that $\|a_j-a\|_2\longrightarrow 0$. For \(T\in\mathbb{B}(H)\), we have
\[
\begin{aligned}
        \Phi(aT)-\Phi(Ta)
        &=
        \Phi((a-a_j)T)
        +
        \Phi(a_jT)-\Phi(Ta_j)
        +
        \Phi(T(a_j-a)).
\end{aligned}
\]
The middle term is zero. By Cauchy--Schwarz for the state \(\Phi\),
\[
        |\Phi((a-a_j)T)|
        \leq
        \Phi((a-a_j)(a-a_j)^*)^{1/2}\,\Phi(T^*T)^{1/2}.
\]
Since \(\Phi|_N=\psi|_N=\tau_N\) we have $|\Phi((a-a_j)T)|\leq  \|a-a_j\|_2\,\|T\|$. Similarly,
\[
        |\Phi(T(a_j-a))|
        \leq
        \|T\|\,\|a_j-a\|_2.
\]
Letting \(j\to\infty\), we obtain $\Phi(aT)=\Phi(Ta)$ so \(\Phi\) is \(N\)-central.\end{proof}
%%%%%%%%%%%%%%%%%%%%%%%%%%%%%%%%%%%%%%%%%%%%%%%%%%%%%%%%%%%%%%%%%%%%%%%%%%%%%%%%%%%%%%%%%%%%%%%%%%%%%%%%
\subsection{The space of von Neumann subalgebras}
%%%%%%%%%%%%%%%%%%%%%%%%%%%%%%%%%%%%%%%%%%%%%%%%%%%%%%%%%%%%%%%%%%%%%%%%%%%%%%%%%%%%%%%%%%%%%%%%%%%%%%%%
We now describe the topology in which all of our minimality and continuity statements take place.
Let $(M,\tau)$ be a finite von Neumann algebra and ${\rm Sub}(M)$ the set of its von Neumann subalgebras (with the same unit). For $N\in\Sub(M)$, let $E_N:M\rightarrow N$ be the unique $\tau$-preserving conditional expectation. The Effros-Mar\'echal topology on $\Sub(M)$ is the smallest topology for which the maps $\Sub(M)\rightarrow{\rm L}^2(M)$, $N\mapsto E_N(x)\xi_\tau$ are continuous for the norm topology on ${\rm L}^2(M)$, $\forall x\in M$ (where $\xi_\tau\in{\rm L}^2(M)$ is the cyclic vector). If $M$ has a separable predual, then $\Sub(M)$ is a Polish space. The Effros-Mar\'echal topology is also the smallest topology on $\Sub(M)$ for which the maps $N\mapsto\Vert E_N(x)\Vert_2$ are continuous, for all $x\in M$. For the remainder of this section, we assume that $M$ has a separable predual. We sometimes abbreviate the Effros-Mar\'echal topology as EM-topology.
%%%%%%%%%%%%%%%%%%%%%%%%%%%%%%%%%%%%%%%%%%%%%%%%%%%%%%%%%%%%%%%%%%%%%%%%%%%%%%%%%%%%%%%%%%%%%%%%%%%%%%%%
%%%%%%%%%%%%%%%%%%%%%%%%%%%%%%%%%%%%%%%%%%%%%%%%%%%%%%%%%%%%%%%%%%%%%%%%%%%%%%%%%%%%%%%%%%%%%%%%%%%%%%%%
\subsection{The Jones and Pimsner--Popa index}
%%%%%%%%%%%%%%%%%%%%%%%%%%%%%%%%%%%%%%%%%%%%%%%%%%%%%%%%%%%%%%%%%%%%%%%%%%%%%%%%%%%%%%%%%%%%%%%%%%%%%%%%
Two of our families of exotic URAs are distinguished by an index invariant, which
we now recall. Let $(M,\tau)$ be a finite von Neumann algebra and
$N\in\Sub(M)$, with $E_N:M\to N$ the trace-preserving conditional expectation.
Following Pimsner--Popa \cite{pimsnerpopa}, set
$$
        \lambda(N):=\sup\{c>0:E_N(x)\geq cx\ \forall x\in M^+\}
        \in[0,1],
$$
where $M^+$ denotes the positive cone of $M$. The \textbf{Pimsner--Popa index}
of $N$ in $M$ is $[M:N]:=\lambda(N)^{-1}\in[1,\infty]$. When $M$ is a $\mathrm{II}_1$ factor and $N\subseteq M$ a subfactor, this
coincides with the Jones index \cite{jones1983index} defined via the basic
construction $\langle M,e_N\rangle$, where $e_N:\LL^2(M,\tau)\to \LL^2(N,\tau)$ is the Jones projection introduced above. We will only need the following two elementary facts.
\begin{lemma}\label{lem:index-conjugation-invariant}
Let $(M,\tau)$ be a finite von Neumann algebra and $\theta\in\Aut(M,\tau)$ a
trace-preserving automorphism. Then $[M:\theta(N)]=[M:N]$ for every
$N\in\Sub(M)$. In particular, for a trace-preserving action
$\Gamma\curvearrowright(M,\tau)$, the index $[M:gNg^{-1}]$ is constant along the
orbit $\{gNg^{-1}:g\in\Gamma\}$.
\end{lemma}
\begin{proof}
Since $\theta$ is a trace-preserving automorphism, $E_{\theta(N)}=\theta\circ
E_N\circ\theta^{-1}$, and $\theta$ preserves positivity and the order. Hence
$E_N(x)\geq cx$ for all $x\in M^+$ if and only if $E_{\theta(N)}(y)\geq cy$ for
all $y\in M^+$, so $\lambda(\theta(N))=\lambda(N)$.
\end{proof}

\noindent The second fact is a computation of the Pimsner--Popa index of a
two-sided block decomposition of a finite factor. Both families of exotic URAs
constructed in Subsection~\ref{subsec:exoticURAs} consist of subalgebras of this
shape, and the numerical value obtained here is what will separate them in
Remark~\ref{rmk:fixed-point-corner-index-comparison}.
\medskip

\noindent The finite-dimensional case of the next Lemma appears in
\cite[Example~6.5(2), p.~94]{pimsnerpopa}; the same formula holds for
arbitrary finite factors.
\medskip
\noindent Whenever $\phi$ is an automorphism of a countable discrete group $\Gamma$, we still denote by $\phi\in\Aut(\LL(\Gamma))$ its canonical trace preserving extension. Also, for $M$ a von Neumann algebra and $\phi\in\Aut(M)$ we write $M^\phi:=\{x\in M:\phi(x)=x\}$.
\begin{lemma}\label{lem:index}
Let $(M,\tau)$ be a finite factor, $m\geq2$ and $p_0,\ldots,p_{m-1}\in M$ be non-zero pairwise orthogonal projections such
that $\sum_{k=0}^{m-1}p_k=1$ and define
$$N:=\bigoplus_{k=0}^{m-1}p_kMp_k.$$
One has $[M:N]=m$. In particular, if $\Gamma$ is an i.c.c.\ group and $g\in\Gamma$ is such that
$\phi=\operatorname{ad}(g)\in\Aut(\Gamma)$ has order $m\geq2$ then, $[\LL(\Gamma):\LL(\Gamma)^\phi]=m$.
\end{lemma}
\begin{proof}
Let $\zeta:=e^{2\pi i/m}$ and put $u:=\sum_{k=0}^{m-1}\zeta^kp_k$. Then $u^m=1$, $N=M^{\Ad(u)}$, and the trace-preserving conditional
expectation onto $N$ is
$$
E_N(x)
=
\sum_{k=0}^{m-1}p_kxp_k
=
\frac1m\sum_{j=0}^{m-1}u^jxu^{-j},
\quad x\in M.
$$
Hence $E_N(x)\geq\frac1m x$ for all $x\in M^+$ and therefore $[M:N]\leq m$.
\medskip
\noindent If $M$ is of type ${\rm I}_d$, choose minimal projections
$e_k\leq p_k$, $0\leq k<m$. If $M$ is of type ${\rm II}_1$, choose
$0<t\leq\min_k\tau(p_k)$ and projections $e_k\leq p_k$ satisfying
$\tau(e_k)=t$, $0\leq k<m$. In either case, the projections $e_k$ are
non-zero and mutually equivalent. Choose matrix units
$(e_{k\ell})_{0\leq k,\ell<m}$ such that $e_{kk}=e_k$, and put $e:=\sum_{k=0}^{m-1}e_k$ and $r:=\frac1m\sum_{k,\ell=0}^{m-1}e_{k\ell}$.
Then $r$ is a projection, $r\leq e$, and $E_N(r)=\frac1m e$. If $E_N(r)\geq cr$ for some $c>0$, compression by $r$ gives
$\frac1m r
=rE_N(r)r
\geq cr$. Thus $c\leq1/m$, and consequently $[M:N]\geq m$.
\medskip
\noindent For the last assertion, the i.c.c.\ hypothesis gives
$Z(\Gamma)=\{e\}$. Hence $\phi$ having order $m$ implies that
$g$ has order $m$. For $0\leq k<m$, let $p_k$ be the spectral projection
of $\lambda_g$ given by
$$
p_k
:=
\frac1m\sum_{j=0}^{m-1}e^{-2\pi ikj/m}\lambda_{g^j}.
$$
These are pairwise orthogonal projections with sum $1$ and
$\tau(p_k)=\frac1m$, so they are non-zero. Moreover,
$$\LL(\Gamma)^\phi
=
\{\lambda_g\}'\cap\LL(\Gamma)
=
\bigoplus_{k=0}^{m-1}p_k\LL(\Gamma)p_k.$$
The first part therefore gives $[\LL(\Gamma):\LL(\Gamma)^\phi]=m$.
\end{proof}

%%%%%%%%%%%%%%%%%%%%%%%%%%%%%%%%%%%%%%%%%%%%%%%%%%%%%%%%%%%%%%%%%%%%%%%%%%%%%
\subsection{\texorpdfstring{$\Gamma$}{}-spaces and \texorpdfstring{$\Gamma$}{}-boundaries}
\label{subsec:GammaBoundary}
%%%%%%%%%%%%%%%%%%%%%%%%%%%%%%%%%%%%%%%%%%%%%%%%%%%%%%%%%%%%%%%%%%%%%%%%%%%%%%%%%%%%%%%%%%%%%%%%%%%%%%%%
We will repeatedly use the notions of minimality and strong proximality. A $\Gamma$-space is called \textbf{minimal} if it has no non-trivial $\Gamma$-invariant closed subset, or equivalently if every $\Gamma$-orbit is dense. A \textbf{$\Gamma$-boundary} (Furstenberg) is a compact Hausdorff $\Gamma$-space $X$ that is:
\begin{itemize}
\item \textbf{minimal}: no non-trivial closed $\Gamma$-invariant subset;
\item \textbf{strongly proximal}: for every $\mu \in \mathrm{Prob}(X)$, the orbit closure $\overline{\Gamma\cdot\mu}$ in $\mathrm{Prob}(X)$ meets the set of Dirac measures $\{\delta_x : x \in X\}$.
\end{itemize}
Up to $\Gamma$-equivariant continuous surjection, there is a unique maximal $\Gamma$-boundary, the \textbf{Furstenberg boundary} $\partial_F \Gamma$.
We will also need a relative version of this notion in the presence of a $\Gamma$-action on a compact space $X$. Following \cite{kawabe2017uniformly} (and \cite{Naghavi}), the \textbf{generalized Furstenberg boundary} $\widetilde{X}$ of $X$ is a compact $\Gamma$-space equipped with a $\Gamma$-equivariant continuous surjection $q:\widetilde{X}\to X$ which is universal among $\Gamma$-boundaries fibring over $X$ that satisfies the following. $\widetilde{X}$ is minimal and for every $\nu\in\text{Prob}(\widetilde{X})$, if the push forward of $\nu$ on X via  $q$ is contractible, then $\nu$ is contractible. This object plays the role for the action $\Gamma\curvearrowright X$ that the Furstenberg boundary plays for $\Gamma$ itself, and underlies our criterion for simplicity of $C(X)\rtimes_r\Gamma$.
%%%%%%%%%%%%%%%%%%%%%%%%%%%%%%%%%%%%%%%%%%%%%%%%%%%%%%%%%%%%%%%%%%%%%%%%%%%%%%%%%%%%%%%%%%%%%%%%%%%%%%%%
\subsection{Automorphisms and fixed-point subalgebras}
%%%%%%%%%%%%%%%%%%%%%%%%%%%%%%%%%%%%%%%%%%%%%%%%%%%%%%%%%%%%%%%%%%%%%%%%%%%%%%%%%%%%%%%%%%%%%%%%%%%%%%%%
Several of the URA constructions later in the paper rest on fixed-point
subalgebras of finite-order automorphisms.
\begin{definition}\label{def:FIR}
We say that a group $\Gamma$ has the
\textbf{finite-index rigidity property ${\rm(FIR)}$} if every automorphism of
$\Gamma$ which fixes pointwise a finite-index subgroup is the identity.
\end{definition}
\noindent We will use the following notation for the centralizer of a subgroup $\Lambda\leq\Gamma$
$$C_\Gamma(\Lambda):=\{g\in\Gamma:gh=hg\text{ for all }h\in\Lambda\}.$$
\begin{example}
The following holds.
\begin{itemize}
\item A non-trivial group $\Gamma$ is i.c.c.\ if and only if it has
${\rm(FIR)}$ and $Z(\Gamma)=\{e\}$.
\end{itemize}
\noindent Indeed, suppose first that $\Gamma$ is i.c.c. If
$\alpha\in\Aut(\Gamma)$ fixes pointwise a finite-index subgroup, let
$N\triangleleft\Gamma$ be its normal core. For $g\in\Gamma$ and $x\in N$,
one has
$$
\alpha(g)x\alpha(g)^{-1}
=\alpha(gxg^{-1})
=gxg^{-1}.
$$
Hence $g^{-1}\alpha(g)\in C_\Gamma(N)$. Since $N$ has finite index, every
element of $C_\Gamma(N)$ has a finite conjugacy class. Thus
$C_\Gamma(N)=\{e\}$ and $\alpha=\id$. Moreover, $Z(\Gamma)=\{e\}$. Conversely, suppose that $\Gamma$ has ${\rm(FIR)}$ and
$Z(\Gamma)=\{e\}$. If $g\in\Gamma$ has a finite conjugacy class, then
$C_\Gamma(g)$ has finite index in $\Gamma$. Since $\Ad(g)$ fixes
$C_\Gamma(g)$ pointwise, ${\rm(FIR)}$ gives $\Ad(g)=\id$. Hence
$g\in Z(\Gamma)=\{e\}$, so $\Gamma$ is i.c.c.
\begin{itemize}
\item $\mathbb Z$ has ${\rm(FIR)}$ but is not i.c.c.
\end{itemize}
\noindent Indeed, $\Aut(\mathbb Z)=\{\id,-\id\}$ and the fixed-point
subgroup of $-\id$ is $\{0\}$, which has infinite index in $\mathbb Z$.
Hence no non-trivial automorphism of $\mathbb Z$ fixes pointwise a
finite-index subgroup.
\end{example}

\noindent The next lemma is the rigidity statement on which the fixed-point
constructions depend. In particular, a finite group of automorphisms of $\Gamma$ is completely
determined by the fixed-point subalgebra it produces in $\LL(\Gamma)$. It is
used in Lemma~\ref{lem:finite-fixed-point-algebra} to compute the stabilizer of
$\LL(\Gamma)^G$, and hence to show that the corresponding orbit is infinite.
\medskip

\begin{lemma}\label{lem:finite-fixed-point-action-rigidity}
Let $\Gamma$ be a group with ${\rm(FIR)}$, and let
$G,H\leq\Aut(\Gamma)$ be finite subgroups. If $\LL(\Gamma)^G=\LL(\Gamma)^H$ then $G=H$.
\end{lemma}
\begin{proof}
Let $E_G$ and $E_H$ be the trace-preserving conditional expectations onto
the two fixed-point algebras. For $g\in\Gamma$, $E_G(\lambda_g)=
\frac1{|\operatorname{Orb}_G(g)|}
\sum_{h\in\operatorname{Orb}_G(g)}\lambda_h$
and similarly for $H$. Hence $E_G=E_H$ and the linear independence of the
canonical unitaries give $\operatorname{Orb}_G(g)=\operatorname{Orb}_H(g)$ for all $g\in\Gamma$. Fix $\alpha\in H$. For $\beta\in G$, put
$$K_\beta:=\{g\in\Gamma:\alpha(g)=\beta(g)\}.$$
Each $K_\beta$ is the fixed-point subgroup of
$\beta^{-1}\circ\alpha$ and $\Gamma=\bigcup_{\beta\in G}K_\beta$. By \cite[Lemma~4.1]{neumann1954groups}, some $K_\beta$ has finite index.
Since $\beta^{-1}\circ\alpha$ fixes $K_\beta$ pointwise, ${\rm(FIR)}$ gives
$\alpha=\beta\in G$. Thus $H\subseteq G$. Interchanging $G$ and $H$ gives
$G\subseteq H$.
\end{proof}

\noindent A direct consequence of ${\rm(FIR)}$ is that the fixed-point subgroup
of a non-trivial automorphism is always of infinite index. This is what
prevents the fixed-point URAs from being finite.
\medskip

\begin{lemma}\label{lem: finite order automorphism has infinite index fixed point subgroups}
Let $\Gamma$ be a group with ${\rm(FIR)}$ and let
$\phi\in\Aut(\Gamma)\setminus\{\id\}$. Then
$$
[\Gamma:\operatorname{Fix}(\phi)]=\infty.
$$
\end{lemma}
\begin{proof}
Otherwise, $\phi$ would fix pointwise the finite-index subgroup
$\operatorname{Fix}(\phi)$, contradicting ${\rm(FIR)}$.
\end{proof}

\noindent Finally, we record an elementary $\LL^2$-rigidity property of finite
sums of canonical unitaries. Since the conditional expectation onto the
fixed-point subalgebra of a finite group of automorphisms is such a sum, this lemma is what converts an approximate
coincidence of two conditional expectations into an exact coincidence of the
corresponding orbits. It is the engine behind the discreteness statement of
Proposition~\ref{prop:finite-action-fixed-point-rigidity}.
\medskip

\begin{lemma}\label{lem: close multisets are equal}
Let $s_1,\ldots,s_n,t_1,\ldots,t_m$ be elements of a group $\Gamma$. For
$g\in\Gamma$, denote by $a_g$ and $b_g$ the multiplicities of $g$ in
$(s_i)_{i=1}^n$ and $(t_j)_{j=1}^m$, respectively. Then
$$
\left\|
\sum_{i=1}^n\lambda_{s_i}
-
\sum_{j=1}^m\lambda_{t_j}
\right\|_2^2
=
\sum_{g\in\Gamma}(a_g-b_g)^2.
$$
Consequently, if the norm on the left is less than $1$, then $n=m$ and the
two multisets are equal. If $n=m$, the same conclusion holds when the norm
is less than $\sqrt2$.
\end{lemma}
\begin{proof}
The identity follows from the orthonormality of
$(\lambda_g)_{g\in\Gamma}$ in $\LL^2(\LL(\Gamma))$. If its right-hand side is less
than $1$, then every integer $a_g-b_g$ is zero. If $n=m$ and the multisets
are distinct, then $\sum_{g\in\Gamma}(a_g-b_g)=0$, so at least two of the integers $a_g-b_g$ are non-zero and the right-hand
side is at least $2$.
The bound $1$ is optimal, as witnessed by the lists $(e,e)$ and $(e)$. If
$\Gamma\neq\{e\}$, the bound $\sqrt2$ in the equal-length case is optimal, as
witnessed by the lists $(e)$ and $(g)$ with $g\neq e$.
\end{proof}

%%%%%%%%%%%%%%%%%%%%%%%%%%%%%%%%%%%%%%%%%%%%%%%%%%%%%%%%%%%%%%%%%%%%%%%%%%%%%%%%%%%%%%%%%%%%%%%%%%%%%%%%
\subsection{Minimal Polish spaces}
%%%%%%%%%%%%%%%%%%%%%%%%%%%%%%%%%%%%%%%%%%%%%%%%%%%%%%%%%%%%%%%%%%%%%%%%%%%%%%%%%%%%%%%%%%%%%%%%%%%%%%%%
When we construct URAs from minimal dynamical systems later, the universal source of countable groups acting minimally on a given space comes from the following classical observation. In the second-countable setting, the abstract homeomorphism-group action carries a countable subgroup that acts just as minimally.
\begin{definition}\label{DefMinSpace}
    A non-empty topological space $X$ will be called \textbf{minimal} if the action $\Homeo(X)\curvearrowright X$ is minimal.
\end{definition}
\begin{proposition}\label{PropMinPolish}
    Suppose that $X$ is a non-empty second countable topological space. The following are equivalent.
    \begin{enumerate}
        \item $X$ is minimal.
        \item There exists a minimal action $\Gamma\curvearrowright X$ of a countable discrete group $\Gamma$.
    \end{enumerate}
\end{proposition}
\begin{proof}\((1)\Rightarrow (2)\). Let $(U_n)_{n\geq 1}$ be a countable basis of non-empty open subsets of $X$. Fix \(n\geq 1\). By minimality, for every
\(x\in X\), the orbit \(\Homeo(X)\cdot x\) is dense in \(X\), thus there exists \(h\in \Homeo(X)\) such that $h(x)\in U_n$. Equivalently, $x\in h^{-1}(U_n)$. Therefore $X=\bigcup_{h\in\Homeo(X)} h^{-1}(U_n)$ is an open cover of \(X\). Since $X$ is second countable, we may extract a countable subcover. Thus, for each \(n\), there exists a sequence $(h_{n,m})_{m\geq 1}\subset \Homeo(X)$ such that $X=\bigcup_{m\geq 1} h_{n,m}^{-1}(U_n)$. Define
\[
        \Gamma:=\langle h_{n,m}:n,m\geq 1\rangle
        \leq \Homeo(X).
\]
Since \(\Gamma\) is generated by countably many elements, it is a countable group. We claim that $\Gamma\curvearrowright X$ is minimal. Let \(x\in X\), and let
\(U\subset X\) be a non-empty open set. Choose \(n\) such that \(U_n\subset U\). By construction, $X=\bigcup_{m\geq 1} h_{n,m}^{-1}(U_n)$. Hence there exists \(m\geq 1\) such that $x\in h_{n,m}^{-1}(U_n)$,
or equivalently, $h_{n,m}(x)\in U_n\subset U$. This shows that $\Gamma\cdot x\cap U\neq\emptyset$. Thus \(\Gamma\curvearrowright X\) is minimal.
\medskip
\noindent\((2)\Rightarrow (1)\). Suppose that a countable group
\(\Gamma\) acts minimally on \(X\) by homeomorphisms. Then, for every \(x\in X\), $\Gamma\cdot x\subseteq \Homeo(X)\cdot x$.
Since $\overline{\Gamma\cdot x}=X$ we have $\overline{\Homeo(X)\cdot x}=X$ thus \(\Homeo(X)\curvearrowright X\) is minimal.\end{proof}
%%%%%%%%%%%%%%%%%%%%%%%%%%%%%%%%%%%%%%%%%%%%%%%%%%%%%%%%%%%%%%%%%%%%%%%%%%%%%%%%%%%%%%%%%%%%%%%%%%%%%%%%
\section{Compactness of \texorpdfstring{$\Sub(M)$}{}}
\label{sec:compactness}
%%%%%%%%%%%%%%%%%%%%%%%%%%%%%%%%%%%%%%%%%%%%%%%%%%%%%%%%%%%%%%%%%%%%%%%%%%%%%%%%%%%%%%%%%%%%%%%%%%%%%%%%
A central feature of the group-theoretic theory of URS is that the ambient space $\Sub(\Gamma)$ is compact. This compactness, inherited from the embedding into $\{0,1\}^\Gamma$, is what makes Zorn's lemma immediately produce minimal $\Gamma$-invariant subsets, and underlies the entire dynamical study of URS in $\mathrm{PD}_1(\Gamma)$ carried out in \cite{AJ}.
For von Neumann subalgebras, the situation is markedly different. The aim of this section is to characterize precisely when $\Sub(M)$ is compact, and to show that whenever compactness fails, it does so in such a way that $\Sub(M)$ captures any Polish space as a closed subspace. The reason compactness can fail is that $M$ may have a diffuse part. With a diffuse direct summand at our disposal, we can build a continuum of pairwise far-apart two-dimensional subalgebras, and these embed any Polish space.
The construction proceeds in three steps. We first encode the real line as a continuous family of projections of trace $\tfrac12$ inside $\mathrm{L}^\infty([0,1])$, using an explicit measurable dynamics on $\{0,1\}^\Z$ to enforce closedness (Lemma~\ref{LemClosedLineProjections}). We then upgrade this single embedding $\R \hookrightarrow \Proj$ to a closed embedding $\R^\N \hookrightarrow \mathcal P(e_0,r)$ inside a sufficiently large set of projections, by exploiting a diffuse direct summand (Lemma~\ref{LemClosedBaireMarkedProjections}). Finally, we pass from marked projections to subalgebras via $p \mapsto \C p \oplus \C(1-p)$ (Lemma~\ref{LemMarkedProjectionsToSubalgebras}); a Kechris embedding theorem of arbitrary Polish spaces into $\R^\N$ then completes the proof of Theorem~\ref{ThmEmbedding}.
\begin{lemma}
\label{LemClosedLineProjections}
There exists a closed topological embedding $\mathbb R\hookrightarrow \Proj({\rm L}^\infty([0,1]))$, $t\mapsto q_t$
for the \({\rm L}^2\)-topology. Moreover, the embedding may be chosen so that $\tau(q_t)=\frac12$ for all $t\in\R$.
\end{lemma}
\begin{proof}
Consider the compact metric space $Y:=\{0,1\}^{\mathbb Z}\times[0,1]$ with Borel probability measure
$\nu
        :=
       \mu\otimes \lambda$ where $\mu:=\bigotimes_{n\in\mathbb Z}
        \left(\frac12\delta_0+\frac12\delta_1\right)$ and $\lambda$ is the normalized Lebesgue measure on $[0,1]$. Since \((Y,\nu)\) is a standard non-atomic probability space, there is a
trace-preserving isomorphism ${\rm L}^\infty(Y)\cong {\rm L}^\infty([0,1])$ and it suffices to construct a closed topological embedding $\mathbb R\hookrightarrow \Proj({\rm L}^\infty(Y))$.
\medskip
\noindent Write a point of \(Y\) as \((\omega,u)\), where $\omega=(\omega_n)_{n\in\mathbb Z}\in\{0,1\}^{\mathbb Z}$ and $u\in[0,1]$. For each \(n\in\mathbb Z\), put $R_n:=\{\omega:\omega_n=1\}$. Then $\mu(R_n)=\frac12$. For \(t\in\mathbb R\), write uniquely $t=n+s$, $n\in\mathbb Z$ and  $s\in[0,1)$. Define
$A_t:=R_n\times(s,1]\sqcup R_{n+1}\times[0,s]$ and  $q_t:=1_{A_t}\in {\rm Proj}({\rm L}^\infty(Y))$. Note that $\nu(A_t)=(1-s)\nu(R_n)+s\nu(R_{n+1})=\frac12$. Hence $\tau(q_t)=\frac12$ $\forall t\in\mathbb R$ and it suffices to show that    $t\mapsto q_t$ is a closed topological embedding for the \({\rm L}^2\)-topology.
\medskip
\noindent Let us show that $t\mapsto q_t$ is continuous. Note that, for all Borel sets $A,B\subseteq Y$,
$$\|1_A-1_B\|_2^2=\nu(A\triangle B).$$
Fix $n\in\Z$ and let us show continuity on $[n,n+1)$. If \(t=n+s\) and \(t'=n+s'\) lie in the same interval
\([n,n+1)\), then \(A_t\) and \(A_{t'}\) differ only on the strip $\min(s,s')<u\leq \max(s,s')$. Therefore
$\Vert q_t-q_{t'}\Vert_2^2
        =
        \nu(A_t\triangle A_{t'})
        \leq |s-s'|$. It remains to check continuity at the integers. Let \(m\in\mathbb Z\). By
definition, $A_m=R_m\times(0,1]\cup R_{m+1}\times\{0\}$ hence, in ${\rm L}^\infty(Y)$ one has $q_m=1_{{R_m}\times[0,1]}$. For $t=m+s$ with $s\in [0,1)$. Since
$$A_t\setminus\left(R_m\times[0,1]\right)\subset R_{m+1}\times[0,s]\text{ and }\left(R_m\times[0,1]\right)\setminus A_t\subset R_m\times[0,s]$$
we have $ \nu(A_t\triangle R_m)\leq s$. Hence, if \(t\to m\) from the right, say \(t=m+s\) with \(s\to0^+\), then $\Vert q_t-q_m\Vert_2^2=\nu(A_t\triangle R_m)\leq s\to0$. If \(t\to m\) from the left, say \(t=(m-1)+s\) with \(s\to1^-\), then
$$A_t\setminus(R_m\times[0,1])\subset R_{m-1}\times(s,1]\text{ and }(R_m\times[0,1])\setminus A_t\subset R_m\times(s,1].$$
Therefore $\nu(A_t\triangle(R_m\times[0,1]))\leq 1-s\to0$. This proves continuity at \(m\).
\medskip
\noindent We now show that the map is injective. Let $t,t'\in\R$ with $t\neq t'$ and write $t=n+s$ and $t'=m+r,$
with \(n,m\in\mathbb Z\) and \(s,r\in[0,1)\). Define $c_t,c_{t'}:[0,1]\to\Z$ by
\[
        c_t(u)=
        \begin{cases}
        n+1, & u\leq s,\\
        n, & u>s,
        \end{cases}
        \qquad
        c_{t'}(u)=
        \begin{cases}
        m+1, & u\leq r,\\
        m, & u>r.
        \end{cases}
\]
Then $(\omega,u)\in A_t\Leftrightarrow\omega_{c_t(u)}=1$. Since \(t\neq t'\), the step functions \(c_t\) and \(c_{t'}\) differ on the Borel set $U:=\{u\in[0,1]:c_t(u)\neq c_{t'}(u)\}$
of positive Lebesgue measure. More precisely, writing
$t=n+s$ and $t'=m+r$, with $n,m\in\mathbb Z$ and $s,r\in[0,1)$, one has
$$
\lambda(U)=
\begin{cases}
|s-r|, & n=m,\\
1-\max\{s-r,0\}, & m=n+1,\\
1-\max\{r-s,0\}, & n=m+1,\\
1, & |n-m|\geq 2.
\end{cases}
$$
In every case, $\lambda(U)>0$. Note that, for every \(u\in U\), $ \mu(\{\omega:\omega_{c_t(u)}\neq\omega_{c_{t'}(u)}\})
        =
        \frac12$. Therefore, by Fubini,
\[
        \nu(A_t\triangle A_{t'})
        =
        \int_0^1
        \mu(\{\omega:\omega_{c_t(u)}\neq\omega_{c_{t'}(u)}\})\,d\lambda(u)
       =
        \frac12\lambda(U)>0.
\]
Thus \(q_t\neq q_{t'}\).
\medskip
\noindent Finally, we show that the image is closed and the inverse is continuous on the image. We will use the following Claim.
\medskip
\noindent\textbf{Claim.} \textit{If \((t_m)_m\) be an unbounded sequence in \(\mathbb R\) then the sequence $(q_{t_m})_m$is not Cauchy in \({\rm L}^2(Y)\).}
\medskip
\noindent\textit{Proof of the Claim.} After passing to a subsequence, we may write $t_m=n_m+s_m$, $n_m\in\mathbb Z$, $s_m\in[0,1)$, with the pairs $\{n_m,n_m+1\}$ pairwise disjoint. Therefore, for every \(u\in[0,1]\) and $m\neq l\in\Z$, $\mu(\{\omega:1_{A_{t_m}}(\omega,u)\neq 1_{A_{t_\ell}}(\omega,u)\})=\frac12$. By Fubini,
\[
        \nu(A_{t_m}\triangle A_{t_\ell})
        =
        \int_0^1
        \mu(\{\omega:1_{A_{t_m}}(\omega,u)\neq 1_{A_{t_\ell}}(\omega,u)\})
        \,d\lambda(u)
        =
        \frac12.
\]
Hence
\[
        \|q_{t_m}-q_{t_\ell}\|_2^2
        =
        \nu(A_{t_m}\triangle A_{t_\ell})
        =
        \frac12.
\]
Thus the subsequence \((q_{t_m})_m\) is not Cauchy.$\hfill\qed$
\medskip
\noindent Let us show that the image is closed in ${\rm L}^2(Y)$. Let $(t_m)_m$ be a sequence in $\R$ such that $(q_{t_m})$ converges in ${\rm L}^2(Y)$. By the Claim, $(t_m)_m$ is bounded. Passing to a subsequence, we may assume $t_m\to t\in\mathbb R$. By continuity, $q_{t_m}\to q_t$. This shows that the image is closed. Finally, we show that the inverse is continuous on the image. Assume that $q_{t_m}\to q_t$ then, by the Claim, the sequence $(t_m)_m$ is bounded, so it suffices to show that $t$ is the only accumulation point of the sequence $(t_m)_m$. Consider a subsequence such that $t_{m_k}\to t'$. By continuity, $q_{t_{m_k}}\to q_{t'}$ and by uniqueness of limits $q_{t'}=q_t$. By injectivity, we have $t'=t$. It completes the proof.\end{proof}
With the line of projections in hand, we now amplify it to a countable product. The key is that a sufficiently large diffuse algebra contains a sequence of mutually orthogonal projections, on each of which we can install an independent copy of the previous lemma.
\begin{lemma}
\label{LemClosedBaireMarkedProjections}
Let \((M,\tau)\) be a diffuse finite von Neumann algebra with faithful normal trace.
Let \(e_0\in M\) be a projection such that $\tau(e_0)<\frac12$. Then, for any $r\in(\tau(e_0),\frac12)$, there exists a closed topological embedding 
$\R^\N\hookrightarrow \mathcal P(e_0,r)$ for the \({\rm L}^2\)-topology, where
        $$\mathcal P(e_0,r):=
        \{p\in\Proj(M):e_0\leq p,\ \tau(p)=r\}.$$
\end{lemma}
\begin{proof}
Write $s:=\tau(e_0)$. Since \((1-e_0)M(1-e_0)\) is diffuse, we may choose pairwise orthogonal non-zero
projections $ f_1,f_2,\dots\leq 1-e_0$ such that $\sum_{j\geq1}\tau(f_j)=2(r-s)$. For instance, one may choose $\tau(f_j)=2^{-j}\,2(r-s)$,  $j\geq1$, inside a diffuse abelian von Neumann subalgebra of $(1-e_0)M(1-e_0)$. For every \(j\geq1\), the finite von Neumann algebra \(f_jMf_j\), equipped with the
normalised trace $\tau_j(x):=\frac{\tau(x)}{\tau(f_j)}$, $x\in f_jMf_j$, is diffuse. Hence,
there exists a trace-preserving embedding ${\rm L}^\infty([0,1])\hookrightarrow f_jMf_j$ where \(f_jMf_j\) is viewed with the normalised trace \(\tau_j\). By Lemma~\ref{LemClosedLineProjections}, transported through this embedding, we may choose a closed topological embedding $\mathbb R\hookrightarrow \Proj(f_jMf_j)$, $t\mapsto q_t^{(j)}$, such that $\tau_j(q_t^{(j)})=\frac12$ $\forall t\in\mathbb R$. In terms of the original trace \(\tau\), this means $\tau(q_t^{(j)})=\frac12\tau(f_j)$. For $y=(y_1,y_2,\dots)\in\mathbb R^{\mathbb N}$ define $p_y:=e_0+\sum_{j\geq1}q_{y_j}^{(j)}$. The sum is the strong operator, hence the \({\rm L}^2\)-limit, of the increasing sequence of
finite partial sums, since the projections are supported on pairwise orthogonal
projections \(f_j\). Thus \(p_y\) is a projection. Moreover, $e_0\leq p_y$ and
\[
        \tau(p_y)
        =
        \tau(e_0)+\sum_{j\geq1}\tau(q_{y_j}^{(j)})
        =
        s+\frac12\sum_{j\geq1}\tau(f_j)
        =
        s+\frac12\cdot 2(r-s)
        =
        r.
\]
Therefore $p_y\in\mathcal P(e_0,r)$. Let us show that $\Phi:\mathbb R^{\mathbb N}\rightarrow \mathcal P(e_0,r)$, $\Phi(y):=p_y$ is a closed topological embedding.
\medskip
\noindent Let us show that \(\Phi\) is continuous. Let $(y^{(m)})_m$ be a sequence in $\R^\N$ converging to $y\in\R^\N$. Then, for every fixed \(j\), $q_{y^{(m)}_j}^{(j)}\to q_{y_j}^{(j)}$ in ${\rm L}^2(f_jMf_j)$. Since $\|q_{y^{(m)}_j}^{(j)}-q_{y_j}^{(j)}\|_2^2
\leq \tau(f_j)$ and $\sum_{j\geq1}\tau(f_j)<\infty$, the dominated convergence theorem gives
\[
        \|p_{y^{(m)}}-p_y\|_2^2
        =
        \sum_{j\geq1}
        \|q_{y^{(m)}_j}^{(j)}-q_{y_j}^{(j)}\|_2^2
        \longrightarrow 0.
\]
Thus \(\Phi\) is continuous.
\medskip
\noindent Let us show that \(\Phi\) is injective. If \(y\neq z\), choose \(j\) such that $y_j\neq z_j$. Since $t\mapsto q_t^{(j)}$ is injective, we have $q_{y_j}^{(j)}\neq q_{z_j}^{(j)}$. Compressing by \(f_j\), we get
\[
        f_jp_yf_j=q_{y_j}^{(j)}
        \neq
        q_{z_j}^{(j)}=f_jp_zf_j.
\]
Hence $p_y\neq p_z$.
\medskip
\noindent Let us show that the image is closed. Suppose that $p_{y^{(m)}}\to p$ in ${\rm L}^2(M)$, where \(p\in\Proj(M)\). For each \(j\geq1\), compression by \(f_j\) gives $q_{y^{(m)}_j}^{(j)}=f_jp_{y^{(m)}}f_j\to f_jpf_j$ in ${\rm L}^2(f_jMf_j)$. Since the image of $t\mapsto q_t^{(j)}$ is closed, there exists \(y_j\in\mathbb R\) such that $f_jpf_j=q_{y_j}^{(j)}$. Also, $e_0p_{y^{(m)}}=e_0$ for all $m$. Passing to the \({\rm L}^2\)-limit gives $e_0p=e_0$. Similarly, $p_{y^{(m)}}e_0=e_0$ for all $m$ so $pe_0=e_0$. Thus \(e_0\leq p\). Define $f_\infty:=1-e_0-\sum_{j\geq1}f_j$ and note that $f_\infty p_{y^{(m)}}=0=p_{y^{(m)}}f_\infty$ for all $m$. It follows that $f_\infty p=0=pf_\infty$. Moreover, for \(j\neq k\), $f_jp_{y^{(m)}}f_k=0$ for all $m$ so $f_jpf_k=0$. It follows that \(p\) is block diagonal with respect to the orthogonal decomposition
\[
        1=e_0+\sum_{j\geq1}f_j+f_\infty,
\]
and its diagonal blocks are $e_0$, $q_{y_j}^{(j)}\ (j\geq1)$ and  $0$. Therefore $p=e_0+\sum_{j\geq1}q_{y_j}^{(j)}=p_y$ where $y=(y_1,y_2,\dots)\in\mathbb R^{\mathbb N}$. Thus, the image of \(\Phi\) is closed.
\medskip
\noindent Finally, we show that the inverse is continuous on the image. Suppose $p_{y^{(m)}}\to p_y$ in ${\rm L}^2(M)$. Compressing by \(f_j\), we obtain $q_{y^{(m)}_j}^{(j)}=f_jp_{y^{(m)}}f_j\rightarrow f_jp_yf_j=q_{y_j}^{(j)}$ in ${\rm L}^2(f_jMf_j)$. Since $t\mapsto q_t^{(j)}$ is a topological embedding, this implies $y^{(m)}_j\to y_j$ for all $j\geq1$ hence $y^{(m)}\to y$ in \(\mathbb R^{\mathbb N}\).\end{proof}
Marked projections produce closed embeddings into a space of subalgebras almost for free, by sending a projection to the two-dimensional algebra it generates.
\begin{lemma}
\label{LemMarkedProjectionsToSubalgebras}
Let \((M,\tau)\) be a finite von Neumann algebra. Let \(0<s<r<1/2\), and let
\(e_0\in M\) be a projection with $\tau(e_0)=s$. The map
$$
        \Psi:\mathcal P(e_0,r)\rightarrow \Sub(M),
        \quad
        p\mapsto W^*(p):=\mathbb C p\oplus \mathbb C(1-p),
$$
is a closed topological embedding, where \(\mathcal P(e_0,r)\) is equipped with the
\({\rm L}^2\)-topology and \(\Sub(M)\) with the Effros-Mar\'echal topology.
\end{lemma}
\begin{proof}
For \(p\in\mathcal P(e_0,r)\), the algebra $W^*(p)$ is two-dimensional, and the trace-preserving conditional expectation onto \(W^*(p)\)
is given by
\[
        E_{W^*(p)}(x)
        =
        \frac{\tau(xp)}{r}p
        +
        \frac{\tau(x(1-p))}{1-r}(1-p)\qquad x\in M.
\]
\noindent We first prove continuity. Let \(p_i\to p\) in \({\rm L}^2(M)\), with
\(p_i,p\in\mathcal P(e_0,r)\). For every \(x\in M\), we have $\tau(xp_i)\to \tau(xp)$ and  $\tau(x(1-p_i))\to \tau(x(1-p))$.
Using the formula above, it follows that $E_{W^*(p_i)}(x)\to E_{W^*(p)}(x)$ in ${\rm L}^2(M)$.
\medskip
\noindent Next, we show that \(\Psi\) is injective. Suppose that $W^*(p)=W^*(q)$ for \(p,q\in\mathcal P(e_0,r)\). The non-zero minimal projections of this two-dimensional
algebra are $p,\ 1-p$ but also $q,\ 1-q$. Hence $q=p$ or $q=1-p$. Since $\tau(q)=r<1/2$ and $\tau(1-p)=1-r>1/2$ we deduce that $q=p$ hence \(\Psi\) is injective.
\medskip
\noindent We now prove that the inverse is continuous on the image. Note that for any $p\in\mathcal{P}(e_0,r)$ we have $e_0p=e_0$. Therefore $\tau(e_0p)=\tau(e_0)=s$ and $\tau(e_0(1-p))=0$. Using the conditional expectation formula we get $E_{W^*(p)}(e_0)
        =
        \frac{\tau(e_0p)}{r}p
        +
        \frac{\tau(e_0(1-p))}{1-r}(1-p)
        =
        \frac{s}{r}p$. Hence $p=\frac{r}{s}E_{W^*(p)}(e_0)$. Consequently, if $W^*(p_i)\to W^*(p)$ in $\Sub(M)$ then
\[
        p_i
        =
        \frac{r}{s}E_{W^*(p_i)}(e_0)
        \longrightarrow
        \frac{r}{s}E_{W^*(p)}(e_0)
        =
        p
\]
in \({\rm L}^2(M)\). Thus \(\Psi^{-1}\) is continuous on its image.
\medskip
\noindent It remains to prove that the image is closed. Let \((p_i)_i\) be a net in
\(\mathcal P(e_0,r)\) such that
\[
        W^*(p_i)\to N
        \quad
        \text{in }\Sub(M).
\]
Then $p_i= \frac{r}{s}E_{W^*(p_i)}(e_0)\rightarrow\frac{r}{s}E_N(e_0)$ in ${\rm L}^2(M)$. Since the set \(\mathcal P(e_0,r)\) is closed in the \({\rm L}^2\)-topology, there exists
\(p\in\mathcal P(e_0,r)\) such that $p_i\to p$ in ${\rm L}^2(M)$. By continuity, $W^*(p_i)\to W^*(p)$ and since  \(\Sub(M)\) is Hausdorff $N=W^*(p)$.\end{proof}
In order to upgrade an embedding into $\Sub(zM)$ for a diffuse summand $zM$ to one into the whole $\Sub(M)$, we use a decompression trick, in particular to extend a subalgebra of $zM$ to a subalgebra of $M$ by adjoining all of $(1-z)M$ on the orthogonal complement.
\begin{lemma}\label{LemDecompression}
Let $M$ be a finite von Neumann algebra with separable predual. For any non-zero central projection $z\in M$, the map
$$\Theta:\Sub(zM)\to\Sub(M)\quad N\mapsto N\oplus(1-z)M$$
is a closed topological embedding.
\end{lemma}
\begin{proof}
Note that, for $N\in\Sub(zM)$, one has $E_{\Theta(N)}(a)=E_N(za)+(1-z)a$ for all $a\in M$. It implies that $\Theta$ is continuous. Note also that $z\Theta(N)=N$, for all $N\in\Sub(zM)$. This implies that $\Theta$ is injective. Let us show that $\operatorname{Im}(\Theta)$ is closed. Let $ \Theta(N_i)\to Q$ in $\Sub(M)$. Since, for all $x\in M$, $(1-z)x\in \Theta(N_i)$ for all $i$, we have $E_{\Theta(N_i)}((1-z)x)=(1-z)x$. Taking the limit we get  $E_Q((1-z)x)=(1-z)x$ hence $(1-z)M\subseteq Q$. Hence,
$\Theta(zQ)=zQ\oplus (1-z)M= zQ\oplus (1-z)Q=Q$. Therefore, the image of \(\Theta\) is closed. Moreover, on the image of \(\Theta\), the inverse map is explicitly $\Theta^{-1}(A)=zA$ for all $A\in \operatorname{Im}(\Theta)$. Hence, it suffices to check that the map ${\rm Im}(\Theta)\to\Sub(zM)$, $A\mapsto zA$ is continuous. Write $E^{zM}_{zA}:zM\to zA$ the trace preserving conditional expectation. If $z\in A$ then $E^{zM}_{zA}(zx)=zE_A(x)$ for all $x\in M$. Since the map $\Sub(M)\to M$, $A\mapsto E_A(x)$ is $\Vert\cdot\Vert_2$-continuous, the left multiplication by $z$ is continuous on ${\rm L}^2$ and, for all $A\in{\rm Im}(\Theta)$ one has $z\in A$, we deduce that $A\mapsto zA$ is continuous on ${\rm Im}(\Theta)$.\end{proof}
\noindent In the next proposition, we characterize relatively compact subsets of $\Sub(M)$.
\begin{proposition}
\label{prop:relative-compactness-subalgebras}
Let $(M,\tau)$ be a finite von Neumann algebra with separable predual, and let
$\mathcal F\subseteq\Sub(M)$. The following are equivalent.
\begin{enumerate}
\item The set $\mathcal F$ is relatively compact in $\Sub(M)$ for the
Effros--Mar\'echal topology.
\item For every $x\in M$, the set $\{E_N(x):N\in\mathcal F\}$ is relatively compact in $\LL^2(M,\tau)$.
\end{enumerate}
\end{proposition}
\begin{proof}
The implication ${\rm(1)}\Rightarrow{\rm(2)}$ follows from the continuity of
the map $N\mapsto E_N(x)$. Assume ${\rm(2)}$. Since $\Sub(M)$ is metrizable, it suffices to show that every sequence in $\mathcal{F}$ has a subsequence converging in $\Sub(M)$. Choose an $\LL^2$-dense sequence
$(x_k)_{k\geq1}$ in $M$. Let $(N_n)_{n\geq1}$ be a sequence in
$\mathcal F$. Since \(M\) has separable
predual, choose a countable \({\rm L}^2\)-dense subset $\{x_1,x_2,\dots\}\subset M$. By the compactness claim and a diagonal extraction, there exists a subsequence $(N_{m_j})_j$ such that, for every \(k\geq1\), the sequence $(E_{N_{m_j}}(x_k))_j$ converges in \({\rm L}^2(M)\). Note that $(E_{N_{m_j}}(x))_j$ converges in \({\rm L}^2(M)\) for all $x\in M$. Indeed, using contractivity in ${\rm L}^2$ of the maps $E_N$ and the density in ${\rm L}^2$ of the sequence $(x_k)_k$, it is easy to check that the sequence $(E_{N_{m_j}}(x))_j$  is Cauchy in ${\rm L}^2$. Define the linear map $E:M\to \LL^2(M)$ by $E(x):=\lim_jE_{N_{m_j}}(x)\in{\rm L}^2(M)$ and note that $\Vert E(x)\Vert_2\leq\Vert x\Vert_2$, for all $x\in M$. We claim that \(E(M)\subset M\). Indeed, for \(x\in M\) and for all $j$ one has  $\|E_{N_{m_j}}(x)\|_\infty\leq \|x\|_\infty$. Since the operator norm closed ball of $M$ (of radius $\Vert x\Vert_\infty$) is ${\rm L}^2$ closed, we deduce that $E(x)\in M$. The linear map \(E:M\to M\) is unital, positive, and trace-preserving, because
these properties pass to the \({\rm L}^2\)-limit from the maps \(E_{N_{m_j}}\). It is
also idempotent. indeed, for any $x\in M$, $E_{N_{m_j}}(E(x))\to E(E(x))$ and $E_{N_{m_j}}(x)\to E(x)$ in ${\rm L}^2(M)$. Moreover, 
\begin{eqnarray*}
\|E_{N_{m_j}}(E(x))-E_{N_{m_j}}(x)\|_2&=&\|E_{N_{m_j}}(E(x))-E_{N_{m_j}}(E_{N_{m_j}}(x))\|_2\\
&\leq&\|E(x)-E_{N_{m_j}}(x)\|_2\to 0.
\end{eqnarray*}
Therefore $ E(E(x))=E(x)$. To finish the proof, it suffices to show that $E(M)$ is a von Neumann subalgebra of $M$. Indeed, Tomiyama's theorem implies that \(E\)
is the trace-preserving conditional expectation onto $E(M)$. By construction, $N_{m_j}\to E(M)$ in $\Sub(M)$. 
\medskip
\noindent We now show that $E(M)$ is a von Neumann subalgebra of $M$. It is a unital $*$-preserving subspace. Let $x,y\in E(M)$ and write $x=E(a)$, $y=E(b)$ with $a,b\in M$. Then,
\begin{eqnarray*}
\Vert E_{N_{m_j}}(a)E_{N_{m_j}}(b)-xy\Vert_2&\leq &\Vert (E_{N_{m_j}}(a)-x)E_{N_{m_j}}(b)\Vert_2+\Vert x((E_{N_{m_j}}(b)-y)\Vert_2\\
&\leq &\Vert b\Vert\,\Vert E_{N_{m_j}}(a)-x\Vert_2+\Vert x\Vert\,\Vert E_{N_{m_j}}(b)-y\Vert_2\to_j0.
\end{eqnarray*}
Hence $\Vert E_{N_{m_j}}(xy)-E_{N_{m_j}}(E_{N_{m_j}}(a)E_{N_{m_j}}(b)))\Vert_2\leq\|xy-E_{N_{m_j}}(a)E_{N_{m_j}}(b)\|_2\to 0$ and since $E_{N_{m_j}}(xy)\to E(xy)$ and $E_{N_{m_j}}(E_{N_{m_j}}(a)E_{N_{m_j}}(b)))=E_{N_{m_j}}(a)E_{N_{m_j}}(b)\to xy$ both in ${\rm L}^2$, we deduce that $xy=E(xy)\in E(M)$. Since \(E\) is a $\Vert\cdot\Vert_2$-contraction and $E(M)=\{x\in M:E(x)=x\}$ the algebra $E(M)$ is ${\rm L}^2$-closed, hence strongly closed. Therefore $E(M)$ is a von Neumann subalgebra of \(M\).\end{proof}
We can now combine the previous proposition and the four lemmas into the main result of the section. As soon as $M$ has any diffuse part, the topology of $\Sub(M)$ is rich enough to accommodate every Polish space; absent a diffuse summand, $M$ is essentially finite-dimensional after central decomposition, $\Sub(M)$ is compact, and the URS-style machinery applies directly.
\begin{theorem}
\label{ThmEmbedding}
Let \(M\) be a finite von Neumann algebra with separable predual. The following
are equivalent.
\begin{enumerate}
    \item \(M\) has a non-zero diffuse direct summand.
    \item For every Polish space \(X\), there exists a closed topological embedding
    $$X\hookrightarrow \Sub(M).$$
    \item $\Sub(M)$ is not compact.
\end{enumerate}
\end{theorem}
\begin{proof}Note that $(2)\Rightarrow(3)$ is obvious.
\medskip
\noindent$(1)\Rightarrow(2)$. Let $z$ be a non-zero central projection such that $zM$ is diffuse. Let $\tau$ be a faithful normal tracial state on \(zM\). Since \(zM\) is diffuse, there exists a  projection $e_0\in zM$ with $0<\tau(e_0)<\frac12$. Let $r\in(\tau(e_0),1/2)$. Combining Lemma \ref{LemClosedBaireMarkedProjections}, Lemma \ref{LemMarkedProjectionsToSubalgebras}, and Lemma \ref{LemDecompression}, we get a closed topological embedding $\theta:\mathbb R^{\mathbb N}\hookrightarrow \Sub(M)$. Let \(X\) be a Polish space. By \cite[Theorem 4.17]{kechrisClassicaldescriptiveset1995}, there exists a closed topological embedding $X\hookrightarrow \mathbb R^{\mathbb N}$, and by composing with $\theta$, we get the result.
\medskip
\noindent $(3)\Rightarrow(1)$. Assume that \(M\) has no non-zero diffuse direct summand and let us show that $\Sub(M)$ is compact. By Proposition \ref{prop:relative-compactness-subalgebras} it suffices to shows that, for every $x\in M$, the set $S_x:=\{E_N(x):N\in\Sub(M)\}$ is relatively compact in ${\rm L}^2(M)$ i.e. totally bounded in ${\rm L}^2(M)$. Let \(\varepsilon>0\).
Since $M$ has separable predual and no non-zero diffuse direct summand, $M$ is a countable direct sum of finite-dimensional factors. Therefore there are increasing central projections $z_n\uparrow1$ such that $z_nM$ is finite-dimensional for every $n$. %Since $M$ is finite with separable predual and no diffuse direct summand, there exists a sequence $(z_n)_n$ of central projections in $M$ such that $z_n\to1$ strongly and $z_nM$ is finite dimensional for all $n$. 
One has $\Vert(1-z_n)x\Vert_2\to0$ and $\tau(1-z_n)\to0$. Choose \(n\) large enough so that $\|(1-z_n)x\|_2<\frac{\varepsilon}{3}$ and $\|x\|_\infty\,\tau(1-z_n)^{1/2}<\frac{\varepsilon}{3}$. For \(N\in\Sub(M)\), one has
\begin{eqnarray*}
        \|E_N(x)-z_nE_N(z_nx)\|_2
        &\leq&
        \|E_N(x)-E_N(z_nx)\|_2
        +
        \|E_N(z_nx)-z_nE_N(z_nx)\|_2  \\
        &=&
        \|E_N((1-z_n)x)\|_2
        +
        \|(1-z_n)E_N(z_nx)\|_2.\\
        &\leq & \|(1-z_n)x\|_2+\|E_N(z_nx)\|_\infty\,\tau(1-z_n)^{1/2}\\
        &<& \frac{\varepsilon}{3}+\|x\|_\infty\,\tau(1-z_n)^{1/2}<\frac{2\varepsilon}{3}.
\end{eqnarray*}
Define $T_n:=\{z_nE_N(z_nx):N\in\Sub(M)\}\subset z_nM$ and note that $T_n$ is bounded for the \({\rm L}^2\)-norm in the finite dimensional space $z_nM$. Therefore \(T_n\) is totally bounded for the \({\rm L}^2\)-norm. Hence there exist $ \eta_1,\dots,\eta_k\in {\rm L}^2(M)$ such that $T_n\subset \bigcup_{\ell=1}^k B_{\LL^2}\left(\eta_\ell,\frac{\varepsilon}{3}\right)$. Let \(N\in\Sub(M)\). Since \(z_nE_N(z_nx)\in T_n\), there exists \(\ell\in\{1,\dots,k\}\) such that $\|z_nE_N(z_nx)-\eta_\ell\|_2<\frac{\varepsilon}{3}$. Using the previous estimate, we get
$$\|E_N(x)-\eta_\ell\|_2\leq\|E_N(x)-z_nE_N(z_nx)\|_2+\|z_nE_N(z_nx)-\eta_\ell\|_2 <\frac{2\varepsilon}{3}+ \frac{\varepsilon}{3}=\varepsilon.$$
Therefore $S_x\subset\bigcup_{\ell=1}^k B_{\LL^2}(\eta_\ell,\varepsilon)$. This proves that \(S_x\) is totally bounded in \(\LL^2(M)\).\end{proof}
The compactness obstruction quantified by Theorem~\ref{ThmEmbedding} is what forces the broader state-space machinery developed in Section~\ref{sec:statespace}. The naive idea of running the URS argument inside $\Sub(M)$ via Zorn's Lemma does not get off the ground whenever $M$ has a diffuse summand. Therefore, we work through the weak-$*$ compact space of states extending the canonical trace, where Zorn's lemma is available; this is the strategy developed in Section~\ref{sec:statespace}.
%%%%%%%%%%%%%%%%%%%%%%%%%%%%%%%%%%%%%%%%%%%%%%%%%%%%%%%%%%%%%%%%%%%%%%%%%%%%%%%%%%%%%%%%%%%%%%%%%%%%%%%%
\section{Uniformly recurrent subalgebras}
\label{sec:URA}
%%%%%%%%%%%%%%%%%%%%%%%%%%%%%%%%%%%%%%%%%%%%%%%%%%%%%%%%%%%%%%%%%%%%%%%%%%%%%%%%%%%%%%%%%%%%%%%%%%%%%%%%
In this section, we fix a finite von Neumann algebra $(M,\tau)$ with separable predual and consider the Effros-Mar\'echal topology on $\Sub(M)$. Let $\alpha:\Gamma\curvearrowright M$ be a trace preserving action of $\Gamma$ on $M$. Then, for all $g\in\Gamma$, the map $N\mapsto\alpha_g(N)$ is a homeomorphism of $\Sub(M)$ since, $\alpha$ being $\tau$-preserving, we have $E_{\alpha_g(N)}=\alpha_g\circ E_N\circ\alpha_{g^{-1}}$ for all $N\in\Sub(M)$. Hence, we get an action $\Gamma\curvearrowright\Sub(M)$ by homeomorphisms, still denoted $\alpha$. A subalgebra $N\subseteq M$ is called \textbf{$\Gamma$-invariant} if $\alpha_g(N)=N$ for all $g\in\Gamma$.
Recall that the Chabauty topology on the set of subgroups of $\Gamma$, denoted by $\Sub(\Gamma)$, is the smallest topology for which the maps $\Sub(\Gamma)\rightarrow\{0,1\}$ $\Lambda\mapsto 1_\Lambda(g)$ are continuous, for all $g\in\Gamma$. It is easy to see that $\Sub(\Gamma)$ is a compact metric space since it identifies with a closed subset of $\{0,1\}^\Gamma$ with the product topology. A URS on $\Gamma$ is a minimal closed $\Gamma$-invariant subset of $\Sub(\Gamma)$ (for the conjugation action). We now formalize the von Neumann analog.
\begin{definition} A \textbf{uniformly recurrent subalgebra of $M$} (URA) is a minimal closed $\Gamma$-invariant subset of $\Sub(M)$. A URA $\mathcal{U}\subset\Sub(M)$ is called \textbf{amenable} if every $N\in\mathcal{U}$ is amenable.
\end{definition}

\noindent A URA is an orbit closure, so a property which is invariant under the
action and stable under Effros--Mar\'echal limits need only be tested on a single
member. Amenability is such a property.
\medskip

\begin{remark}\label{rmk-amenbale-URA}A URA $\mathcal{U}\subset\Sub(M)$ is amenable if and only if there exists $N\in\mathcal{U}$ such that $N$ is amenable. Indeed, take $N\in\Ucal$ amenable. Then each $\alpha_g(N)$ is amenable. Hence $\{\alpha_g(N):g\in\Gamma\}$ is contained in the set $\Sub_{{\rm am}}(M)$ of amenable von Neumann subalgebras of $M$. Since $\Sub_{{\rm am}}(M)$ is closed in $\Sub(M)$
by \cite[Proposition~4.7]{THO25} (see also
\cite[Proposition 5.2]{fima2026spacesucpmapssubalgebras} for an easy proof) and by minimality we have $\Ucal=\overline{\{\alpha_g(N):g\in\Gamma\}}\subset\Sub_{{\rm am}}(M)$.
\end{remark}
\noindent We now discuss two basic invariants of an URA.
\begin{lemma}\label{lem:group-index-of-ura}
Let $\Gamma$ be a countable discrete group acting by trace-preserving
automorphisms on a finite von Neumann algebra $(M,\tau)$, and let
$\mathcal U\subseteq\Sub(M)$ be a URA. Then $[\Gamma:\operatorname{Stab}_\Gamma(N)]$
does not depend on $N\in\mathcal U$.
\end{lemma}
\begin{proof}
Suppose first that
$[\Gamma:\operatorname{Stab}_\Gamma(N)]<\infty$ for some
$N\in\mathcal U$. Then $\Gamma\cdot N$ is finite and hence closed in
$\Sub(M)$. Since $\Gamma\cdot N$ is dense in $\mathcal U$ by minimality, $
\mathcal U=\Gamma\cdot N$. Thus the action on $\mathcal U$ is transitive. The stabilizers of any two
elements of $\mathcal U$ are conjugate and therefore have the same index.
If no element of $\mathcal U$ has a finite-index stabilizer, then all the
indices are infinite and, since $\Gamma$ is countable, they are all equal
to $\aleph_0$.
\end{proof}
\begin{definition}\label{def:group-index-of-ura}
Let $\Gamma$ be a countable discrete group acting by trace-preserving
automorphisms on a finite von Neumann algebra $(M,\tau)$, and let
$\mathcal U\subseteq\Sub(M)$ be a URA. The constant value of the map $\Ucal\to\N\cup\{+\infty\}$, $N\mapsto [\Gamma:\operatorname{Stab}_\Gamma(N)]$ is called the
\textbf{index of $\mathcal U$ in $\Gamma$} and denoted by $
[\Gamma:\mathcal U]$.
\end{definition}

\noindent The second invariant does not refer to the acting group at all. It is
the Pimsner--Popa index of the members of $\Ucal$ in $M$, and its constancy along
$\Ucal$ comes from a different mechanism. The Pimsner--Popa constant $\lambda$ is
upper semicontinuous on $\Sub(M)$ and invariant under trace-preserving
automorphisms, so it cannot drop along an orbit and cannot drop along a limit
either. In Subsection~\ref{subsec:exoticURAs} this invariant will separate two
families of URAs built from the same element of $\Gamma$.
\medskip

\begin{lemma}\label{lem:pimsner-popa-index-of-ura}
Let $\Gamma$ be a countable discrete group acting by trace-preserving
automorphisms on a finite von Neumann algebra $(M,\tau)$ with separable
predual, and let $\mathcal U\subseteq\Sub(M)$ be a URA. Then the Pimsner--Popa index
$[M:N]$ does not depend on $N\in\mathcal U$.
\end{lemma}
\begin{proof}
For $N\in\Sub(M)$, put $\lambda(N):=\sup\{c>0:E_N(x)\geq cx\text{ for every }x\in M^+\}$. The map $\lambda:\Sub(M)\to[0,1]$ is upper semicontinuous. Indeed, for
every $c\in[0,1]$, the set $\{N\in\Sub(M):\lambda(N)\geq c\}$ is closed: if $N_i\to N$ and $\lambda(N_i)\geq c$ for every $i$, then,
for every $x\in M^+$, $E_{N_i}(x)-cx\geq0$ and $E_{N_i}(x)-cx\to E_N(x)-cx$ in $\LL^2(M,\tau)$. Since the positive cone of $\LL^2(M,\tau)$ is closed,
$E_N(x)-cx\geq0$, and hence $\lambda(N)\geq c$.
\medskip
\noindent Lemma~\ref{lem:index-conjugation-invariant} shows that $\lambda$ is
$\Gamma$-invariant. Let $N,P\in\mathcal U$. By minimality, there is a net
$(g_i)_i$ in $\Gamma$ such that $\alpha_{g_i}(N)\to P$. Upper
semicontinuity and invariance give
$$
\lambda(P)\geq\limsup_i\lambda(\alpha_{g_i}(N))=\lambda(N).
$$
Interchanging $N$ and $P$ gives $\lambda(N)\geq\lambda(P)$. Therefore
$\lambda(N)=\lambda(P)$ and
$$
[M:N]=\lambda(N)^{-1}=\lambda(P)^{-1}=[M:P].
$$\end{proof}
\begin{definition}\label{def:pimsner-popa-index-of-ura}
Let $\Gamma$ be a countable discrete group acting by trace-preserving
automorphisms on a finite von Neumann algebra $(M,\tau)$, and let
$\mathcal U\subseteq\Sub(M)$ be a URA. The constant value of the map $\Ucal\to[1,+\infty]$, $N\mapsto [M:N]$ is called the
\textbf{Pimsner--Popa index of $\mathcal U$ in $M$} and denoted by $
[M:\mathcal U]$.
\end{definition}
\noindent Let us now present some examples of URA. Two basic constructions of URAs are immediate from the definition. First, every URS of $\Gamma$ gives a URA of ${\rm L}(\Gamma)$ through the continuous map $\Lambda \mapsto L(\Lambda)$; the resulting URA is automatically compact since the Chabauty space $\Sub(\Gamma)$ is. Second, the singleton $\{\C 1\}$ is always an amenable URA, the trivial one. 
\begin{example}\label{ExBasic}The following are always URAs of ${\rm L}(\Gamma)$.
\begin{enumerate}
    \item For every URS $X\subset\Sub(\Gamma)$, the image $\psi(X)\subset\Sub({\rm L}(\Gamma))$ under the continuous injective map $\psi:\Sub(\Gamma)\rightarrow\Sub({\rm L}(\Gamma))$, $\Lambda\mapsto L(\Lambda)$ (continuity is the special case $M=\C$ of Lemma~\ref{LemmaChabauty} below). Such a URA is always compact.
    \item The trivial URA $\mathcal{U}:=\{\C 1\}$, which is amenable.
\end{enumerate}
\end{example}
Neither of these settles the interesting question --- whether ${\rm L}(\Gamma)$ admits URAs that genuinely belong to the operator-algebraic world rather than reducing to subgroup data --- so we set apart the URAs that do. Consequently, we define the following.
\begin{definition}\label{def:exotic}
A URA of ${\rm L}(\Gamma)$ is called \textbf{exotic} if it is not of the form $\psi(X)$ for any URS $X$ of $\Gamma$. In particular, every non-compact URA is exotic.
\end{definition}
The rest of this section is devoted to constructing an explicit family of exotic URAs. Specifically, we prove that every minimal Polish space arises as a URA inside ${\rm L}(\Lambda \wr \mathbb F_\infty)$ for $\Lambda$ countable abelian. Two observations motivate this construction. First, an obstruction lemma says that if a continuous $\Gamma$-equivariant image of a minimal compact space has a discrete orbit closure, then the stabilizer must have finite index. Second, every URA in $\Sub(L(\Gamma))$ is automatically a minimal Polish space.
\begin{lemma}\label{LemObstruction}
Let \(\Gamma\curvearrowright Z\) be a minimal action on a compact space, \(\Gamma\curvearrowright X\) an action on a topological space $X$ and $\Phi:Z\to X$ a \(\Gamma\)-equivariant map. If \(z_0\in Z\) is a point of continuity of
\(\Phi\) and $\mathcal O:=\overline{\Gamma\cdot \Phi(z_0)}$ is discrete then $\operatorname{Stab}_\Gamma(\Phi(z_0))$ has finite index.
\end{lemma}
\begin{proof}
Write $x_0:=\Phi(z_0)$. Since \(\mathcal O\) is discrete, there exists an open neighborhood
\(W\subset X\) of \(x_0\) such that $W\cap\mathcal O=\{x_0\}$. Since \(\Phi\) is continuous at \(z_0\), there exists an open neighborhood
\(V\subset Z\) of \(z_0\) such that $\Phi(V)\subset W$. Since $\Gamma\curvearrowright Z$ is minimal, the return-time set $N(z_0,V):=\{g\in\Gamma:gz_0\in V\}$ is syndetic. Hence, it suffices to show that $N(z_0,V)\subset \operatorname{Stab}_\Gamma(x_0)$. Let \(g\in N(z_0,V)\), then \(gz_0\in V\), hence $\Phi(gz_0)\in W$. By equivariance, $\Phi(gz_0)=g\Phi(z_0)=gx_0$. Moreover \(gx_0\in\mathcal O\). Thus $gx_0\in W\cap\mathcal O=\{x_0\}$ so \(gx_0=x_0\), i.e. $g\in\operatorname{Stab}_\Gamma(x_0)$.\end{proof}
Proposition~\ref{PropMinPolish} provides the dual fact that any minimal Polish space $X$ is acted upon minimally by some countable group $\Gamma\leq\Homeo(X)$. To realize $X$ as a URA in $\Sub(M)$, one therefore needs only a closed topological $\Gamma$-equivariant encoding $X\hookrightarrow\Sub(M)$; minimality of the $\Gamma$-orbit on $X$ then transports automatically. The next lemma constructs such an encoding inside an infinite tensor product, starting from a closed topological embedding $\beta:X\hookrightarrow\Sub(M)$ that need not yet be equivariant.
\begin{lemma}\label{LemInfiniteTens} Let $(M,\tau)$ be a finite von Neumann algebra with separable predual. Let $\Gamma\curvearrowright X$ be an action by homeomorphisms of a countable group $\Gamma$ on the topological space $X$. If $\beta:X\hookrightarrow\Sub(M)$ is a closed topological embedding then the map
$$\widetilde{\beta}:X\to\Sub\left(\bigotimes_{g\in\Gamma}(M,\tau)\right),\quad x\mapsto\bigotimes_{g\in\Gamma}(\beta(g^{-1}\cdot x),\tau)$$
is a closed topological embedding.\end{lemma}
\begin{proof}
Recall that the von Neumann algebra $\mathcal M=\bigotimes_{g\in\Gamma}(M,\tau)$ satisfies $\mathcal{M}=\mathcal{A}''$, where $\mathcal{A}$ is the linear span of elementary tensors $ \bigotimes_{g\in\Gamma}a_g$, $a_g\in M$,  $a_g=1$ for all but finitely many $g$ with product trace $\tau_{\mathcal M}\left(\bigotimes_{g\in\Gamma}a_g\right)=\prod_{g\in\Gamma}\tau(a_g)$. Note that $\mathcal{A}$ is a unital $*$-subalgebra of $\mathcal{M}$ and it is $\Vert\cdot\Vert_2$ dense in $\mathcal{M}$.
Note that $\widetilde\beta(x)
        =
        \bigotimes_{g\in\Gamma}\beta(g^{-1}\cdot x)$ is the von Neumann subalgebra of \(\mathcal M\) generated by all elementary tensors $\bigotimes_{g\in\Gamma}b_g$, $b_g\in\beta(g^{-1}\cdot x)$, with  $b_g=1$ for all but finitely many $g$. Moreover, the trace-preserving conditional expectation onto \(\widetilde\beta(x)\) satisfies
\[
        E_{\widetilde\beta(x)}
        \left(
        \bigotimes_{g\in\Gamma}a_g
        \right)
        =
        \bigotimes_{g\in\Gamma}E_{\beta(g^{-1}\cdot x)}(a_g)
\]
for every elementary tensor \(\bigotimes_g a_g\), where \(a_g=1\) for all but
finitely many \(g\).
\medskip
\noindent We now prove continuity. Define the vector space
$$\mathcal{Z}:=\{a\in\mathcal{M}:X\to{\rm L}^2(\mathcal{M}),\,\,x\mapsto E_{\widetilde{\beta}(x)}(a)\xi_{\tau_{\mathcal{M}}}\text{ is continuous}\}$$
and note that $\mathcal{Z}$ is $\Vert\cdot\Vert_2$-closed in $\mathcal{M}$. Indeed, if $\Vert a-a_n\Vert_2\to 0$ with $a_n\in\mathcal{Z}$ for all $n$ and $a\in\mathcal{M}$ then, for all $x\in X$,
$$\Vert E_{\widetilde{\beta}(x)}(a)\xi_{\tau_{\mathcal{M}}}-E_{\widetilde{\beta}(x)}(a_n)\xi_{\tau_{\mathcal{M}}}\Vert \leq\Vert a-a_n\Vert_2\rightarrow 0.$$
hence, the sequence of continuous functions $(x\mapsto E_{\widetilde{\beta}(x)}(a_n)\xi_{\tau_{\mathcal{M}}})$ converges uniformly on $X$ to the function $(x\mapsto E_{\widetilde{\beta}(x)}(a)\xi_{\tau_{\mathcal{M}}})$ which implies that $a\in\mathcal{Z}$. Hence, to show continuity of $\widetilde{\beta}$ it suffices to show that $\mathcal{A}\subset\mathcal{Z}$ and since $\mathcal{Z}$ is a linear space, it suffices to show that it contains elementary tensors. Let $a=\bigotimes_{g\in\Gamma}a_g$ be an elementary tensor and consider the finite set $F:=\{g\in\Gamma:a_g\neq1\}$. Let $(x_i)_i$ be a net in $X$ converging to $x\in X$. The conditional expectation formula and the trace formula give that $\Vert E_{\widetilde\beta(x)}(a)\xi_{\tau_{\mathcal{M}}}-E_{\widetilde\beta(x_i)}(a)\xi_{\tau_{\mathcal{M}}}\Vert$ equals:
\begin{eqnarray*}
\Vert E_{\widetilde\beta(x)}(a)-E_{\widetilde\beta(x_i)}(a)\Vert_2
&=& \left\Vert\bigotimes_{g\in \Gamma}E_{\beta(g^{-1}\cdot x)}(a_g)-\bigotimes_{g\in \Gamma}E_{\beta(g^{-1}\cdot x_i)}(a_g)\right\Vert_2\\
&=& \left\Vert\bigotimes_{g\in F}E_{\beta(g^{-1}\cdot x)}(a_g)-\bigotimes_{g\in F}E_{\beta(g^{-1}\cdot x_i)}(a_g)\right\Vert_2\\
&\leq& C^{\vert F\vert-1}\sum_{g\in F}\left\Vert E_{\beta(g^{-1}\cdot x)}(a_g)-E_{\beta(g^{-1}\cdot x_i)}(a_g)\right\Vert_2,
\end{eqnarray*}
where $C:=\max\{1,\|a_g\|_2:g\in F\}$. Since the action is by homeomorphisms and \(\beta\) is continuous, we deduce that $a\in\mathcal{Z}$.
\medskip
\noindent Next, we prove that \(\widetilde\beta\) is injective and that its inverse on the image
is continuous. For \(a\in M\) define $a^{(e)}$ to be the elementary tensor $a^{(e)}:=\bigotimes_{g\in F} a_g\in\mathcal{A}$, where $a_e=a$ and $a_g=1$ for $g\neq e$. Note that $M\rightarrow\mathcal{M}$, $a\mapsto a^{(e)}$ is a normal isomorphism and $E_{\widetilde\beta(x)}(a^{(e)})=E_{\beta(x)}(a)^{(e)}$. Let $(x_i)_i$ be a net in $X$ and $x\in X$ and suppose that $\widetilde\beta(x_i)\to \widetilde\beta(x)$ in $\Sub(\mathcal{M})$. Testing on \(a^{(e)}\), we get
\[
        E_{\beta(x_i)}(a)^{(e)}
        =
        E_{\widetilde\beta(x_i)}(a^{(e)})
        \to
        E_{\widetilde\beta(x)}(a^{(e)})
        =
        E_{\beta(x)}(a)^{(e)}
\]
in \({\rm L}^2(\mathcal M)\). Hence $E_{\beta(x_i)}(a)\to E_{\beta(x)}(a)$ in ${\rm L}^2(M)$ for every \(a\in M\). Thus $\beta(x_i)\to \beta(x)$ in $\Sub(M)$. Since \(\beta\) is a topological embedding, this implies that $ x_i\to x$. In particular, \(\widetilde\beta\) is injective and its inverse on its image is
continuous.
\medskip
\noindent It remains to show that the image is closed. Let $(x_i)_i$ be a net in $X$ and $Q\in\Sub(\mathcal M)$ be such that $\widetilde\beta(x_i)\to Q$. We again restrict to the identity coordinate \(e\). For \(a\in M\),
\[
        E_{\widetilde\beta(x_i)}(a^{(e)})
        =
        E_{\beta(x_i)}(a)^{(e)}.
\]
The left-hand side converges in \(L^2(\mathcal M)\) to \(E_Q(a^{(e)})\). Let $M_e\subset\mathcal{M}$ be the von Neumann subalgebra consisting of elementary tensor $\bigotimes_{g\in\Gamma}a_g$ with $a_g=1$ for all $g\neq 1$. Since ${\rm L}^2(M_e)\subset {\rm L}^2(\mathcal M)$ is a closed subspace and each \(E_{\beta(x_i)}(a)^{(e)}\) belongs to \(L^2(M_e)\),
we get $E_Q(a^{(e)})\in L^2(M_e)$. But \(E_Q(a^{(e)})\in Q\subset\mathcal M\), and $L^2(M_e)\cap \mathcal M = M_e$
inside \(L^2(\mathcal M)\). Therefore $E_Q(a^{(e)})\in M_e$. Hence there is a unique element \(T(a)\in M\) such that $E_Q(a^{(e)})=T(a)^{(e)}$. Define $B:=\{a\in M:a^{(e)}\in Q\}$. Since $B^{(e)}=Q\cap M_e$, \(B\) is a von Neumann subalgebra of \(M\). We claim that \(T=E_B\). Indeed, since \(E_Q\) is linear, unital, positive, normal, and trace-preserving, the same is true for \(T\). Being positive and unital, $T$ is a contraction and, since it is trace-preserving, it suffices to show, by Tomiyama's Theorem, that $T(M)\subset B$ and $T(b)=b$ for all $b\in B$. For every \(a\in M\), $T(a)^{(e)}=E_Q(a^{(e)})\in Q$.
Since \(T(a)^{(e)}\in M_e\), we get
\[
        T(a)^{(e)}\in Q\cap M_e=B^{(e)}.
\]
Hence $T(a)\in B$ so $T(M)\subset B$. Let \(b\in B\), then \(b^{(e)}\in Q\), and since \(E_Q\) fixes \(Q\)
pointwise, $T(b)^{(e)}=E_Q(b^{(e)})=b^{(e)}$. Thus $T(b)=b$. This shows that $T=E_B$. Consequently, for every \(a\in M\), $E_{\beta(x_i)}(a) \to E_B(a)$ in ${\rm L}^2(M)$. Thus $ \beta(x_i)\to B$.
Since \(\beta(X)\) is closed in \(\Sub(M)\), there exists \(x\in X\) such that $B=\beta(x)$. Since \(\beta\) is a topological embedding, the convergence $\beta(x_i)\to \beta(x)$ implies $ x_i\to x$. 
By the continuity of \(\widetilde\beta\), already proved, $\widetilde\beta(x_i)\to \widetilde\beta(x)$. But also $\widetilde\beta(x_i)\to Q$ and since \(\Sub(\mathcal M)\) is Hausdorff for the Effros-Mar\'echal topology, limits are
unique. Therefore $Q=\widetilde\beta(x)$. So the image of \(\widetilde\beta\) is closed.\end{proof}
With all the pieces in place, we now prove the universality theorem for exotic URAs. We use the strategy of Theorem~\ref{ThmEmbedding} to encode an arbitrary minimal Polish space $X$ as a closed topological subset of $\Sub(L(\Lambda))$, transport this embedding through Lemma~\ref{LemInfiniteTens} to $\Sub(L(\Lambda^{(\mathbb F_\infty)}))$, and finally use a minimal action of $\mathbb F_\infty$ on $X$ (guaranteed by Proposition~\ref{PropMinPolish}) to make the image a URA.
\begin{theorem}
\label{UniversalWreath}
Let \(\Lambda\) be a countably infinite abelian group. For every minimal Polish space \(X\), there exists an exotic
URA
$$\mathcal U_X\subset \Sub({\rm L}(\Lambda\wr \mathbb F_\infty))$$
such that $\Ucal_X$ is homeomorphic to $X$.
\end{theorem}
\begin{proof}
Write $\Gamma:=\Lambda\wr \mathbb F_\infty$ and let \(X\) be a minimal Polish space. By Proposition~\ref{PropMinPolish}, there exists
a countable group $\Gamma_X\leq \Homeo(X)$ acting minimally on $X$. Choose a surjective homomorphism $\pi:\mathbb F_\infty\to \Gamma_X$. Thus \(\mathbb F_\infty\) acts minimally on \(X\) by $g\cdot x:=\pi(g)x$.
\medskip
\noindent Since \(\Lambda\) is infinite, the finite von Neumann algebra \({\rm L}(\Lambda)\) is
diffuse hence,  there exist a projection \(e_0\in{\rm L}(\Lambda)\) with $0<\tau(e_0)<\frac12$. Fix a $r\in(\tau(e_0),1/2)$. Combining Lemma \ref{LemClosedBaireMarkedProjections} and Lemma \ref{LemMarkedProjectionsToSubalgebras} we get a closed topological embedding $\theta:\mathbb R^{\mathbb N}\hookrightarrow \Sub({\rm L}(\Lambda))$. By Kechris \cite[Theorem 4.17]{kechrisClassicaldescriptiveset1995} there exists a closed topological embedding $X\hookrightarrow \mathbb R^{\mathbb N}$ and by compositing with $\theta$ we get a closed topological embedding $\beta:X\hookrightarrow \Sub({\rm L}(\Lambda))$, $x\mapsto B_x:=W^*(p_x)=\C p_x\oplus\C(1-p_x)$ where $p_x\in\mathcal{P}(e_0,r)$ for all $x\in X$.
\medskip
\noindent We use the standard identification:
$$\left({\rm L}(\Lambda^{(\mathbb{F}_\infty)}),\tau_{\Lambda^{(\mathbb{F}_\infty)}}\right)\simeq\bigotimes_{g\in\mathbb{F}_\infty}\left({\rm L}(\Lambda),\tau_\Lambda\right)$$
Under this identification, the trace-preserving action $\mathbb{F}_\infty\curvearrowright {\rm L}(\Lambda^{(\mathbb{F}_\infty)})$ becomes for $a=\bigotimes_{g\in\mathbb{F}_\infty} a_g$ an elementary tensor, with $a_g\in{\rm L}(\Lambda)$ such that $a_g=1$ for all but finitely many $g\in\mathbb{F}_\infty$ and for $h\in \mathbb{F}_\infty$, $h\cdot a:=\bigotimes_{g\in\mathbb{F}_\infty} a_{h^{-1}g}$.
For $x\in X$ define $M_x:=\bigotimes_{g\in\mathbb{F}_\infty}\left(B_{g^{-1}\cdot x},\tau_\Lambda\right)\subset{\rm L}(\Lambda^{(\mathbb{F}_\infty)})\subset {\rm L}(\Gamma)$. It follows from Lemma \ref{LemInfiniteTens} that the map $\iota:X\to{\rm Sub}({\rm L}(\Gamma))$, $x\mapsto M_x$, is a closed topological embedding and, using the identification above, we see that $\iota$ is $\mathbb{F}_\infty$-equivariant  where the $\mathbb{F}_\infty$-action on ${\rm Sub}({\rm L}(\Gamma))$ is given by restriction of the $\Gamma$ action to the subgroup $\mathbb{F}_\infty\leq\Gamma$. Then, the set $\mathcal{U}_X:=\iota(X)=\{M_x:x\in X\}$ is a minimal closed $\mathbb{F}_\infty$-invariant subset of ${\rm L}(\Gamma)$ (since the action $\mathbb{F}_\infty\curvearrowright X$ is minimal). Since $\Lambda^{(\mathbb{F}_\infty)}$ is abelian and $\mathcal{U}_X\subset \Sub({\rm L}(\Lambda^{(\mathbb{F}_\infty)}))$, the set $\mathcal{U}_X$ is actually $\Gamma$-invariant. Hence, $\mathcal{U}_X\subset \Sub({\rm L}(\Gamma))$ is a URA, and it is homeomorphic to $X$. Let us show that $\mathcal{U}_X$ is exotic. Fix \(x\in X\) and suppose, by
contradiction, that $M_x={\rm L}(H)$ for some subgroup \(H\leq\Gamma\). Since $M_x\subset{\rm L}\bigl(\Lambda^{(\mathbb F_\infty)}\bigr)$ we must have $H\leq \Lambda^{(\mathbb F_\infty)}$.
Let $\Lambda_e\leq \Lambda^{(\mathbb F_\infty)}$ be the subgroup of $\Lambda^{(\mathbb{F}_\infty)}$ consisting of maps $a:\mathbb{F}_\infty\to\Lambda$ satisfying $a(g)=e$ for all $g\neq e$. One has $E_{{\rm L}(\Lambda_e)}({\rm L}(H))={\rm L}(H\cap \Lambda_e)$ and, using the notations of the proof of Lemma \ref{LemInfiniteTens}, $E_{{\rm L}(\Lambda_e)}(M_x)=B_x^{(e)}$. Since $M_x={\rm L}(H)$ we deduce that $B_x^{(e)}={\rm L}(H\cap \Lambda_e)={\rm L}(H_e)^{(e)}$, where $H_e:=\{\gamma\in\Lambda:a_\gamma\in H\}\leq\Lambda$, where $a_\gamma(e)=\gamma$ and $\gamma(g)=e$ for all $g\neq e$. By injectivity of the map ${\rm L}(\Lambda)\to{\rm L}(\Lambda^{(\mathbb{F}_\infty)})$, $x\mapsto x^{(e)}$, we deduce that $B_x={\rm L}(H_e)$. This is a contradiction since for all $x\in X$ and all $K\in\Sub(\Lambda)$ one has \(B_x\neq{\rm L}(K)\). Indeed, $B_x=\mathbb C p_x\oplus \mathbb C(1-p_x)$ is two-dimensional, and its two non-zero minimal projections have traces $r$ and $1-r$ and $r<\frac12$. However, if ${\rm L}(K)$ is two-dimensional, then \(K\simeq\mathbb Z/2\mathbb Z\), and the
two non-zero minimal projections of \({\rm L}(K)\) both have trace \(1/2\).
\end{proof}
This completes the proof of Theorem~\ref{thm:B}. The contrast with the URS world is sharp: while compactness of $\Sub(\Gamma)$ forces every URS to be a compact metric space, the failure of compactness in $\Sub({\rm L}(\Gamma))$ --- precisely characterized by Theorem~\ref{ThmEmbedding} --- allows URAs to inherit every minimal Polish topology. %%%%%%%%%%%%%%%%%%%%%%%%%%%%%%%%%%%%%%%%%%%%%%%%%%%%%%%%%%%%%%%%%%%%%%%%%%%%%%%%%%%%%%%%%%%%%%%%%%%%%%%%
\section{Confined subalgebras, simplicity and crossed products}
\label{sec:crossedproduct}
%%%%%%%%%%%%%%%%%%%%%%%%%%%%%%%%%%%%%%%%%%%%%%%%%%%%%%%%%%%%%%%%%%%%%%%%%%%%%%%%%%%%%%%%%%%%%%%%%%%%%%%%
Having developed URAs in $\Sub(M)$ in Section~\ref{sec:URA} for an arbitrary trace-preserving action, we now specialize the setup to crossed products, where the C*-simplicity question naturally lives. Given a trace-preserving action $\Gamma\curvearrowright (M,\tau_M)$ on a finite von Neumann algebra, the natural ambient algebra is the crossed product $M\rtimes\Gamma$, and the appropriate notion replacing confinement of subgroups is relative $M$-confinement of intermediate subalgebras. We first fix the setup.
Let $(M,\tau_M)$ be a finite von Neumann algebra with separable predual and faithful normal trace, and let $\alpha:\Gamma\curvearrowright (M,\tau_M)$ be a trace-preserving action of a countable discrete group. We view $M\subset\Lcal({\rm L}^2(M))$ in standard form, with cyclic vector $\xi_\tau$, and write $u_\alpha:\Gamma\rightarrow\mathcal{U}({\rm L}^2(M))$ for the Koopman representation determined by $u_\alpha(g)(x\xi_\tau)=\alpha_g(x)\xi_\tau$. The crossed product $M\rtimes\Gamma$ is the von Neumann algebra acting on $\ell^2(\Gamma,{\rm L}^2(M))$ generated by $\pi(M)\cup\{u_g:g\in\Gamma\}$, where $\pi$ is the diagonal representation $[\pi(x)\xi](h)=x\xi(h)$ and $[u_g\xi](h):=u_\alpha(g)\xi(g^{-1}h)$. We will use the standard properties:
\begin{itemize}
\item $u_g\pi(x)u_g^*=\pi(\alpha_g(x))$ for all $g\in\Gamma$ and $x\in M$;
\item there is a faithful normal $M$-bimodular conditional expectation $E_M:M\rtimes\Gamma\rightarrow M$ with $E_M(\pi(x)u_g^*)=\delta_{g,e}x$.
\end{itemize}
Then $M\rtimes\Gamma$ is a finite von Neumann algebra with faithful normal trace $\tau:=\tau_M\circ E_M$, and the triple $(\ell^2(\Gamma,{\rm L}^2(M)),\id,\Omega)$ with $\Omega(h):=\delta_{h,e}\xi_\tau$ is an explicit GNS construction for $\tau$. We will routinely identify $M$ with $\pi(M)\subset M\rtimes\Gamma$.

\noindent We record for later use the value of the conditional expectation onto
a crossed product by a subgroup. It is the identity which makes the Chabauty
topology on $\Sub(\Gamma)$ interact with the Effros--Mar\'echal topology on
$\Sub(M\rtimes\Gamma)$.
\medskip

\begin{remark}\label{RmkEMrtimesLambda}
For any subgroup $\Lambda\leq\Gamma$, the conditional expectation onto $M\rtimes\Lambda\subset M\rtimes\Gamma$ satisfies $E_{M\rtimes\Lambda}(u_s)=\mathbf{1}_\Lambda(s)\,u_s$ for all $s\in\Gamma$.
\end{remark}
To test whether a net of intermediate subalgebras of $M\rtimes\Gamma$ containing $M$ converges to $M$ in the Effros-Mar\'echal topology, it is enough to test on the canonical unitaries $u_g$. This reduction is the content of the next lemma, which will be applied repeatedly throughout the section.
\begin{lemma}\label{LemExpectationConvergenceFromBoundary}
Let $\Gamma \curvearrowright (M,\tau_M)$ be a trace-preserving action on a finite
von Neumann algebra and let $(N_i)_i$ be a net of von Neumann subalgebras of $M\rtimes\Gamma$ containing $M$. The following are equivalent.
\begin{enumerate}
\item $\|E_{N_i}(u_g)\|_2\longrightarrow 0$ $\forall g\in \Gamma\setminus\{e\}$.
\item $N_i\rightarrow M$.
\end{enumerate}
\end{lemma}
\begin{proof}
$(1)\Rightarrow(2)$. It suffices to show that $\Vert E_{N_i}(x)-E_M(x)\Vert_2\rightarrow 0$ for all $x\in M\rtimes\Gamma$. Let $X:=\{x\in M\rtimes\Gamma:\Vert E_{N_i}(x)-E_M(x)\Vert_2\rightarrow 0\}$. Then $X$ is clearly a linear subspace of $M\rtimes\Gamma$ and it contains the set $\{au_g:g\in\Gamma,a\in M\}$ since:
$$\Vert E_{N_i}(au_g)-E_M(au_g)\Vert_2=\Vert a(E_{N_i}(u_g)-E_M(u_g)\Vert_2\leq\Vert a\Vert\,\Vert E_{N_i}(u_g)-\delta_{g,e}1\Vert_2\rightarrow 0.$$
Hence, it suffices to show that $X$ is $\Vert\cdot\Vert_2$-closed. Let $(x_n)_n$ be a sequence in $X$ and $x\in M\rtimes\Gamma$ be such that $\Vert x_n-x\Vert_2\rightarrow 0$. Fix $\varepsilon>0$ and let $n_0\in\N$ be such that $\Vert x_{n_0}-x\Vert_2<\frac{\varepsilon}{4}$. Since $x_{n_0}\in X$, there exists $i_0$ such that, for all $i\geq i_0$ $\Vert E_{N_i}(x_{n_0})-E_M(x_{n_0})\Vert_2<\frac{\varepsilon}{2}$. Then, since $E_N$ is contractive in $\Vert\cdot\Vert_2$-norm for any subalgebra $N$ one has, for all $i\geq i_0$,
\begin{align*}
 \Vert E_{N_i}(x)-E_M(x)\Vert_2&\leq \Vert E_{N_i}(x-x_{n_0})\Vert_2+\Vert E_{N_i}(x_{n_0})-E_M(x_{n_0})\Vert_2\\&+\Vert E_M(x_{n_0}-x)\Vert_2
 \\&\leq 2\Vert x_{n_0}-x\Vert_2+\Vert E_{N_i}(x_{n_0})-E_M(x_{n_0})\Vert_2<\varepsilon.
 \end{align*}
\medskip
\noindent $(2)\Rightarrow(1)$. This is obvious since $E_M(u_g)=0$ $\forall g\in \Gamma\setminus\{e\}$.
\end{proof}
The next lemma transports continuity from the Chabauty topology on $\Sub(\Gamma)$ to the Effros-Mar\'echal topology on $\Sub(M\rtimes\Gamma)$. It is what makes the comparison between confined subgroups and confined subalgebras work seamlessly.
\begin{lemma}\label{LemmaChabauty} The map $\Sub(\Gamma)\to\Sub(M\rtimes\Gamma)$, $\Lambda\mapsto M\rtimes\Lambda$ is continuous.\end{lemma}
\begin{proof} Let $\Lambda_n\rightarrow\Lambda$ in $\Sub(\Gamma)$.  It suffices to show that $\Vert E_{M\rtimes\Lambda_n}(x)-E_{M\rtimes\Lambda}(x)\Vert_2\rightarrow 0$ for all $x\in M\rtimes\Gamma$. Let $X:=\{x\in M\rtimes\Gamma:\Vert E_{M\rtimes\Lambda_n}(x)-E_{M\rtimes\Lambda}(x)\Vert_2\rightarrow 0\}$. Then $X$ is a linear subspace of $M\rtimes\Gamma$ and it contains the set $\{au_g:g\in\Gamma,a\in M\}$ since:
\begin{eqnarray*}
\Vert E_{M\rtimes\Lambda_n}(au_g)-E_{M\rtimes\Lambda}(au_g)\Vert_2&=&\Vert a(E_{M\rtimes\Lambda_n}(u_g)-E_{M\rtimes\Lambda}(u_g))\Vert_2\\
&\leq&\Vert a\Vert\,\Vert E_{M\rtimes\Lambda_n}(u_g)- E_{M\rtimes\Lambda}(u_g)\Vert_2\\
&=&\Vert a\Vert\,\Vert 1_{\Lambda_n}(g)u_g- 1_\Lambda(g)u_g\Vert_2\\
&=&\Vert a\Vert\,\vert 1_{\Lambda_n}(g)- 1_\Lambda(g)\vert\to 0.\\
\end{eqnarray*}
Hence, it suffices to show that $X$ is $\Vert\cdot\Vert_2$-closed and it can be proved as in Lemma \ref{LemExpectationConvergenceFromBoundary}.\end{proof}
We consider the action $\Gamma\curvearrowright M\rtimes\Gamma$ given by $g\cdot x:=u_gxu_g^*$, for $x\in M\rtimes\Gamma$ and $g\in\Gamma$ and the associated action $\Gamma\curvearrowright\Sub(M\rtimes\Gamma)$ given by $g\cdot N:=u_gNu_g^*$ for $N\in\Sub(M\rtimes\Gamma)$ and $g\in\Gamma$. We are now in a position to introduce the central notion of this section.
\begin{definition}\label{DefRelativeConfinedSubalg}
Let $\Gamma \curvearrowright (M,\tau_M)$ be a trace-preserving action on a finite
von Neumann algebra. A subalgebra $N\subset M\rtimes\Gamma$ containing $M$ is \textbf{$M$-confined} if
\[
M\notin \overline{\{u_gNu_g^*:g\in \Gamma\}}
\]
where the closure is taken in the Effros-Mar\'echal topology on $\Sub(M\rtimes\Gamma)$.
\end{definition}
\noindent This definition above is motivated by \cite[Definition 1.1]{AJ} which was itself motivated by the notion of confined subgroup, recalled below.
\begin{definition}
Let $\Gamma$ be a countable discrete group. A subgroup $\Lambda$ is called \textbf{confined} if the trivial subgroup is not in the closure of $\{g\Lambda g^{-1}:g\in\Gamma\}$ in the Chabauty topology on the space of subgroups.
\end{definition}
\noindent These two notions are related as follows.
\begin{proposition}\label{PropConfinedSubgroup}
Let $\Gamma \curvearrowright (M,\tau_M)$ be a trace-preserving action on a finite
von Neumann algebra and $\Lambda$ a subgroup of $\Gamma$. The following are equivalent.
\begin{enumerate}
\item $\Lambda$ is a confined subgroup of $\Gamma$
\item The subalgebra $M\rtimes\Lambda\subset M\rtimes\Gamma$ is $M$-confined.
\end{enumerate}
\end{proposition}
\begin{proof}
\noindent$(1)\Rightarrow(2)$. Assume $M\rtimes\Lambda\subset M\rtimes\Gamma$ is not $M$-confined. Then there exists
a net $(g_i)_i$ in $\Gamma$ such that $u_{g_i}M\rtimes\Lambda u_{g_i}^* = M\rtimes g_i\Lambda g_i^{-1}\longrightarrow M$
in the Effros-Mar\'echal topology on $\Sub(M\rtimes \Gamma)$. Hence, for all $s\in\Gamma$,
$$1_{g_i\Lambda g_i^{-1}}(s)=\Vert 1_{g_i\Lambda g_i^{-1}}(s)u_s\Vert_2=\Vert E_{M\rtimes g_i\Lambda g_i^{-1}}(u_s))\Vert_2\longrightarrow \Vert E_M(u_s)\Vert_2=\delta_{s,e}=1_{\{e\}}(s).$$
This shows that $g_i\Lambda g_i^{-1}\longrightarrow \{e\}$ hence $\Lambda$ is not confined.
\medskip
\noindent $(2)\Rightarrow(1).$ Assume that $\Lambda$ is not confined. Then there exists a net $(g_i)_i$ in $\Gamma$ such that $g_i\Lambda g_i^{-1}\longrightarrow \{e\}$. Then, Lemma \ref{LemmaChabauty} implies that:
$$u_{g_i}M\rtimes\Lambda u_{g_i}^*=M\rtimes g_i\Lambda g_i^{-1}\rightarrow M.$$
Hence, $M\rtimes\Lambda$ is not $M$-confined.
\end{proof}
To compare confined subalgebras with hypertraces and the Furstenberg boundary, we need a Hilbert-space framework that simultaneously carries the $\Gamma$-action on $C(X)$ and the covariant representation defining the crossed product. A topological space $X$ is called a \textbf{$\Gamma$-space} if $X$ is equipped with an action $\Gamma\curvearrowright X$ by homeomorphisms. Fix a $\Gamma$-space $X$ and a Hilbert space $H$. For each $x\in X$, define the map:
$$\varphi_x^H:C(X)\rightarrow\Lcal(\ell^2(\Gamma, H))\quad[\varphi_x^H(f)\xi](g)=f(gx)\xi(g).$$
\noindent Then, for all $x\in X$, $\varphi_x^H$ is a unital $*$-homomorphism and, if $x\in X$ has a dense $\Gamma$-orbit then $\varphi_x^H$ is faithful.
\begin{remark}\label{RmkCovariant} The map $\varphi_{x_0}^{{\rm L}^2(M)}:C(X)\rightarrow \Lcal(\ell^2(\Gamma, {\rm L}^2(M)))$ is covariant i.e. for all $f\in C(X)$,
$u_h^*\varphi_{x_0}^{{\rm L}^2(M)}(f)u_h
=\varphi_{x_0}^{{\rm L}^2(M)}(h^{-1}\cdot f)$. Indeed, for \(\xi\in \ell^2(\Gamma,{\rm L}^2(M))\) and \(g\in \Gamma\), we have
\begin{eqnarray*}
\bigl[u_h^*\varphi_{x_0}^{{\rm L}^2(M)}(f)u_h\xi\bigr](g)
&=&
u_\alpha(h^{-1})\bigl[\varphi_{x_0}^{{\rm L}^2(M)}(f)u_h\xi\bigr](hg)\\
&=&
f(hgx_0)u_\alpha(h^{-1})\,(u_h\xi)(hg)\\
&=&
f(hgx_0)\,\xi(g)=(h^{-1}\cdot f)(gx_0)\xi(g)=[\varphi_{x_0}^{{\rm L}^2(M)}(h^{-1}\cdot f)\xi](g).
\end{eqnarray*}
\end{remark}
\noindent The next lemmas are the technical engine of the simplicity criterion. For the remainder of this section, we fix a trace-preserving action $\Gamma\curvearrowright(M,\tau_M)$ and a minimal compact $\Gamma$-space $X$. We write $P:=M\rtimes\Gamma$ with canonical trace $\tau$ and, for $x\in X$, $\Gamma_x:={\rm Stab}_\Gamma(x)\leq\Gamma$.
\begin{lemma}\label{LemmaFourier}\label{LemApproxBoundaryVanishing}
%Let $\Gamma \curvearrowright (M,\tau_M)$ be a trace-preserving action and $X$ a compact minimal $\Gamma$-space.
Let $(N_i)_i$ be a net of amenable von Neumann subalgebras of $P$ such that $M\subseteq N_i$ for all $i$, and let $\Phi_i\in \Hyp_\tau(N_i)$. Let $x_0\in X$ and $\mu_i:=\Phi_i\circ\varphi_{x_0}^{{\rm L}^2(M)}\in \Prob(X)$. Assume that there exists $x\in X$ such that $\mu_i\to \delta_x$ weakly-$*$. Then, for every $g\in \Gamma$ such that $gx\neq x$, and every net $(a_i)_i$ with $a_i\in (N_i)_1$, one has $\tau(a_i u_g^*)\to 0$.
\end{lemma}
\begin{proof} To ease the notations, we view $C(X)\subset\mathbb{B}(\ell^2(\Gamma,{\rm L}^2(M)))$ via the map $\varphi_{x_0}^{{\rm L}^2(M)}$. Fix $g\in \Gamma$ with $gx\neq x$, and let $(a_i)_i$ be as in the statement with
$\sup_i\|a_i\|\leq 1$. Let $f\in C(X)$ satisfy $0\leq f\leq 1$. By covariance, $u_g^*f=(g^{-1}\!\cdot f)\,u_g^*$, we get
$$
a_i u_g^*f = a_i (g^{-1}\!\cdot f)u_g^*=a_i (g^{-1}\!\cdot f)^{\frac{1}{2}}(g^{-1}\!\cdot f)^{\frac{1}{2}}u_g^*.
$$
Therefore, by $N_i$-centrality and Cauchy-Schwarz for the positive functional $\Phi_i$,
\begin{eqnarray*}
|\Phi_i(a_i u_g^*f)|^2
&=&\vert\Phi_i((g^{-1}\!\cdot f)^{\frac{1}{2}}(g^{-1}\!\cdot f)^{\frac{1}{2}}u_g^*a_i)\vert^2\leq
\Phi_i\big(g^{-1}\!\cdot f\big)\,
\Phi_i(y^*y)\iffalse)\fi,
\end{eqnarray*}
where $y:=(g^{-1}\!\cdot f)^{\frac{1}{2}}u_g^*a_i$ has norm less than $1$ which implies that $\Phi_i(y^*y)\leq 1$. We conclude that $\vert\Phi_i(a_iu_g^*f)\vert\leq \mu_i(g^{-1}\cdot f)^\frac{1}{2}$. Also note that Cauchy--Schwarz gives
$$
|\Phi_i(a_i u_g^*(1-f))|^2
\leq
\Phi_i(a_i a_i^*)\,\Phi_i((1-f)^2)
\leq
\Phi_i(1-f),
$$
because $\|a_i\|\leq 1$ and $0\leq 1-f\leq 1$. Hence $|\Phi_i(a_i u_g^*(1-f))|
\leq
(1-\mu_i(f))^{1/2}$.
\noindent Since $\Phi_i\vert_{M\rtimes\Gamma}=\tau$, combining the two bounds yields
$$
|\tau(a_i u_g^*)|=\vert\Phi_i(a_iu_g^*)\vert\leq \vert\Phi_i(a_iu_g^*f)\vert+\vert\Phi_i(a_iu_g^*(1-f))\vert
\leq
\mu_i(g^{-1}\!\cdot f)^{1/2} + (1-\mu_i(f))^{1/2}.
$$
\noindent Now choose $f\in C(X)$ such that $0\leq f\leq 1$, $f(x)=1$ and $f(gx)=0$.
Since $\mu_i\to \delta_x$, we get $\mu_i(f)\to 1$ and $\mu_i(g^{-1}\!\cdot f)\to (g^{-1}\!\cdot f)(x)=f(gx)=0$. Therefore $\tau(a_i u_g^*)\to 0$.\end{proof}

\noindent The previous lemma is only of use if one can produce, for a given
amenable subalgebra of $M\rtimes\Gamma_x$, a hypertrace whose associated
probability measure on $X$ is exactly the Dirac mass at $x$. This is the content
of the next lemma, and it is the only place where the strengthened form of the
amenability criterion, Proposition~\ref{lem:N-central-extension}, is needed.

\begin{lemma}\label{lem:pointed-hypertrace}
Let \(x\in X\), and let $N\subset M\rtimes\Gamma_x\subset P$ be an amenable von Neumann subalgebra containing $M$. Then there exists $\Phi\in\operatorname{Hyp}_\tau(N)$
such that $\Phi\circ\varphi_x^{{\rm L}^2(M)}=\delta_x$.
\end{lemma}
\begin{proof}
We use the notation $P_x:=M\rtimes\Gamma_x$,
$$D_x:=C^*(P_x,\varphi_x^{{\rm L}^2(M)}(C(X)))\subset\Lcal(\ell^2(\Gamma,{\rm L}^2(M)))$$
and $\mathcal S_x:=M\rtimes\Gamma+D_x\subset \mathbb B(\ell^2(\Gamma,{\rm L}^2(M)))$. We start with the following Claim.
\medskip
\noindent\textbf{Claim.} \textit{There exists unital positive linear map $\psi_x:\mathcal{S}_x\to\C$ such that  $\psi_x|_{P}=\tau$, $\psi_x\circ\varphi_x^{{\rm L}^2(M)}=\delta_x$, and $\psi_x(as)=\psi_x(sa)$ $\forall a\in P_x,\ s\in\mathcal S_x$.}
\medskip
\noindent\textit{Proof of the Claim.} 
Let \(e_x=1_{\Gamma_x}\ot 1\) be the orthogonal projection of
\(\ell^2(\Gamma,{\rm L}^2(M))\) onto \(\ell^2(\Gamma_x,{\rm L}^2(M))\) and consider the unital completely positive map
$$\rho_x:\Lcal(\ell^2(\Gamma,{\rm L}^2(M)))\to\Lcal(\ell^2(\Gamma_x,{\rm L}^2(M))),\quad\rho_x(T)=e_xTe_x.$$
Note that, viewing $P_x\subset \Lcal(\ell^2(\Gamma_x,{\rm L}^2(M)))$, we have $E_{P_x}(a)=\rho_x(a)$ for all $a\in P$. Moreover, for \(f\in C(X)\) and $\xi\in\ell^2(\Gamma_x,{\rm L}^2(M))$, we have
$$
        [\rho_x(\varphi^{{\rm L}^2(M)}_x(f))\xi](g)
        =
        f(gx)\xi(g)
        =
        f(x)\xi(g)\quad\text{for all }g\in\Gamma_x.
$$
Hence $\rho_x(\varphi^{{\rm L}^2(M)}_x(f))=f(x)1$. Since  \(e_x\in D_x'\), $\rho_x\vert_{D_x}$ is a unital $*$-homomorphism hence, $\rho_x(D_x)\subset P_x$. It follows that $\rho_x(\mathcal{S}_x)\subset P_x$. Define $\psi_x:=\tau\circ\rho_x\vert_{\mathcal{S}_x}$. Then $\psi_x$ is a unital and positive functional on $\mathcal{S}_x$, $\psi_x|_{P}=\tau$, $\psi_x\circ\varphi^{{\rm L}^2(M)}_x|_{C(X)}=\delta_x$. Moreover, if \(a\in P_x\) and \(d\in D_x\), then $\psi_x(ad) =\tau(a\rho_x(d))=\tau(\rho_x(d)a)=\psi_x(da)$. Finally, let \(a\in P_x\) and \(s\in\mathcal S_x\). Write $s=b+d$ with $b\in P$, $d\in D_x$ and note that $ab,\ ba\in P$ and $ad,\ da\in D_x$. Therefore,
$$\psi_x(as)=\psi_x(ab+ad)=\tau(ab)+\psi_x(ad)=\tau(ba)+\psi_x(da)=\psi_x(sa).\hfill\qed$$
\noindent\textit{End of the Proof of the Lemma.} Applying Proposition \ref{lem:N-central-extension} to the unital positive functional $\psi_x:\mathcal{S}_x\to\C$ on the $N$-$N$-bimodule $\mathcal{S}_x\subset\Lcal({\rm L}^2(P))$ we get an extension $\Phi\in S(\mathbb{B}({\rm L}^2(P)))$ such that $\Phi|_{\mathcal S_x}=\psi_x$ and $\Phi(aT)=\Phi(Ta)$ for all $a\in N,\ T\in\mathbb B({\rm L}^2(P))$. Since \(P\subset\mathcal S_x\), we get $\Phi|_P=\psi_x\vert_P=\tau$ hence $\Phi\in\operatorname{Hyp}_\tau(N)$. Moreover, since $\varphi_x^{{\rm L}^2(M)}(C(X))\subset\mathcal S_x$, we get $\Phi\circ\varphi_x^{{\rm L}^2(M)}=\delta_x$.\end{proof}
\noindent We are now in a position to prove the main simplicity theorem in the crossed-product setting.
\begin{theorem}\label{ThmUnified}
Let $\Gamma \curvearrowright (M,\tau_M)$ be a trace-preserving action of a countable discrete group on an
amenable finite von Neumann algebra with separable predual, and  $X$ be a non-empty compact Hausdorff minimal $\Gamma$-space. The following are equivalent.
\begin{enumerate}
\item The reduced crossed product $C(X) \rtimes_r \Gamma$ is simple.
\item For all $x \in X$, any amenable von Neumann subalgebra $N\subset M\rtimes\Gamma_x$ containing $M$ is not $M$-confined in the crossed product $M\rtimes\Gamma$.
\item There exists $x \in X$ such that any amenable von Neumann subalgebra $N\subset M\rtimes\Gamma_x$ containing $M$ is not $M$-confined in the crossed product $M\rtimes\Gamma$.
\end{enumerate}
\end{theorem}
\begin{proof}
It is clear that $(2)$ implies $(3)$.
\medskip
\noindent$(1)\Rightarrow(2)$. Assume that \(C(X)\rtimes_{r}\Gamma\) is simple and let $x\in X$ and $N\subset M\rtimes\Gamma_x$ an amenable subalgebra containing $M$. By Lemma \ref{lem:pointed-hypertrace} there exists $\Phi\in\operatorname{Hyp}_\tau(N)$
such that $\Phi\circ\varphi_x^{{\rm L}^2(M)}=\delta_x$. Let $\widetilde{X}$ be the generalized Furstenberg boundary of $X$ with canonical $\Gamma$-equivariant continuous surjection $q:\widetilde{X}\to X$. Fix $y_0\in q^{-1}(\{x\})$ and define $\mu:=\Phi\circ\varphi_{y_0}^{{\rm L}^2(M)}\in\Prob(\widetilde{X})$. Note that, for all $f\in C(X)$ and $\xi\in l^2(\Gamma,{\rm L}^2(M))$,
$$(\varphi_{y_0}^{{\rm L}^2(M)}(f\circ q)\xi)(g)=(f\circ q)(gy_0)\xi(g)=f(gq(y_0))\xi(g)=f(gx)\xi(g)=(\varphi_x^{{\rm L}^2(M)}(f)\xi)(g).$$
Hence, $\Phi\circ\varphi_{y_0}^{{\rm L}^2(M)}(f\circ q)=\Phi\circ\varphi_{x}^{{\rm L}^2(M)}(f)=f(x)$. It follows that $q_*\mu=\delta_x$. By the properties of $\widetilde{X}$, there exists a net \((s_i)_i\) in \(\Gamma\) and
a point \(y\in \widetilde{X}\) such that $s_i\mu \to \delta_y$ in the weak-* topology of \(\operatorname{Prob}(\widetilde{X})\). Put $N_i:=u_{s_i}Nu_{s_i}^*$ and $\Phi_i:=s_i\Phi$, where $\Gamma$ acts on states by $(g\cdot\Phi)(T):=\Phi(u_g^*Tu_g)$; with this convention $g\cdot\Phi\in\operatorname{Hyp}_\tau(u_gNu_g^*)$ whenever $\Phi\in\operatorname{Hyp}_\tau(N)$. Then, since $\Phi\in\operatorname{Hyp}_{\tau}(N)$, we deduce that \(\Phi_i\in \operatorname{Hyp}_{\tau}(N_i)\). Moreover, by Remark \ref{RmkCovariant},
%% [TEAL CORRECTION -- state-action convention: adjoints corrected in the display below; the original display is kept in the \iffalse block]
\iffalse
$$\Phi_i\circ\varphi_{y_0}^{{\rm L}^2(M)}(f)=\Phi(u_{s_i}\varphi_{y_0}^{{\rm L}^2(M)}(f)u_{s_i}^*)=\Phi\circ\varphi_{y_0}^{{\rm L}^2(M)}(s_i\cdot f)=s_i\mu(f)\to f(y),$$
\fi
$$\Phi_i\circ\varphi_{y_0}^{{\rm L}^2(M)}(f)=\Phi(u_{s_i}^*\varphi_{y_0}^{{\rm L}^2(M)}(f)u_{s_i})=\Phi\circ\varphi_{y_0}^{{\rm L}^2(M)}(s_i^{-1}\cdot f)=s_i\mu(f)\to f(y),$$
for all $f\in C(\widetilde{X})$. Hence, $\Phi_i\circ\varphi_{y_0}^{{\rm L}^2(M)}\to\delta_y$ weak*.
\medskip
\noindent Let \(g\in\Gamma\setminus\{e\}\).
Since
$C(X)\rtimes_r\Gamma$ is simple \cite[Theorem~3.4]{kawabe2017uniformly} implies that
$\Gamma\curvearrowright\widetilde X$ is topologically free.  By
\cite[Proposition~3.3(i)]{kawabe2017uniformly}, $\operatorname{Fix}_{\widetilde
X}(g)$ is clopen.  Hence $\operatorname{Fix}_{\widetilde X}(g)=\operatorname{int}\bigl(\operatorname{Fix}_{\widetilde X}(g)\bigr)=\emptyset$,
so $gy\neq y$. Applying
Lemma \ref{LemApproxBoundaryVanishing} to \(\widetilde{X}\), to the hypertraces \(\Phi_i\), and to the elements $a_i:=E_{N_i}(u_g)\in (N_i)_1$, we obtain $\tau\bigl(E_{N_i}(u_g)u_g^*\bigr)\to 0$. On the other hand,
$$\|E_{N_i}(u_g)\|_2^2=\tau\bigl(E_{N_i}(u_g)^*E_{N_i}(u_g)\bigr)=\tau\bigl(E_{N_i}(u_g)u_g^*\bigr)\to 0.$$
By Lemma \ref{LemExpectationConvergenceFromBoundary}, this implies $N_i\to M$. Hence $M\in \overline{\{u_sNu_s^*:s\in\Gamma\}}$.
\medskip
\noindent $(3)\Rightarrow(1)$. Assume that \(C(X)\rtimes_{ r}\Gamma\) is not simple. By \cite[Theorem 1.7 (iii)]{kawabe2017uniformly}, there is an amenable subgroup \(\Lambda\leq \Gamma_x\) such that $\{e\}\notin\Ucal:=\overline{\{g \Lambda g^{-1}:g\in\Gamma\}}$. Then, $N:=M\rtimes\Lambda$ is amenable (indeed, $M$ is amenable and $\Lambda$ is an amenable group, so the crossed product $M\rtimes\Lambda$ is amenable), contains $M$ and, by
Proposition \ref{PropConfinedSubgroup}, \(N\) is \(M\)-confined. This contradicts (3).\end{proof}
\noindent We say that a URA $\mathcal{U}\subset M\rtimes \Gamma$ \textbf{contains $M$} if every $N\in \mathcal{U}$ contains $M$.
\medskip
\noindent Theorem~\ref{ThmUnified} is stated in terms of $M$-confinement, which
is a condition on a single subalgebra together with its whole orbit closure. We
now reformulate it intrinsically, in terms of the URAs that orbit closure
contains. This is the form in which the criterion becomes a statement about the
Effros--Mar\'echal dynamics of amenable subalgebras.

\begin{corollary}\label{CorURACrossed} Let $\Gamma \curvearrowright (M,\tau_M)$ be a trace-preserving action of a countable discrete group on an
amenable finite von Neumann algebra with separable predual and  $X$ be a non-empty compact Hausdorff minimal $\Gamma$-space. The following are equivalent.
\begin{enumerate}
\item $C(X)\rtimes_r\Gamma$ is simple.
\item For all $x\in X$, the only amenable URA $\Ucal$ of $M\rtimes\Gamma$ containing $M$ and such that $\Ucal\cap\Sub(M\rtimes\Gamma_x)\neq\emptyset$  is $\Ucal=\{M\}$.
\item There exists $x\in X$ such that the only amenable URA $\Ucal$ of $M\rtimes\Gamma$ containing $M$ and such that $\Ucal\cap\Sub(M\rtimes\Gamma_x)\neq\emptyset$  is $\Ucal=\{M\}$.
\end{enumerate}
\end{corollary}
\begin{proof}
Write $P:=M\rtimes\Gamma$. Note that $(2)\Rightarrow(3)$ is obvious.
\medskip
\noindent$(1)\Rightarrow(2)$. Let $x\in X$ and $\mathcal U\subset \operatorname{Sub}(P)$ be an amenable URA containing $M$ and such that there exists $N\in\Ucal$ with $M\subset N\subset M\rtimes\Gamma_x$. By Theorem \ref{ThmUnified}, $M\in \overline{\{u_gNu_g^*:g\in\Gamma\}}$. Since \(\mathcal U\) is closed and \(\Gamma\)-invariant, and since \(N\in\mathcal U\), it follows that $M\in\Ucal$. But \(\{M\}\) is a closed \(\Gamma\)-invariant subset of \(\mathcal U\). Since \(\mathcal U\) is minimal, we get $ \mathcal U=\{M\}$.
\medskip
\noindent$(3)\Rightarrow(1)$. Assume that \(C(X)\rtimes_{\mathrm r}\Gamma\) is not simple. By \cite[Theorem 1.7 (iii)]{kawabe2017uniformly}, there is an amenable subgroup \(\Lambda\leq \Gamma_x\) such that $\{e\}\notin\Ucal:=\overline{\{g \Lambda g^{-1}:g\in\Gamma\}}$. Since $\Sub(\Gamma)$ is compact, there exists $\mathcal H\subset\overline{\{g\Lambda g^{-1}:g\in\Gamma\}}$ a minimal closed \(\Gamma\)-invariant subset. Note that \(\mathcal H\) is an amenable URS (since the set of amenable subgroups is closed in $\Sub(\Gamma)$) and it is non-trivial, since \(\Lambda\) is confined and hence the orbit closure of \(\Lambda\) does not contain \(\{e\}\). By Lemma \ref{LemmaChabauty}, the map $\Theta:\operatorname{Sub}(\Gamma)\to \operatorname{Sub}(P)$, $H\mapsto M\rtimes H$  is continuous and it is obviously \(\Gamma\)-equivariant. Therefore $\mathcal U:=\Theta(\mathcal H)$ is a compact minimal closed \(\Gamma\)-invariant subset of \(\operatorname{Sub}(P)\). Hence, \(\mathcal U\) is an amenable URA  (each $M\rtimes H$, $H\in\mathcal H$, being amenable because $M$ is amenable and $H$ is an amenable group) and contains $M$. It remains to see that \(\mathcal U\) meets \(\operatorname{Sub}(M\rtimes\Gamma_x)\). Let \(H\in\mathcal H\) and let \((g_i)_i\) be a net in \(\Gamma\) such that $g_i\Lambda g_i^{-1}\to H$. Since \(X\) is compact, after passing to a subnet we may and will assume that \(g_i x\to y\) for some \(y\in X\). As \(\Lambda\leq \Gamma_x\), we have $g_i\Lambda g_i^{-1}\leq \Gamma_{g_i x}$. By upper semi-continuity of the stabilizer map (see \cite{glasner2015uniformly}), it follows that $H\leq \Gamma_y$. Since \(X\) is minimal, choose a net \(t_j\in\Gamma\) with $t_jy\to x$. Passing to a subnet, we may and will assume that $t_jHt_j^{-1}\to K$ for some \(K\in\mathcal H\), because \(\mathcal H\) is compact and \(\Gamma\)-invariant. Again, by upper semicontinuity of the stabilizer map, $K\leq \Gamma_x$. Therefore $M\rtimes K\in \mathcal U\cap\operatorname{Sub}(M\rtimes\Gamma_x)$. Finally, \(\mathcal U\neq \{M\}\). Indeed, if \(M\rtimes H=M\) for some \(H\in\mathcal H\), then \(H=\{e\}\), contradicting the fact that \(\mathcal H\) is non-trivial. Hence \(\mathcal U\) is a non-trivial amenable URA containing \(M\) and meeting \(\operatorname{Sub}(M\rtimes\Gamma_x)\).\end{proof}
%%%%%%%%%%%%%%%%%%%%%%%%%%%%%%%%%%%%%%%%%%%%%%%%%%%%%%%%%%%%%%%%%%%%%%%%%%%%%%%%%%%%%%%%%%%%%%%%%%%%%%%%
\section{The state-space construction of URAs}
\label{sec:statespace}
%%%%%%%%%%%%%%%%%%%%%%%%%%%%%%%%%%%%%%%%%%%%%%%%%%%%%%%%%%%%%%%%%%%%%%%%%%%%%%%%%%%%%%%%%%%%%%%%%%%%%%%%
We now turn to the central construction of the paper. As discussed in the introduction, the URS theory of \cite{AJ} carries out a Glasner--Weiss-style dynamics inside the pointwise-compact space $\mathrm{PD}_1(\Gamma)$ of positive definite functions, and recovers URS as $\Gamma$-orbit closures of $\phi_\Lambda(g)=\mathbf{1}_\Lambda(g)$. For von Neumann algebras the analogue of $\phi_\Lambda$ is the positive definite function $\phi_M(g)=\tau(u_g^*E_M(u_g))$ attached to a subalgebra $M\subset L(\Gamma)$, but a fundamental compactness obstruction --- the failure of compactness of $\Sub(M)$, characterized in Theorem~\ref{ThmEmbedding} --- prevents the URS argument from running verbatim in this setting.

To bypass this obstruction, we enlarge the ambient space to states on $\mathbb{B}(\ell^2(\Gamma))$ extending the canonical trace, which form a compact convex $\Gamma$-space $K$. Zorn's lemma then produces minimal closed $\Gamma$-invariant subsets $\mathcal{Z}\subset K$; the centralizer of any $\varphi\in\mathcal{Z}$ relative to a chosen $\Gamma$-invariant test algebra $A$ is a von Neumann subalgebra of the algebra of interest. \iffalse The catch is that the centralizer map $\Phi$ is only upper semi-continuous; the construction of URAs from it relies on a Baire-category argument supplying a dense $G_\delta$ of continuity points.\fi
The catch is that the centralizer map $\Phi$ need not be continuous; whenever it is continuous at a point of a minimal state space, the Effros--Mar\'echal orbit closure of the associated subalgebra is a URA. This construction of a URA from a continuity-point is the content of Theorem~\ref{thm:general-ura-construction}. In the compact and discrete realization theorems (Theorems \ref{ThmCompactURA} and \ref{thm:discrete-uras}), the required continuity is verified directly. %The general construction is the content of Theorem~\ref{thm:general-ura-construction}.
%% [RESOLVED below in teal: Lemma \ref{lem:fix-is-subalgebra} reformulated as a four-part lemma via the full relative centralizer C_phi and the centralizer-density condition (CD) (audit C1). Remaining open: existence of continuity points (dense G_delta).]
\subsection{The general state-space construction}
\label{subsec:generalconstruction} In this entire section, we fix the following data.
\begin{itemize}
    \item $M$ a finite von Neumann algebra acting on a separable Hilbert space $\mathcal{H}$, equipped with a faithful normal tracial state $\tau$;
    \item $g \mapsto u_g$ a unitary representation of a countable discrete group $\Gamma$ on $\mathcal{H}$ such that the conjugation action $\alpha_g(T) = u_g T u_g^*$ satisfies $$\alpha_g(M)=M\text{ and }\tau \circ \alpha_g = \tau\quad\text{for all }g\in\Gamma$$
    \item $A \subset \mathbb{B}(\mathcal{H})$ a $\Gamma$-invariant unital $*$-subalgebra such that $M \subseteq A$.
\end{itemize}
\noindent Let $S(\mathbb{B}(\mathcal{H}))$ denote the state space of $\mathbb{B}(\mathcal{H})$. We define
$$K = \left\{ \varphi \in S(\mathbb{B}(\mathcal{H})) \;\middle|\; \varphi|_{M} = \tau \right\}.$$
\begin{definition}
Let $\varphi\in K$.
\begin{itemize}
    \item \textbf{The centralizer of $\varphi$ relative to $A$} is
$$M_\varphi := \left\{ z \in M \;\middle|\; \varphi(zy) = \varphi(yz) \text{ for all } y \in A \right\}.$$
\item The \textbf{centralizer map relative to $A$} is $\Phi:K\to\Sub(M)$ defined by $$\Phi(\varphi):=M_\varphi.$$
\end{itemize}
\end{definition}
\noindent We collect elementary properties of relative centralizers and of the centralizer map in the next lemma.
\begin{lemma}\label{lem:fix-is-subalgebra}
Let $\varphi\in K$. The following hold.
\begin{enumerate}
\item $M_\varphi$ is a von Neumann subalgebra of $M$.
\item For every $g\in\Gamma$, one has $M_{g\cdot\varphi}=\alpha_g(M_\varphi)$. In particular, the centralizer map is $\Gamma$-equivariant.
\item If $(\varphi_i)_{i\in I}$ is a net in $K$ converging weak-$*$ to $\varphi\in K$ and $M_{\varphi_i}\to\mathcal N$ in the Effros--Mar\'echal topology, then $\mathcal N\subseteq M_\varphi$.
\end{enumerate}
\end{lemma}
\begin{proof}
(1) Clearly $1\in M_\varphi$. If $z\in M_\varphi$ and
$y\in A$, then
$$
\varphi(z^*y)
=\overline{\varphi(y^*z)}
=\overline{\varphi(zy^*)}
=\varphi(yz^*)$$ so $z^*\in M_\varphi$. Let $z,w\in M_\varphi$ and $y\in A$. Since
$w\in M\subseteq A$ and $z\in M\subseteq A$, one has
$wy,yz\in A$. Therefore
$
\varphi(zwy)
=\varphi(wyz)
=\varphi(yzw),
$ and hence $zw\in M_\varphi$. Let $(z_j)_j$ be a bounded net in $M_\varphi$ converging strong-$*$
to $z\in M$. For every $y\in A$, the Cauchy-Schwarz inequality gives
$
|\varphi((z_j-z)y)|^2
\leq
\Vert(z_j-z)^*\Vert_2^2\varphi(y^*y)\longrightarrow0
$
and
$$
|\varphi(y(z_j-z))|^2
\leq
\varphi(yy^*)\Vert z_j-z\Vert_2^2\longrightarrow0.
$$
Consequently, $\varphi(zy)=\varphi(yz)$, so
$z\in M_\varphi$. Thus $M_\varphi$ is a unital
$*$-subalgebra of $M$ whose unit ball is strong-$*$ closed, and
hence it is a von Neumann subalgebra of $M$.
\medskip
\noindent(2) For $z\in M$ and $y\in A$ we have $(g\cdot\varphi)(zy)=\varphi\bigl(\alpha_{g^{-1}}(z)\,\alpha_{g^{-1}}(y)\bigr)$ and $(g\cdot\varphi)(yz)=\varphi\bigl(\alpha_{g^{-1}}(y)\,\alpha_{g^{-1}}(z)\bigr)$. Since $A$ is globally $\alpha$-invariant, $z\in M_{g\cdot\varphi}
\Leftrightarrow
\alpha_{g^{-1}}(z)\in M_\varphi,
$. Hence
$M_{g\cdot\varphi}=\alpha_g(M_\varphi)$.
\medskip
\noindent(3) Let $z\in\mathcal N$ and put $z_i:=E_{M_{\varphi_i}}(z)\in M_{\varphi_i}$. Effros--Mar\'echal convergence gives
$$
\Vert z_i-z\Vert_2=\Vert E_{M_{\varphi_i}}(z)-E_{\mathcal{N}}(z)\Vert_2
\to
0.
$$
The same argument gives $\Vert(z_i-z)^*\Vert_2\longrightarrow0$.
For every $y\in A$, $\varphi_i(z_iy)=\varphi_i(yz_i)$. Moreover, $|\varphi_i(z_iy)-\varphi(zy)|\leq|\varphi_i((z_i-z)y)|+|\varphi_i(zy)-\varphi(zy)|$. The second term tends to $0$ because $\varphi_i\to\varphi$ weak-$*$. For the first term, Cauchy--Schwarz gives
$$|\varphi_i((z_i-z)y)|^2\leq\varphi_i\bigl((z_i-z)(z_i-z)^*\bigr)\,\varphi_i(yy^*)\leq\Vert(z_i-z)^*\Vert_2^2\,\Vert y\Vert^2\longrightarrow0,$$
using $\varphi_i|_{M}=\tau$. Hence
$\varphi_i(z_iy)\to\varphi(zy)$. The symmetric estimate gives
$\varphi_i(yz_i)\to\varphi(yz)$. Therefore
$\varphi(zy)=\varphi(yz)$ for every $y\in A$, and thus
$z\in M_\varphi$. Hence,
$\mathcal N\subseteq M_\varphi$.
\end{proof}
\noindent The next general Lemma will be applied to the relative centralizer map.

\begin{lemma}\label{lem:semicontinuous-coordinate-map}
Let $(M,\tau)$ be a finite von Neumann algebra with separable predual, let
$Z$ be a Baire space, and let $\Phi:Z\to\Sub(M)$ be a map. For $x\in M$, define
$$
f_x:Z\to[0,+\infty)\quad f_x(z):=\Vert E_{\Phi(z)}(x)\Vert_2.
$$
Assume that either every $f_x$, $x\in M$, is upper semicontinuous, or every
$f_x$, $x\in M$, is lower semicontinuous. Then
the set $\operatorname{Cont}(\Phi)$ of points of continuity of $\Phi$ is a dense $G_\delta$ subset of $Z$.
\end{lemma}
\begin{proof}
Choose an $\LL^2$-dense sequence $(x_k)_{k\geq1}$ in $M$. The maps
$$
\Sub(M)\ni N\mapsto\Vert E_N(x_k)\Vert_2,
\quad k\geq1,
$$
generate the Effros--Mar\'echal topology. For each $k$, apply
\cite[Theorem~1]{fort1955category} to $f_{x_k}$ in the lower
semicontinuous case and to $-f_{x_k}$ in the upper semicontinuous case. The
set $\operatorname{Cont}(f_{x_k})$ contains a residual $G_\delta$ subset of
$Z$. Since the continuity set of a real-valued function is a $G_\delta$
subset of its domain and $Z$ is a Baire space,
$\operatorname{Cont}(f_{x_k})$ is a dense $G_\delta$ subset of $Z$. Hence $\operatorname{Cont}(\Phi)=\bigcap_{k\geq1}\operatorname{Cont}(f_{x_k})$
is a dense $G_\delta$ subset of $Z$.
\end{proof}

\noindent For the centralizer map, neither of the semicontinuity hypotheses of the
previous lemma is available directly. What is available is the one-sided
statement of Lemma~\ref{lem:fix-is-subalgebra} $(3)$. The next lemma shows that this
containment, together with relative compactness of the range, already forces
upper semicontinuity of all the coordinate maps, and hence a dense $G_\delta$ of
continuity points.
\medskip

\begin{lemma}
\label{lem:Baire-compact-range}
Let $(M,\tau)$ be a finite von Neumann algebra with separable
predual, $Z$ be a Baire space and $\Phi:Z\to\Sub(M)$ be any map.  If
\begin{enumerate}
\item $\Phi(Z)$ is relatively compact in $\Sub(M)$;
\item whenever a net $(z_i)_i$ in $Z$ satisfies $z_i\to z$ and
$\Phi(z_i)\to\mathcal N$ in the Effros--Mar\'echal topology, one has
$\mathcal N\subseteq\Phi(z)$,
\end{enumerate}
then the set $\operatorname{Cont}(\Phi)$ of points of continuity of $\Phi$ is a dense $G_\delta$
subset of $Z$. \end{lemma}
\begin{proof}
For $x\in M$, define $f_x:Z\to\mathbb R$, $f_x(z):=\Vert E_{\Phi(z)}(x)\Vert_2$.  We first claim that the function $f_x$ is upper
semicontinuous.  Indeed, otherwise there would exist $\varepsilon>0$ and a net $(z_i)_i$ converging to $z$ such that $f_x(z_i)\geq f_x(z)+\varepsilon$ for every $i$.  After passing
to a subnet, relative compactness gives
$\Phi(z_i)\to\mathcal N$ so $f_x(z_i)=\Vert E_{\Phi(z_i)}(x)\Vert_2\to\Vert E_{\mathcal{N}}(x)\Vert_2$.  Hence, $f_x(z)+\varepsilon
\leq\Vert E_{\mathcal{N}}(x)\Vert_2$ and assumption {\rm(2)} yields
$\mathcal N\subseteq\Phi(z)$ so
$$
f_x(z)+\varepsilon
\leq \Vert E_{\mathcal{N}}(x)\Vert_2=\Vert E_{\mathcal{N}}(E_{\Phi(z)}(x))\Vert_2
\leq \Vert E_{\Phi(z)}(x)\Vert_2
=f_x(z),
$$
a contradiction.
The conclusion follows from
Lemma~\ref{lem:semicontinuous-coordinate-map}.
\end{proof}
\noindent The following general Lemma will be useful to construct URA.
\begin{lemma}\label{lem:Minimality from a continuity point}
Let $\Gamma\curvearrowright Z$ be a minimal action on a compact Hausdorff
space, let $\Gamma\curvearrowright Y$ be an action by homeomorphisms on a
regular Hausdorff space, and let $\Phi:Z\to Y$ be $\Gamma$-equivariant.  If
$z_0\in Z$ is a point of continuity of $\Phi$, then $\overline{\Gamma\cdot\Phi(z_0)}$ is a minimal closed $\Gamma$-invariant subset of $Y$. Moreover $\overline{\Gamma\cdot\Phi(z_0)}=\overline{\Phi(\operatorname{Cont}(\Phi))}$.
\end{lemma}
\begin{proof}
Put $H:=\overline{\Gamma\cdot\Phi(z_0)}$, and let $y\in H$.  Choose a net
$(g_i)_i$ in $\Gamma$ such that
\[
\Phi(g_i z_0)=g_i\Phi(z_0)\longrightarrow y.
\]
After passing to a subnet, compactness of $Z$ gives $g_i z_0\to z$ for some
$z\in Z$.  By minimality, there is a net $(h_j)_j$ in $\Gamma$ such that
$h_jz\to z_0$.
Let $W$ be an arbitrary open neighborhood of $\Phi(z_0)$.  By regularity,
choose an open set $U$ such that
\[
\Phi(z_0)\in U\subset\overline U\subset W.
\]
Continuity at $z_0$ gives an open neighborhood $V$ of $z_0$ such that
$\Phi(V)\subset U$.  Choose $j$ such that $h_jz\in V$.  Then, for all
sufficiently large $i$, one has $h_jg_i z_0\in V$, and therefore $h_jg_i\Phi(z_0)=\Phi(h_jg_i z_0)\in U$. Passing to the limit in $i$ gives $h_jy\in\overline U\subset W$.  Thus every
neighborhood of $\Phi(z_0)$ meets $\Gamma\cdot y$, so $\Phi(z_0)\in\overline{\Gamma\cdot y}$. Since $H=\overline{\Gamma\cdot\Phi(z_0)}$, it follows that
$\overline{\Gamma\cdot y}=H$.  Hence $H$ is minimal.
\medskip
\noindent Continuity of both $\Gamma$-actions and equivariance of $\Phi$
imply that $Z_0:=\operatorname{Cont}(\Phi)$ is $\Gamma$-invariant.
Consequently, $\Gamma\cdot\Phi(z_0)=\Phi(\Gamma\cdot z_0)\subseteq\Phi(Z_0)$
hence $H\subseteq\overline{\Phi(Z_0)}$. Conversely, let $z\in Z_0$.
Minimality of $Z$ gives a net $(g_i)_i$ in $\Gamma$ such that
$g_iz_0\to z$. Continuity of $\Phi$ at $z$ and equivariance give
$g_i\Phi(z_0)=\Phi(g_iz_0)\longrightarrow\Phi(z)$.
Thus $\Phi(Z_0)\subseteq H$.
\end{proof}
\noindent To apply the previous Lemma and construct URA we need minimality. This is obtained as follows: the state space $K$ is weak-$*$ compact, so Zorn's lemma applies. The next lemma harvests its consequence.
\begin{lemma}\label{lem:hypertrace-space}The set $K$ is a compact, convex $\Gamma$-space. Furthermore, there exists a minimal, non-empty, closed, $\Gamma$-invariant subset $\mathcal{Z} \subset K$.\end{lemma}%% [TEAL REPLACEMENT L6.1-proof] The original passage is preserved verbatim inside the \iffalse block; the corrected version follows in teal.
\begin{proof}
The state $\tau$ on the unital $C^*$-subalgebra $M\subseteq
\mathbb B(\mathcal H)$ extends to a state on $\mathbb B(\mathcal H)$.
Consequently, $K\neq\emptyset$.  The state space
$S(\mathbb B(\mathcal H))$ is weak-$*$ compact, and
$$
K=\bigcap_{x\in M}
\{\varphi\in S(\mathbb B(\mathcal H)):\varphi(x)=\tau(x)\}
$$
is weak-$*$ closed.  Hence $K$ is compact.  The defining equations of $K$
are affine, so $K$ is convex. For $g\in\Gamma$, $\varphi\in S(\mathbb B(\mathcal H))$, and
$T\in\mathbb B(\mathcal H)$, set $(g\cdot\varphi)(T):=\varphi(u_g^*Tu_g)$. This defines a weak-$*$ continuous affine action.  If $\varphi\in K$ and
$x\in M$, then
$$
(g\cdot\varphi)(x)
=\varphi(\alpha_{g^{-1}}(x))
=\tau(\alpha_{g^{-1}}(x))
=\tau(x).
$$
Thus $K$ is $\Gamma$-invariant.
\medskip
\noindent Let $\mathcal F$ be the family of all non-empty weak-$*$ closed
$\Gamma$-invariant subsets of $K$, ordered by reverse inclusion.  The family
$\mathcal F$ is non-empty because $K\in\mathcal F$.  Let
$(F_i)_{i\in I}$ be a chain in $\mathcal F$.  Every finite subfamily of
$(F_i)_{i\in I}$ has a non-empty intersection, since the family is totally
ordered by inclusion.  Compactness of $K$ therefore gives $F:=\bigcap_{i\in I}F_i\neq\emptyset$. The set $F$ is weak-$*$ closed and $\Gamma$-invariant, so
$F\in\mathcal F$.  Moreover, $F$ is an upper bound of the chain for the
reverse-inclusion order.  Zorn's lemma yields a maximal element
$\mathcal Z\in\mathcal F$ for that order.\end{proof}
\noindent We can finally state our main result on the state space construction.
\begin{theorem}\label{thm:general-ura-construction}
Let $\mathcal{Z}\subset K$ be any minimal, non-empty, closed, $\Gamma$-invariant subset and $\Phi\vert_{\mathcal{Z}}:\mathcal{Z}\to\Sub(M)$ be the restriction of the centralizer map relative to $A$ and $\operatorname{Cont}(\Phi\vert_{\mathcal{Z}})\subseteq\mathcal{Z}$ be the set of continuity points of $\Phi\vert_{\mathcal{Z}}$. The following holds:
\begin{enumerate}
    \item If
$\varphi_0\in\operatorname{Cont}(\Phi\vert_{\mathcal{Z}})$ then $\mathcal{H} = \overline{ \left\{ \alpha_g(M_{\varphi_0}) \;\middle|\; g \in \Gamma \right\} }^{\,\mathrm{EM}}$ is a URA.
\item If $\Phi(\mathcal{Z})$ is relatively compact in $\Sub(M)$ then $\operatorname{Cont}(\Phi\vert_{\mathcal{Z}})$ is $G_\delta$-dense in $\mathcal{Z}$.
\end{enumerate}
\end{theorem}
\begin{proof}
$(1)$ is a direct consequence of Lemmas \ref{lem:fix-is-subalgebra} and \ref{lem:Minimality from a continuity point} and $(2)$ follows from Lemmas \ref{lem:fix-is-subalgebra} and \ref{lem:Baire-compact-range}.
\end{proof}

\noindent In the Remark below we discuss the degenerated case $A=\B(H)$.
\begin{remark}\label{rmk:maximal-test-algebra-amenable-ura}
Assume that $A=\mathbb B(\mathcal H)$. The following holds.
\begin{itemize}
\item For every $\varphi\in K$, $M_\varphi$ is amenable.
\end{itemize}
Indeed, this is a direct consequence of Proposition~\ref{prop:amenability-hypertrace}
\begin{itemize}
\item Every URA obtained from Theorem~\ref{thm:general-ura-construction} is amenable.
\end{itemize}
It follows from the first point and Remark \ref{rmk-amenbale-URA}.
\begin{itemize}
\item If $M=\LL(\Gamma)\subseteq A=\mathbb B(\ell^2(\Gamma))$ and $\Gamma$ is
$C^*$-simple then every URA obtained from Theorem~\ref{thm:general-ura-construction} is trivial.
\end{itemize}
If $\Ucal$ is such an URA then it is amenable by the previous point and trivial by Corollary \ref{CorURACrossed}.
\end{remark}
\noindent Part~{\rm(2)} of
Theorem~\ref{thm:general-ura-construction} cannot be strengthened to
continuity on all of $\mathcal Z$: the centralizer range may be relatively
compact while the centralizer map is discontinuous as shown by the following example.
\begin{example}
\label{ex:minimal-discontinuous-centralizer-compact-range}
Fix $\theta\in\mathbb R\setminus\mathbb Q$, let $X:=\mathbb R/\mathbb Z$, $\Gamma:=D_\infty=\langle a,b\mid b^2=e,\ bab=a^{-1}\rangle$ and define the action $\Gamma\curvearrowright X$ by 
$a\cdot x:=x+\theta$ and $b\cdot x:=-x$. The action $\Gamma\curvearrowright X$ is minimal because the rotation by
$\theta$ is minimal. Put $M:=\LL(\Gamma)$ and $C:=C(X)\rtimes_r\Gamma$ and form the full amalgamated free product
$\mathcal A:=M*_{C_r^*(\Gamma)}C$. For $x\in X$, let $\rho_x:\mathcal A\to\mathbb B(\ell^2(\Gamma))$ be the
representation whose restriction to $M$ is the standard representation and
whose restriction to $C$ is determined by
$$
\rho_x(f)\delta_g=f(gx)\delta_g,
\quad
\rho_x(\lambda_s)\delta_g=\delta_{sg}\quad\text{for all }f\in C(X), g,s\in\Gamma.
$$
These two representations agree on $C_r^*(\Gamma)$, so $\rho_x$ is well
defined. Put $x_0:=0+\Z\in\mathbb R/\mathbb Z$, $D:=\Gamma x_0$ and let
$H:=\bigoplus_{x\in D}\ell^2(\Gamma)$, $\rho:=\bigoplus_{x\in D}\rho_x$ and $A:=\rho(\mathcal A)\subseteq\mathbb B(H)$. The Hilbert space $H$ is separable. Since every $\rho_x|_M$ is
the standard faithful representation, $\rho|_M$ is faithful; we identify
$M$ with its image in $A$.
\medskip
\noindent Let $V$ be the set of all $c\in\mathcal A$ such that, for every
$\xi\in\ell^2(\Gamma)$, both maps
$$
X\ni x\longmapsto\rho_x(c)\xi\in\ell^2(\Gamma)
\quad
X\ni x\longmapsto\rho_x(c^*)\xi\in\ell^2(\Gamma)
$$
are norm-continuous. The set $V$ is a norm-closed involutive subalgebra of $\mathcal A$. Indeed,
linearity and stability under involution are immediate. If $c,d\in V$,
$x_0\in X$ and $\xi\in\ell^2(\Gamma)$, then
$$
\|\rho_x(cd)\xi-\rho_{x_0}(cd)\xi\|
\leq \|c\|\,\|\rho_x(d)\xi-\rho_{x_0}(d)\xi\|+\|(\rho_x(c)-\rho_{x_0}(c))\rho_{x_0}(d)\xi\|,
$$
and both terms tend to $0$ as $x\to x_0$. Applying the same argument to
$(cd)^*=d^*c^*$ gives $cd\in V$. If $c_n\in V$ and $\|c_n-c\|\to0$,
then, for every $x\in X$ and $\xi\in\ell^2(\Gamma)$,
$$
\|\rho_x(c_n-c)\xi\|\leq\|c_n-c\|\,\|\xi\|,
$$
and the analogous estimate holds for $c_n^*-c^*$; hence $c\in V$.
\medskip
\noindent Note that $M\cup C^*_r(\Gamma)\subset V$ and the canonical unitaries $\lambda_s$, $s\in\Gamma$, are
contained in $V$, since their images under $\rho_x$ are independent of $x$. For all $f\in C(X)$ and $\xi\in \ell^2(\Gamma)$
$$
\|\rho_x(f)\xi-\rho_{x_0}(f)\xi\|^2
=\sum_{g\in\Gamma}
|f(gx)-f(gx_0)|^2|\xi(g)|^2
\underset{x\to x_0}{\longrightarrow}0,
$$
by continuity of $f$ and the dominated convergence Theorem. The same argument applies to $f^*$, so $C(X)\subseteq V$.
Consequently, $V$ contains $C=C(X)\rtimes_r\Gamma$ and $M$; since
$\mathcal A=M*_{C_r^*(\Gamma)}C$ is generated by $M$ and $C$, one has
$V=\mathcal A$.
\medskip
\noindent We claim that $\ker\rho\subseteq\ker\rho_x$ for all $x\in X$. Indeed, let $c\in\ker\rho$. Then,
$\rho_d(c)=0$ for every $d\in D$. For $\xi,\eta\in\ell^2(\Gamma)$, the
function $X\ni x\longmapsto\langle\rho_x(c)\xi,\eta\rangle$ is continuous because $c\in V$, and it vanishes on the dense subset $D$.
It therefore vanishes on $X$. Thus $\rho_x(c)=0$ for every $x\in X$, which
proves the claim.
\medskip
\noindent Conjugation by the unitary $u_s:=\bigoplus_{x\in D}\lambda_s$ defines a trace-preserving action of $\Gamma$ on $M$ and leaves $A$
invariant. For $x\in X$, define a state $\omega_x\in S(A)$ by
$$\omega_x(\rho(c)):=\langle\rho_x(c)\delta_e,\delta_e\rangle\quad \text{for all }c\in\mathcal A.$$
This is well defined because $\ker\rho\subseteq\ker\rho_x$. The map
$X\to S(A)$, $x\mapsto\omega_x$ is weak*-continuous (since $V=\mathcal{A}$) and injective because $\omega_x(f)=f(x)$ for all $f\in C(X)$ and the continuous functions separate the points of $X$. Moreover,
$\omega_x|_M=\tau$ and $x\mapsto\omega_x$ is $\Gamma$-equivariant.
\medskip
\noindent Let $r:S(\mathbb B(H))\to S(A)$ be the weak*-weak*-continuous restriction map and put $Y:=r^{-1}(\{\omega_x:x\in X\})$. The Hahn Banach state-extension theorem implies that $Y$ is non-empty. Since $x\mapsto\omega_x$ is continuous and $X$ is compact, $\{\omega_x:x\in X\}$ is compact hence closed. Since $r$ is continuous, $Y$ is closed hence weak*-compact subset of $S(\mathbb B(H))$. Note that $r$ is equivariant since $A$ is $\Gamma$-invariant hence, because $x\mapsto\omega_x$ is equivariant, $Y$ is $\Gamma$-invariant. So $Y$ is a 
compact $\Gamma$-invariant subset of the state space $K:=\{\varphi\in S(\mathbb B(H)):\varphi|_M=\tau\}$. Since $X$ is compact and $S(A)$ is Hausdorff, the continuous injective map $\iota:X\to S(A)$, $x\mapsto\omega_x$ is a homeomorphism onto its image $\Omega$. Hence $q:=\iota^{-1}\circ r|_Y:Y\to X$ is the unique continuous equivariant map satisfying $\varphi|_A=\omega_{q(\varphi)}$ for all $\varphi\in Y$. Choose a minimal non-empty closed $\Gamma$-invariant subset
$\mathcal Z\subseteq Y$. Since $q(\mathcal Z)$ is a non-empty closed
invariant subset of the minimal space $X$, one has $q(\mathcal Z)=X$.
\medskip
\noindent We compute the centralizer relative to $A$. For every $x\in X$,
the inclusion $\ker\rho\subseteq\ker\rho_x$ yields a representation $\widetilde\rho_x:A\to\mathbb B(\ell^2(\Gamma))$, $\widetilde\rho_x(\rho(c)):=\rho_x(c)$ for all $c\in\mathcal A$ which satisfies $\omega_x(y)=\langle\widetilde\rho_x(y)\delta_e,\delta_e\rangle$ for all $y\in A$. Let $\varphi\in Y$ and put $x:=q(\varphi)$, so that
$\varphi|_A=\omega_x$. Let $s\in\Gamma_x$ and define, for $g,s\in\Gamma$, $R_s\delta_g:=\delta_{gs^{-1}}$ so that $R_s$ commutes with the standard left representation of $M$. For $f\in C(X)$ and $g\in\Gamma$, the equality
$s^{-1}x=x$ gives
$$
R_s\widetilde\rho_x(f)\delta_g
=f(gx)\delta_{gs^{-1}}
=f(gs^{-1}x)\delta_{gs^{-1}}
=\widetilde\rho_x(f)R_s\delta_g.
$$
Since $A$ is generated by the images of $M$ and $C(X)$, it follows that $R_s\in\widetilde\rho_x(A)'$. Furthermore, for every $y\in A$,
\begin{align*}
\omega_x(\lambda_sy)
&=\langle\lambda_s\widetilde\rho_x(y)\delta_e,\delta_e\rangle=\langle\widetilde\rho_x(y)\delta_e,\lambda_s^*\delta_e\rangle=\langle\widetilde\rho_x(y)\delta_e,R_s\delta_e\rangle=\langle R_s^*\widetilde\rho_x(y)\delta_e,\delta_e\rangle\\
&=\langle\widetilde\rho_x(y)R_s^*\delta_e,\delta_e\rangle=\langle\widetilde\rho_x(y)\lambda_s\delta_e,\delta_e\rangle=\omega_x(y\lambda_s).
\end{align*}
Since $\varphi|_A=\omega_x$, this proves $\lambda_s\in M_\varphi$. By Lemma~\ref{lem:fix-is-subalgebra} $(1)$ we deduce that $L(\Gamma_x)\subseteq M_\varphi$. Conversely, let $z\in M_\varphi$ and $s\in\Gamma\setminus\Gamma_x$. Since
$sx\neq x$ and $X$ is compact Hausdorff, there exists $f\in C(X)$ such that $f(x)=1$ and $f(sx)=0$. Set $y:=\lambda_s^*f\in A$. Since $z,\lambda_s,f\in A$ and
$\varphi|_A=\omega_x$, one has
$$
\varphi(zy)=\omega_x(z\lambda_s^*f)=\langle\widetilde\rho_x(z)\lambda_s^*
       \widetilde\rho_x(f)\delta_e,\delta_e\rangle=f(x)\langle\widetilde\rho_x(z)\lambda_s^*\delta_e,\delta_e\rangle=f(x)\tau(z\lambda_s^*).$$
Similarly,
\begin{align*}
\varphi(yz)
&=\omega_x(\lambda_s^*fz)
=\langle\lambda_s^*\widetilde\rho_x(f)
       \widetilde\rho_x(z)\delta_e,\delta_e\rangle
=\langle\widetilde\rho_x(f)
       \widetilde\rho_x(z)\delta_e,\lambda_s\delta_e\rangle
=f(sx)\langle\widetilde\rho_x(z)\delta_e,\delta_s\rangle\\
&=f(sx)\tau(z\lambda_s^*).
\end{align*}
The equality $\varphi(zy)=\varphi(yz)$ and the identities $f(x)=1$ and
$f(sx)=0$ imply
$$
\tau(z\lambda_s^*)=0
\quad\text{for all }s\in\Gamma\setminus\Gamma_x.
$$
Hence $z\in \LL(\Gamma_x)$. This shows that $M_\varphi=\LL(\Gamma_{q(\varphi)})$ for all $\varphi\in Y$.
\medskip
\noindent Every element of $\Gamma$ is either a rotation $a^n$ or a reflection $a^nb$ for a unique
$n\in\mathbb Z$. If $n\neq0$, then $a^n$ acts by the irrational rotation
$x\mapsto x+n\theta$ and has no fixed point. Moreover, $a^nb\cdot x=-x+n\theta$, so $
\operatorname{Fix}_X(a^nb)
=
\left\{\frac{n\theta}{2},\frac{n\theta}{2}+\frac12\right\}$ for all $n\in\mathbb Z$. If $n\neq m$, then $(a^nb)(a^mb)=a^{n-m}$ thus no point is fixed by two distinct reflections. Consequently, every
stabilizer is either trivial or generated by a unique reflection. Since
$\Gamma$ is countable and every non-trivial fixed-point set is finite,
$$
X_{\rm free}
:=\{x\in X:\Gamma_x=\{e\}\}
=X\setminus\bigcup_{g\in\Gamma\setminus\{e\}}\operatorname{Fix}_X(g)
$$
is an invariant dense $G_\delta$ subset of $X$. Moreover,
$\Gamma_{x_0}=\langle b\rangle$. Since $q(\mathcal Z)=X$, choose $\varphi_0\in\mathcal Z$ such that
$q(\varphi_0)=x_0$. The set $\mathcal U:=q^{-1}(X_{\rm free})\cap\mathcal Z$ is non-empty because $q(\mathcal Z)=X$, and it is $\Gamma$-invariant
because $q$ is equivariant and $X_{\rm free}$ is $\Gamma$-invariant. Its
closure in $\mathcal Z$ is therefore a non-empty closed $\Gamma$-invariant subset
of the minimal space $\mathcal Z$, so $\overline{\mathcal U}=\mathcal Z$. There exists a net $(\varphi_i)_i$ in $\mathcal U$ converging to
$\varphi_0$. The centralizer formula gives $M_{\varphi_i}=\mathbb C1$ for all $i$ and $M_{\varphi_0}=\LL(\langle b\rangle)\neq\mathbb C1$ so $\Phi|_{\mathcal Z}$ is not continuous at $\varphi_0$.
\medskip
\noindent Finally, put
$\mathcal C
:=\{\{e\}\}\cup
\{\langle a^nb\rangle:n\in\mathbb Z\}
\subseteq\Sub(\Gamma)$. Let $\widehat{\mathbb Z}:=\mathbb Z\cup\{\infty\}$ be the one-point
compactification of the discrete space $\mathbb Z$. Define
$$
\kappa:\widehat{\mathbb Z}\to\Sub(\Gamma),
\quad
\kappa(n):=\langle a^nb\rangle,
\quad
\kappa(\infty):=\{e\}.
$$
For every $g\in\Gamma\setminus\{e\}$, the element $g$ belongs to at most one
subgroup $\langle a^nb\rangle$. Hence, for every $g\in\Gamma$, the coordinate map
$\widehat{\mathbb Z}\ni n\longmapsto
1_{\kappa(n)}(g)$ is continuous at $\infty$, and it is continuous at every
$n\in\mathbb Z$ because $\mathbb Z$ is discrete. Hence, the map $\kappa$ is continuous and
$\mathcal C=\kappa(\widehat{\mathbb Z})$ is compact. By
Lemma~\ref{LemmaChabauty}, $
\{\mathbb C1\}\cup
\{\LL(\langle a^nb\rangle):n\in\mathbb Z\}
$ is compact in $\Sub(M)$. Since
$$
\Phi(\mathcal Z)
\subseteq
\{\mathbb C1\}\cup
\{\LL(\langle a^nb\rangle):n\in\mathbb Z\},
$$
the range $\Phi(\mathcal Z)$ is relatively compact. Thus all the
hypotheses of part~{\rm(2)} of
Theorem~\ref{thm:general-ura-construction} hold, whereas
$\Phi|_{\mathcal Z}$ is not continuous everywhere.
\end{example}
\noindent We can show continuity of the centralizer map under stronger assumptions. We will need the following  general Lemma that we will apply to the centralizer map.
\begin{lemma}
\label{lem:minimal-equicontinuous-propagation}
Let $\Gamma\curvearrowright Z$ be a minimal action on a topological space,
let $\Gamma\curvearrowright(Y,d)$ be an equicontinuous action on a metric
space, and let $\Psi:Z\to Y$ be $\Gamma$-equivariant. If $\Psi$ has a point
of continuity, then $\Psi$ is continuous.
\end{lemma}
\begin{proof}
Fix $z_0\in\operatorname{Cont}(\Psi)$, $z\in Z$, and $\varepsilon>0$.
By equicontinuity at $\Psi(z_0)$, there exists an open neighborhood $W$ of
$\Psi(z_0)$ such that $d(hy,h\Psi(z_0))<\varepsilon/2$ for all $h\in\Gamma$ and $y\in W$. By continuity at $z_0$, there exists an
open neighborhood $U$ of $z_0$ such that $\Psi(U)\subseteq W$.
By minimality, there exists $g\in\Gamma$ such that $gz\in U$. Put
$V:=g^{-1}U$. For every $w\in V$, one has $gw,gz\in U$, and hence
equivariance gives
$$
\begin{aligned}
d(\Psi(w),\Psi(z))
&=d(g^{-1}\Psi(gw),g^{-1}\Psi(gz))\\
&\leq d(g^{-1}\Psi(gw),g^{-1}\Psi(z_0))
+d(g^{-1}\Psi(z_0),g^{-1}\Psi(gz))<\varepsilon.
\end{aligned}
$$
Thus $\Psi$ is continuous at $z$.
\end{proof}
\begin{proposition}
\label{prop:equicontinuous-centralizer-continuity}
Let $\mathcal Z\subset K$ be a minimal non-empty compact
$\Gamma$-invariant subset of the state space $K$ and $
\Phi\vert_{\mathcal{Z}}:\mathcal Z\to\Sub(M)$ the centralizer map relative to $A$. Assume that
$\mathcal C:=\overline{\Phi(\mathcal Z)}^{\,\mathrm{EM}}$
is compact and that the action $\Gamma\curvearrowright\mathcal C$ is
equicontinuous. Then $\Phi\vert_{\mathcal{Z}}$ is continuous.
\end{proposition}
\begin{proof}
By Theorem~\ref{thm:general-ura-construction} $\rm(2)$,
$\Phi|_{\mathcal Z}$ has a point of continuity. Apply
Lemma~\ref{lem:minimal-equicontinuous-propagation}.
\end{proof}
\subsection{Compact URAs}
\label{subsec:compactURAs}
We now show that every compact URA is realized by the minimal
state-space construction of Theorem~\ref{thm:general-ura-construction}.
Lemma~\ref{LemJones} supplies the basic Jones identity, the diagonal-field
Lemma~\ref{lem:diagonal-jones-centralizer} packages the centralizer
calculation used in both the compact and discrete cases, and a
Stone--\v{C}ech argument produces the continuous parametrization required
in the compact case.
\begin{lemma}\label{LemJones}
Let \((M,\tau)\) be a finite von Neumann algebra, represented on
\(\LL^2(M,\tau)\), and let \(N\in\operatorname{Sub}(M)\). Denote by $e_N:\LL^2(M,\tau)\to \LL^2(N,\tau)$ the Jones projection, so that $e_N(a\xi_\tau)=E_N(a)\xi_\tau$ for all $a\in {M}$. Let $\omega\in \B(\LL^2(M,\tau))_*$,
\[
        \omega(T):=\langle T\xi_\tau,\xi_\tau\rangle,
        \qquad T\in \B(\LL^2(M,\tau)).
\]
Let \(\mathcal A\) be the $C^*$-algebra generated by \(M\) and
\(e_N\) inside \(\B(\LL^2(M,\tau))\). Then, for every \(x\in N\) and every
\(y\in\mathcal A\), $\omega(xy)=\omega(yx)$.\end{lemma}
\begin{proof}
Let \(\mathcal A_0\) be the \(^*\)-algebra generated by \(M\) and
\(e_N\) inside \(\B(\LL^2(M,\tau))\). By continuity and density, it is enough to show the relation for $y\in\mathcal{A}_0$. We first recall the basic Jones relation $e_Nae_N=E_N(a)e_N$ for all $a\in M$. It follows that $\mathcal{A}_0=M+Me_NM$. It is therefore enough to verify the desired identity for \(y\in M\) and for
\(y=ae_Nb\), with \(a,b\in M\). First, if \(y=a\in M\), then
\[
        \omega(xa)
        =
        \tau(xa)
        =
        \tau(ax)
        =
        \omega(ax),
\]
by traciality of \(\tau\). Now let \(y=ae_Nb\), with \(a,b\in M\). Then, since $x\in N$,
\begin{eqnarray*}
        \omega(xae_Nb)
        &=&
        \langle xaE_N(b)\xi_\tau,\xi_\tau\rangle
        =
        \tau(xaE_N(b))= \tau(aE_N(b)x)=\tau(aE_N(bx))\\
        &=&\langle aE_N(bx)\xi_\tau,\xi_\tau\rangle= \omega(ae_Nbx)
\end{eqnarray*}
By linearity, the equality $\omega(xy)=\omega(yx)$ holds for every \(y\in\mathcal{A}_0\).\end{proof}

\noindent Lemma~\ref{LemJones} says that one Jones projection suffices to make
one prescribed subalgebra central for the vector state $\omega$. To realize an
entire URA we let the subalgebra vary along an index set and assemble the
corresponding Jones projections into a single diagonal operator, which then plays
the role of the test algebra $A$. The next lemma records, once and for all, the
resulting centralizer computation together with the equivariance and isolation
properties. Both the compact realization theorem and the discrete one are read
off from it, the difference between them lying only in the choice of the index
set and in whether the associated state space is minimal.
\medskip

{
\begin{lemma}\label{lem:diagonal-jones-centralizer}
Let $(M,\tau)$ be a finite von Neumann algebra, let $I$ be a set, and let
$(N_i)_{i\in I}$ be a family in $\Sub(M)$. Put $H:=\ell^2(I,\LL^2(M,\tau))$, represent $M$ diagonally on $H$, and define $(T\xi)(i):=e_{N_i}\xi(i)$ for all $\xi\in H,\ i\in I$. Let $A_I:=C^*(M,T)\subseteq\B(H)$, let
$\eta_i:=\delta_i\otimes\xi_\tau$, and let $\omega_i(S):=\langle S\eta_i,\eta_i\rangle$ for all $S\in{\B(H)}$. Then
$$
\omega_i(aTb)=\tau(aE_{N_i}(b))
\quad\text{for all }a,b\in M,\ i\in I.$$
Suppose that $(i_j)_j$ is a net in $I$ such that
$$
N_{i_j}\to N
\quad\text{in }\Sub(M)
\quad\text{and}\quad
\omega_{i_j}\to\varphi
\quad\text{weak-$*$ in }S(\B(H)).
$$
Then $\varphi|_M=\tau$, and the centralizer of $\varphi$ relative to $A_I$ is $M_\varphi=N$. In particular, $M_{\omega_i}=N_i$ for every $i\in I$.
\medskip
\noindent Suppose, in addition, that a group $\Gamma$ acts on $I$, that
$\alpha:\Gamma\curvearrowright(M,\tau)$ is trace preserving, and that
$N_{gi}=\alpha_g(N_i)$ for all $g\in\Gamma,\ i\in I$. Let $u_\alpha$ be the Koopman representation on $\LL^2(M,\tau)$, and
define $(u(g)\xi)(i):=u_\alpha(g)\xi(g^{-1}i)$. Then $u$ is a unitary representation on $H$ satisfying
$u(g)xu(g)^*=\alpha_g(x)$ and $u(g)Tu(g)^*=T$. In particular, $A_I$ is $\Gamma$-invariant and, for the induced action on
states, $g\cdot\omega_i=\omega_{gi}$.
Moreover, every $\omega_i$ is isolated in $Y_I:=\overline{\{\omega_j:j\in I\}}^{\,w^*}
\subseteq S(\B(H))$.
\end{lemma}
\begin{proof}
The first identity follows from $\omega_i(aTb)
=\langle ae_{N_i}b\xi_\tau,\xi_\tau\rangle
=\tau(aE_{N_i}(b))$. Effros--Mar\'echal convergence gives
$E_{N_{i_j}}(b)\to E_N(b)$ in $\LL^2(M,\tau)$ for every $b\in M$.
Consequently,
$\varphi(aTb)=\tau(aE_N(b))$ for all $a,b\in M$. The equality $\varphi|_M=\tau$ follows from
$\omega_i|_M=\tau$ for every $i\in I$.
Let $x\in M_\varphi$. Testing the centralizer identity against $aTb$ gives
$\tau(aE_N(b)x)=\tau(aE_N(bx))$ for all $a,b\in M$. The faithfulness of $\tau$ yields, for all $b\in M$, $E_N(b)x=E_N(bx)$. Taking $b=1$ gives $x=E_N(x)$, and hence $M_\varphi\subseteq N$. Conversely, let $x\in N$, and let $\mathcal A_{I}\subset A_I$ be the unital
$*$-algebra generated by $M$ and $T$. Fix $y\in\mathcal A_{I}$, choose
a noncommutative polynomial expression for $y$ in $M$ and $T$, and, for
$P\in\Sub(M)$, let $y_P\in\B(\LL^2(M,\tau))$ be the operator obtained
from this expression by replacing $T$ with $e_P$.
The convergence $N_{i_j}\to N$ implies $e_{N_{i_j}}\longrightarrow e_N$ strongly on $\LL^2(M,\tau)$: this holds on the dense subspace
$M\xi_\tau$ by the defining convergence of the conditional expectations,
and the projections are uniformly bounded (see the proof of \cite[Theorem 4.3]{fima2026spacesucpmapssubalgebras}). Therefore
$y_{N_{i_j}}\to y_N$ strongly for every $y\in\mathcal A_{I}$. It follows
that
$$
\varphi(xy)
=\lim_j\langle x y_{N_{i_j}}\xi_\tau,\xi_\tau\rangle
=\langle x y_N\xi_\tau,\xi_\tau\rangle\text{ and }
\varphi(yx)
=\lim_j\langle y_{N_{i_j}}x\xi_\tau,\xi_\tau\rangle
=\langle y_Nx\xi_\tau,\xi_\tau\rangle.
$$
Lemma~\ref{LemJones}, applied to $N$, shows that the last two expressions
are equal. By norm density of
$\mathcal A_{I}$ in $A_I$, one obtains $x\in M_\varphi$. Hence, $N\subseteq M_\varphi$.
\medskip
\noindent Under the additional equivariance assumptions, it is clear that the formula defining $u$
gives a unitary representation and $
u(g)xu(g)^*=\alpha_g(x)$ for all $x\in M, g\in\Gamma$. Furthermore,
$u_\alpha(g)e_{N_{g^{-1}i}}u_\alpha(g)^*
=e_{\alpha_g(N_{g^{-1}i})}
=e_{N_i}$ so $u(g)Tu(g)^*=T$. Finally,
$u(g)\eta_i=\eta_{gi}$, and therefore
$g\cdot\omega_i=\omega_{gi}$. For $i\in I$, let $p_i\in\B(H)$ be the orthogonal projection onto the
$i$-th summand. Then
$$
\omega_j(p_i)=
\begin{cases}
1,&j=i,\\
0,&j\neq i.
\end{cases}
$$
Hence the weak-$*$ open set $V_i:=\{\psi\in S(\B(H)):\psi(p_i)>1/2\}$
satisfies
$$V_i\cap\{\omega_j:j\in I\}=\{\omega_i\}.$$
If
$\psi\in V_i\cap Y_I$, a net in $\{\omega_j:j\in I\}$ converging to
$\psi$ is eventually in $V_i$, hence eventually equal to $\omega_i$.
Thus $V_i\cap Y_I=\{\omega_i\}$.
\end{proof}
\noindent The next theorem says precisely that the general minimal state space construction from Theorem \ref{thm:general-ura-construction} captures every compact URA.} The construction is explicit: given a compact URA $\mathcal{U}$, we represent $M$ diagonally on $\ell^2(\Gamma/\Sigma,{\rm L}^2(M))$ where $\Sigma$ is the stabilizer of a basepoint $N_0\in\mathcal{U}$, and use a coset-indexed family of Jones projections to encode the URA itself.
\begin{theorem}\label{ThmCompactURA}
{
Let $(M,\tau)$ be a finite von Neumann algebra with separable predual, let
$\Gamma$ be a countable discrete group, and let
$\alpha:\Gamma\curvearrowright(M,\tau)$ be a trace-preserving action. Let
$\mathcal U\subseteq\Sub(M)$ be a compact URA. Then there exist
\begin{itemize}
\item a unitary representation
$u:\Gamma\to\Ucal(H)$ on a separable Hilbert space $H$;
\item a faithful $\Gamma$-equivariant representation
$M\subseteq\B(H)$;
\item  a $\Gamma$-invariant unital $C^*$-subalgebra
$A\subseteq\B(H)$ containing $M$,
\end{itemize}
with the following property: there is a
compact $\Gamma$-invariant subset
$$
Y\subseteq K:=\{\varphi\in S(\B(H)):\varphi|_M=\tau\}
$$
such that the centralizer map relative to $A$
$$
\Phi:Y\to\Sub(M),
\quad
\Phi(\varphi)
:=
\{z\in M:\varphi(zy)=\varphi(yz)\text{ for every }y\in A\},
$$
is continuous and satisfies $\Phi(Y)=\mathcal U$. Moreover, for every
minimal non-empty closed $\Gamma$-invariant subset
$\mathcal Z\subseteq Y$, one has
$$\Phi(\mathcal Z)=\mathcal U\quad\text{and}\quad\overline{\Gamma\cdot M_{\varphi_0}}^{\,\mathrm{EM}}=\mathcal U\text{ for all }\varphi_0\in\mathcal Z.$$
}
\end{theorem}
\begin{proof}
{
Fix $N_0\in\mathcal U$, put $\Sigma:=\operatorname{Stab}_\Gamma(N_0)$, $I:=\Gamma/\Sigma$, $H:=\ell^2(I,\LL^2(M,\tau))$ and represent $M$ diagonally on $H$. Since $I$ is countable and $\LL^2(M,\tau)$ is
separable, $H$ is separable. For $g\Sigma\in I$, put $N_{g\Sigma}:=\alpha_g(N_0)$. This is well defined by the definition of $\Sigma$. Define $(T\xi)(g\Sigma):=e_{N_{g\Sigma}}\xi(g\Sigma)$ for all $\xi\in H$ and put $A:=C^*(M,T)\subseteq\B(H)$. The equivariant part of
Lemma~\ref{lem:diagonal-jones-centralizer}, applied to the left action of
$\Gamma$ on $I$, supplies a unitary representation $u$ for which the
diagonal representation of $M$ is $\Gamma$-equivariant, $A$ is
$\Gamma$-invariant, and the associated vector states satisfy, for all $g,h\in\Gamma$,
$\omega_{g\Sigma}|_M=\tau$ and $g\cdot\omega_{h\Sigma}=\omega_{gh\Sigma}$.
Therefore $Y:=\overline{\{\omega_{g\Sigma}:g\Sigma\in I\}}^{\,w^*}$ is a compact $\Gamma$-invariant subset of $K$.
\medskip
\noindent Let $\beta I$ be the Stone--\v{C}ech compactification of the discrete
space $I$. Since both $\mathcal{U}$ and $Y$ are compact Hausdorff, the $\Gamma$-equivariant maps
$$
I\to\mathcal U,
\,\,
g\Sigma\mapsto N_{g\Sigma}\quad\text{and}\quad
I\to Y,
\,\,
g\Sigma\mapsto\omega_{g\Sigma},
$$
extend to continuous $\Gamma$-equivariant maps $\widetilde q:\beta I\to\mathcal U$ and $r:\beta I\to Y$ respectively. Note that $\widetilde{q}$ is surjective since, by minimality, the orbit of $N_0$ is dense in $\mathcal{U}$. Also, $r$ is surjective by definition of $Y$. We claim that $\widetilde q$ factors through $r$. Suppose that
$p,p'\in\beta I$ satisfy $r(p)=r(p')=:\varphi$. Along the ultrafilter $p$,
one has $\omega_{g\Sigma}\to r(p)=\varphi$ and $N_{g\Sigma}\to\widetilde q(p)$ and Lemma~\ref{lem:diagonal-jones-centralizer} gives $M_\varphi=\widetilde q(p)$. The same argument along $p'$ gives $M_\varphi=\widetilde q(p')$. Thus $\widetilde q(p)=\widetilde q(p')$. {Since $r$ is a continuous surjection from the compact space $\beta I$ onto the Hausdorff space $Y$, it is a quotient map. Therefore}
there is a unique continuous $\Gamma$-equivariant map $q:Y\to\mathcal U$ such that $\widetilde q=q\circ r$.
\medskip
\noindent Let $\varphi\in Y$ and choose $p\in\beta I$ such that $r(p)=\varphi$.
Another application of Lemma~\ref{lem:diagonal-jones-centralizer} yields $M_\varphi=\widetilde q(p)=q(\varphi)$. Thus the centralizer map $\Phi:Y\to\Sub(M)$ is equal to $q$, and hence is
continuous. Moreover, $\Phi(Y)=q(Y)=\widetilde q(\beta I)=\mathcal U$. Choose a minimal non-empty closed $\Gamma$-invariant subset
$\mathcal Z\subseteq Y$. Since $q(\mathcal Z)$ is a non-empty closed
$\Gamma$-invariant subset of the minimal space $\mathcal U$, one has $q(\mathcal Z)=\mathcal U$. Therefore the centralizer map on $\mathcal Z$ is $\Phi=q|_{\mathcal Z}$ so it is continuous and $\Phi(\mathcal Z)=\mathcal U$. For every
$\varphi_0\in\mathcal Z$, minimality of $\mathcal U$ gives
$\overline{\Gamma\cdot M_{\varphi_0}}^{\,\mathrm{EM}}
=
\overline{\Gamma\cdot q(\varphi_0)}^{\,\mathrm{EM}}
=
\mathcal U$.
}
\end{proof}
\begin{remark}\label{rmk:urs-compact-state-space-realization}
Note that every URS on $\Gamma$ is obtained by the state space construction. Indeed, let $\Gamma$ be a countable discrete group and let
$\mathcal S\subseteq\Sub(\Gamma)$ be a URS. The continuous injective
$\Gamma$-equivariant map $\Sub(\Gamma)\to\Sub(\LL(\Gamma))$, $H\mapsto\LL(H)$
restricts to a homeomorphism from $\mathcal S$ onto the compact URA
$\mathcal U_{\mathcal S}:=\{\LL(H):H\in\mathcal S\}$. Consequently, Theorem~\ref{ThmCompactURA} gives a compact
minimal state-space realization of $\mathcal U_{\mathcal S}$.
\end{remark}
%%%%%%%%
\subsection{The discrete state space construction and discrete URA}
\label{subsec:discreteURAs}
The minimal-state-space framework of
Theorem~\ref{thm:general-ura-construction} captures every compact URA, as we
just saw. More generally, we now characterize the discrete URAs which arise
from a continuity point of a $\Gamma$-equivariant map defined on an arbitrary
compact minimal $\Gamma$-space: they are precisely the finite ones. We then
give a non-minimal state-space construction which realizes every discrete URA
(Theorem~\ref{thm:discrete-uras}).
\begin{proposition}\label{prop:discrete-minimal-state-space-characterization}
Let $\Gamma$ be a countable discrete group acting by trace-preserving
automorphisms on a finite von Neumann algebra $(M,\tau)$ with separable
predual, and let $\mathcal U\subseteq\Sub(M)$ be a discrete URA. The
following are equivalent.
\begin{enumerate}
\item There is a compact minimal $\Gamma$-space $\mathcal Z$, a
$\Gamma$-equivariant map $\Phi:\mathcal Z\to\Sub(M)$ and a point of continuity $z_0\in\mathcal{Z}$ of $\Phi$ such that $\mathcal U=\overline{\Gamma\cdot\Phi(z_0)}^{\,\mathrm{EM}}$.
\item $\mathcal U$ is finite.
\item $\mathcal U$ is compact.
\item $[\Gamma:\mathcal U]<\infty$.
\end{enumerate}
\end{proposition}
\begin{proof}
${\rm(1)}\Rightarrow{\rm(4)}$. Since $\mathcal U$ is discrete,
Lemma~\ref{LemObstruction} gives $[\Gamma:\operatorname{Stab}_\Gamma(\Phi(z_0))]<\infty$. As $\Phi(z_0)\in\mathcal U$, Definition~\ref{def:group-index-of-ura} gives ${\rm(4)}$.
\medskip
\noindent${\rm(4)}\Rightarrow{\rm(2)}$. Fix $N\in\mathcal U$ so $\Gamma\cdot N$ is finite
and therefore closed. Minimality gives
$\mathcal U=\overline{\Gamma\cdot N}^{\,\mathrm{EM}}=\Gamma\cdot N$ so ${\rm(2)}$ holds. Since $\mathcal U$ is discrete,
\medskip
\noindent${\rm(2)}\Leftrightarrow{\rm(3)}$ is clear since $\Ucal$ is discrete.
\medskip
\noindent${\rm(2)}\Rightarrow{\rm(1)}$. Put $\mathcal Z:=\mathcal U$, endowed with the
restricted action, let $\Phi:\mathcal Z\to\Sub(M)$ be the inclusion map, and
fix $z_0\in\mathcal Z$. Then $\mathcal Z$ is compact and minimal, $\Phi$ is
continuous and $\Gamma$-equivariant, and $\overline{\Gamma\cdot\Phi(z_0)}^{\,\mathrm{EM}}=\mathcal U$. Thus ${\rm(1)}$ holds.\end{proof}
\medskip
\noindent By
Proposition~\ref{prop:discrete-minimal-state-space-characterization}, every
infinite discrete URA is invisible to every construction based on a
continuity point of an equivariant map defined on a compact minimal
$\Gamma$-space (and an explicit example of an infinite
discrete URA will be constructed in the next section, Proposition~\ref{prop:finite-action-fixed-point-rigidity} and Remark~\ref{rmk:finite-centralizer-examples}). Nonetheless, every discrete URA can be captured by a
state-space construction if the ambient state space is not required to be
minimal and isolated points are used as continuity points of the centralizer
map.
\begin{theorem}
\label{thm:discrete-uras}
Use the setting of Subsection~\ref{subsec:generalconstruction}. Let
$Y$ be a compact $\Gamma$-invariant subset of the state space $K$. If
$\varphi_0\in Y$ is isolated then $\Phi:Y\to\Sub(M)$, the centralizer
map relative to $A$, is continuous at every point of
$\Gamma\cdot\varphi_0$ and the following are equivalent:
\begin{enumerate}
\item $\Gamma\cdot M_{\varphi_0}$ is a discrete URA;
\item $\Gamma\cdot M_{\varphi_0}$ is closed and discrete in $\Sub(M)$;
constant;
\end{enumerate}
\medskip
\noindent Conversely, let $(M,\tau)$ be a finite von Neumann algebra with separable
predual, $\alpha:\Gamma\curvearrowright(M,\tau)$ a trace preserving action and $\mathcal U\subseteq\Sub(M)$ a discrete URA. There exist
\begin{itemize}
\item a unitary representation $u:\Gamma\to\Ucal(H)$ on a
separable Hilbert space $H$;
\item a faithful $\Gamma$-equivariant representation
$M\subseteq\B(H)$;
\item a $\Gamma$-invariant unital $C^*$-subalgebra
$A\subseteq\B(H)$ containing $M$,
\end{itemize}
with the following property: there is a compact $\Gamma$-invariant subset
$$Y\subseteq\{\varphi\in S(\B(H)):\varphi|_M=\tau\}$$
and an isolated point $\varphi_0\in Y$ such that the centralizer map
relative to $A$ is $\Gamma$-equivariant, is continuous at every point of
$\Gamma\cdot\varphi_0$, and satisfies
$\Gamma\cdot M_{\varphi_0}=\mathcal U$. In particular, the equivalent conditions above hold for these data.
\end{theorem}
\begin{proof}
Since a discrete URA is
closed and discrete by definition, ${\rm(1)}$ implies ${\rm(2)}$.
Conversely, assume ${\rm(2)}$, and let $\mathcal F$ be a non-empty closed
$\Gamma$-invariant subset of $\Gamma\cdot M_{\varphi_0}$. Choose
$Q\in\mathcal F$. Transitivity gives
$\Gamma\cdot Q=\Gamma\cdot M_{\varphi_0}$, and $\Gamma$-invariance of
$\mathcal F$ yields $\Gamma\cdot M_{\varphi_0}\subseteq\mathcal F$.
Hence $\mathcal F=\Gamma\cdot M_{\varphi_0}$, so
$\Gamma\cdot M_{\varphi_0}$ is minimal and ${\rm(1)}$ holds. Finally,
$\varphi_0$ is isolated and the action on $Y$ is by homeomorphisms. Hence
every point of $\Gamma\cdot\varphi_0$ is isolated in
$Y$, and therefore $\Phi$ is continuous at every such point.
\medskip
\noindent For the converse, fix $N\in\mathcal U$ and put
$\Sigma:=\operatorname{Stab}_\Gamma(N)$, $I:=\Gamma/\Sigma$, and
$N_{g\Sigma}:=\alpha_g(N)$. Minimality gives $\overline{\Gamma\cdot N}^{\,\mathrm{EM}}=\mathcal U$.
Since $\mathcal U$ is discrete, every dense subset of $\mathcal U$ is
equal to $\mathcal U$. Therefore $\mathcal U=\Gamma\cdot N=\{N_{g\Sigma}:g\Sigma\in I\}$. Apply
Lemma~\ref{lem:diagonal-jones-centralizer} to
$(N_{g\Sigma})_{g\Sigma\in I}$ and put
$H:=\ell^2(I,\LL^2(M,\tau))$, $A:=C^*(M,T)$, and
$Y:=\overline{\{\omega_{g\Sigma}:g\Sigma\in I\}}^{\,w^*}$. Since
$\Gamma$ is countable and $M$ has separable predual, $H$ is separable.
Lemma~\ref{lem:diagonal-jones-centralizer} shows that the diagonal
representation of $M$ is faithful and $\Gamma$-equivariant, that $A$ and
$Y$ are $\Gamma$-invariant, and that $\omega_{g\Sigma}|_M=\tau$, $M_{\omega_{g\Sigma}}=N_{g\Sigma}=\alpha_g(N)$ for all $g\in\Gamma$. Since restriction to $M$ is weak-$*$ continuous, it follows that $Y\subseteq\{\varphi\in S(\B(H)):\varphi|_M=\tau\}$. The same lemma shows that $\omega_\Sigma$ is isolated in $Y$. Taking
$\varphi_0:=\omega_\Sigma$ gives
$M_{\varphi_0}=N$ and $\Gamma\cdot M_{\varphi_0}=\mathcal U$. Finally,
Lemma~\ref{lem:fix-is-subalgebra} \rm(2) gives the equivariance of the
centralizer map, the first part of the proof gives its continuity on
$\Gamma\cdot\varphi_0$, and ${\rm(1)}$--${\rm(3)}$ hold because this
centralizer orbit is $\mathcal U$.
\end{proof}

\noindent It is natural to ask when the realization just constructed is already
covered by the minimal framework of Theorem~\ref{thm:general-ura-construction}.

\begin{remark}\label{rmk:compactnessandbroaderconstruction}
In the realization of Theorem~\ref{thm:discrete-uras}, the
equivalent conditions of
Proposition~\ref{prop:discrete-minimal-state-space-characterization} are also
equivalent to each of the following conditions:
\begin{enumerate}
\item $Y$ is finite;
\item $Y$ is minimal;
\item $\varphi_0$ belongs to a minimal non-empty closed
$\Gamma$-invariant subset of $Y$.
\end{enumerate}
Indeed, put $\Sigma:=\operatorname{Stab}_\Gamma(M_{\varphi_0})$. By
Lemma~\ref{lem:diagonal-jones-centralizer}, the states
$g\varphi_0=\omega_{g\Sigma}$, for $g\Sigma\in\Gamma/\Sigma$, are pairwise
distinct and $Y=\overline{\Gamma\cdot\varphi_0}^{\,w^*}$. Hence $Y$ is finite if and only if $[\Gamma:\Sigma]<\infty$, which is
equivalent to the conditions of
Proposition~\ref{prop:discrete-minimal-state-space-characterization}.
If $Y$ is finite, then
$Y=\Gamma\cdot\varphi_0$ and transitivity gives minimality. Conversely,
if $Y$ is minimal, then every point of $\Gamma\cdot\varphi_0$ is isolated,
so $\Gamma\cdot\varphi_0$ is an open invariant subset of $Y$. Hence
$\Gamma\cdot\varphi_0$ is equal to $Y$. Thus $Y$ is compact and discrete,
and hence finite. Finally, ${\rm(2)}\Rightarrow{\rm(3)}$ is immediate. If
${\rm(3)}$ holds and
$\mathcal Z\subseteq Y$ is minimal, closed, invariant, and contains
$\varphi_0$, then we have $
Y=\overline{\Gamma\cdot\varphi_0}^{\,w^*}\subseteq\mathcal Z$, so $\mathcal Z=Y$ and ${\rm(2)}$ holds.
\medskip
\noindent When these conditions hold, the realization belongs both to the compact
framework of Theorem~\ref{ThmCompactURA} and to the minimal framework of
Theorem~\ref{thm:general-ura-construction}. If they fail, every minimal
closed $\Gamma$-invariant subset of $Y$ is contained in
$Y\setminus\Gamma\cdot\varphi_0$.
Proposition~\ref{prop:discrete-minimal-state-space-characterization}
shows that no continuity point of a $\Gamma$-equivariant map defined on a
compact minimal $\Gamma$-space can have orbit closure $\mathcal U$.
Thus Theorem~\ref{thm:discrete-uras} strictly extends
Theorem~\ref{thm:general-ura-construction} for infinite discrete URAs.
\end{remark}

\noindent Discreteness of a $\Gamma$-orbit in $\Sub(M)$ is a property of the
acting group, and not of the ambient algebra alone. If one considers the inner action, the orbit of a subalgebra under
the whole unitary group of $M$ is path connected, hence it is discrete only when
it reduces to a single point. Consequently, the discreteness hypothesis in
Theorem~\ref{thm:discrete-uras} constrains the chosen subgroup of inner
automorphisms and says nothing about the full inner orbit.
\medskip

\begin{remark}\label{rmk:inner-orbit-continuity}
Let $(M,\tau)$ be a finite von Neumann algebra with separable predual and
let $N\in\Sub(M)$. Put $\mathcal O_{\mathrm{inn}}(N):=
\{uNu^*:u\in\Ucal(M)\}$. The orbit map
$$\Ucal(M)\to\Sub(M)\quad u\mapsto uNu^*,$$
is continuous when $\Ucal(M)$ is equipped with the $\|\cdot\|_2$-topology.
Indeed, for $u,v\in\Ucal(M)$ and $x\in M$,
$$
\begin{aligned}
\|E_{uNu^*}(x)-E_{vNv^*}(x)\|_2
&\leq
\|E_N(u^*xu)-E_N(v^*xv)\|_2
+
\|uE_N(v^*xv)u^*-vE_N(v^*xv)v^*\|_2\\
&\leq
\|u^*xu-v^*xv\|_2+2\|x\|\,\|u-v\|_2
\leq4\|x\|\,\|u-v\|_2.
\end{aligned}
$$
\noindent Moreover, $\mathcal O_{\mathrm{inn}}(N)$ is path connected. Indeed, by
bounded Borel functional calculus, every $u\in\Ucal(M)$ has the form
$u=e^{ih}$ for some $h=h^*\in M$, and the continuous path $[0,1]\to\Sub(M)$, $t\mapsto e^{ith}Ne^{-ith}$
joins $N$ to $uNu^*$. Consequently, exactly one of the following holds:
\begin{enumerate}
\item $uNu^*=N$ for every $u\in\Ucal(M)$, and
$\mathcal O_{\mathrm{inn}}(N)=\{N\}$;
\item $\mathcal O_{\mathrm{inn}}(N)$ has no isolated point.
\end{enumerate}
In particular, $\mathcal O_{\mathrm{inn}}(N)$ is discrete if and only if it
is a singleton. Thus, when the $\Gamma$-action is inner, discreteness of the
$\Gamma$-orbit in Theorem~\ref{thm:discrete-uras} is a property of the
prescribed subgroup of inner automorphisms and does not imply discreteness
of the full inner orbit.
\end{remark}
%%%%%%%%%%%%%%%%%%%%%%%%
\subsection{Exotic URAs from fixed-point and corner subalgebras}
\label{subsec:exoticURAs}
We now present additional families of exotic URAs, complementing the
universality theorem of Section~\ref{sec:URA}. We construct two genuinely
different families: \emph{fixed-point URAs}, arising from finite groups of
automorphisms of $\Gamma$, and \emph{corner URAs}, arising from projections in
$\LL(\Gamma)$. For the paired constructions associated with an
element $g\in\Gamma$ of finite order, the resulting fixed-point and corner
URAs are disjoint when the order of $g$ is at least $3$ and coincide when
the order is $2$, under the hypotheses of
Remark~\ref{rmk:fixed-point-corner-index-comparison}. Closedness and discreteness
follow from finite rigidity hypotheses which are strictly weaker than finite
generation. A later example shows that such a hypothesis cannot be omitted.
Throughout this subsection, $\{\lambda_g:g\in\Gamma\}$ denotes the
canonical unitaries of $\LL(\Gamma)$, and every element of
$\Aut(\Gamma)$ is denoted by the same symbol as its trace-preserving
extension to $\LL(\Gamma)$.
\medskip
\noindent We begin with the fixed-point family and isolate the construction in a
definition.
\begin{definition}\label{def:fixed-point-ura}
Let $\Gamma$ be a countable group and let
$G\leq\Aut(\Gamma)$ be a non-trivial finite subgroup. We call
$\overline{\Gamma\cdot \LL(\Gamma)^G}^{\,\mathrm{EM}}\subset\Sub(\LL(\Gamma))$ the
\textbf{fixed-point URA} associated with $G$, when it is a URA. If
$G=\langle\phi\rangle$ is cyclic, we also call it the fixed-point URA
associated with $\phi$.
\end{definition}

\noindent Two things remain to be checked. First, the members of such a family are not group von
Neumann algebras, and second, that the orbit is infinite. Both are consequences of the
rigidity lemmas of Section~\ref{sec:preliminaries} and are settled first. The
delicate point, closedness and discreteness of the orbit, is treated afterwards
and is where a finiteness hypothesis on centralizers enters.
\medskip

\begin{lemma}\label{lem:finite-fixed-point-algebra}
Let $\Gamma$ be a group and let
$G\leq\Aut(\Gamma)$ be a non-trivial finite subgroup.
\begin{enumerate}
\item There is no subgroup $H\leq\Gamma$ such that $\LL(\Gamma)^G=\LL(H)$.
\item If $\Gamma$ is i.c.c. then $[\Gamma:\operatorname{Stab}_\Gamma(\LL(\Gamma)^G)]=\infty$.
\end{enumerate}
\end{lemma}
\begin{proof}
Put $N:=\LL(\Gamma)^G$.
\medskip
\noindent{\rm(1)}. Suppose that $N=\LL(H)$. For every $g\in\Gamma$, $E_N(\lambda_g)=\frac1{|G|}\sum_{\alpha\in G}\lambda_{\alpha(g)}\in\LL(H)$. The Fourier coefficient at $g$ is non-zero, so $g\in H$. Hence $H=\Gamma$.
On the other hand, choose $\alpha\in G$ and $g\in\Gamma$ such that
$\alpha(g)\neq g$. Then $E_N(\lambda_g)\neq\lambda_g$, so
$N\neq\LL(\Gamma)$.
\medskip
\noindent{\rm(2)}. Since $\Gamma$ is i.c.c. it has ${\rm(FIR)}$ and
$Z(\Gamma)=\{e\}$. Put $K:=\operatorname{Stab}_\Gamma(N)$. For $s\in K$, one has $\LL(\Gamma)^{\Ad(s)G\Ad(s^{-1})}
=\Ad(s)(N)
=N
=\LL(\Gamma)^G.$ Lemma~\ref{lem:finite-fixed-point-action-rigidity} therefore gives $\Ad(s)G\Ad(s^{-1})=G$. Thus conjugation defines a homomorphism
$$
\rho:K\longrightarrow\operatorname{Aut}(G),
\quad
\rho(s)(\alpha)
=\Ad(s)\circ\alpha\circ\Ad(s^{-1}).
$$
Its kernel $K_0$ has finite index in $K$. Fix
$\phi\in G\setminus\{\id\}$. If $s\in K_0$, then
$\Ad(s)\circ\phi\circ\Ad(s^{-1})=\phi$ and hence $s\phi(s^{-1})\in Z(\Gamma)=\{e\}$. Therefore $K_0\subseteq\operatorname{Fix}(\phi)$. If $K$ had finite index
in $\Gamma$, then $K_0$, and hence $\operatorname{Fix}(\phi)$, would have
finite index in $\Gamma$, contradicting
Lemma~\ref{lem: finite order automorphism has infinite index fixed point subgroups}.
Thus $[\Gamma:K]=\infty$.\end{proof}

\begin{proposition}\label{prop:finite-action-fixed-point-rigidity}
Let $\Gamma$ be a countable group and let
$G\leq\Aut(\Gamma)$ be a non-trivial finite subgroup. Assume that
there exists a finite set $F\subset\Gamma$ such that
$C_\Gamma(F)/Z(\Gamma)$ is finite. Then
\begin{enumerate}
\item $\mathcal U:=\{\lambda_s\LL(\Gamma)^G\lambda_s^*:s\in\Gamma\}$ is a discrete exotic URA.
\item If moreover $\Gamma$ is i.c.c., then $\mathcal U$ has infinite
cardinality.
\end{enumerate}
\end{proposition}
\begin{proof}
{\rm(1)}. Put $N:=\LL(\Gamma)^G$. It suffices to show that $\mathcal U$ is closed
and discrete since, in that case, it will be minimal by transitivity and exotic by
Lemma~\ref{lem:finite-fixed-point-algebra}. Suppose that $\lambda_{s_i}N\lambda_{s_i}^*\to Q$. It suffices to show that the sequence $(\lambda_{s_i}N\lambda_{s_i}^*)_i$ has a constant subsequence. For
$\alpha\in G$, put $c_{i,\alpha}:=s_i\alpha(s_i^{-1})$. For every $t\in\Gamma$,
\begin{equation}\label{eq:finite-action-conjugate-expectation}
|G|E_{\lambda_{s_i}N\lambda_{s_i}^*}(\lambda_t)
=
\sum_{\alpha\in G}
\lambda_{c_{i,\alpha}\alpha(t)c_{i,\alpha}^{-1}}.
\end{equation}
For all sufficiently large $i,j$ and every $t\in F$, the two sums in
\eqref{eq:finite-action-conjugate-expectation} have $\Vert\cdot\Vert_2$-distance less than
$\sqrt2$.
Lemma~\ref{lem: close multisets are equal} implies that the corresponding
multisets coincide. Passing to a subsequence, all the finitely many elements
$$
c_{i,\alpha}\alpha(t)c_{i,\alpha}^{-1},
\qquad
(t,\alpha)\in F\times G,
$$
are independent of $i$. Hence, for any two indices in this subsequence,
$$
c_{j,\alpha}^{-1}c_{i,\alpha}
\in C_\Gamma(\alpha(F))
=
\alpha(C_\Gamma(F)).
$$
Since
$\alpha(C_\Gamma(F))/Z(\Gamma)$ is finite, a further subsequence makes
$\Ad(c_{i,\alpha})$ independent of $i$ for every $\alpha\in G$. Thus the
conditional expectations in the subsequence are equal.
\medskip
\noindent{\rm(2)} Apply point {\rm(1)} and
Lemma~\ref{lem:finite-fixed-point-algebra} {\rm(2)}.\end{proof}

\noindent We now record the two natural sources of groups satisfying the
hypothesis above. The second one shows that it is strictly weaker than finite
generation.
\medskip

\begin{remark}\label{rmk:finite-centralizer-examples}
Let $\Gamma$ be a finitely generated i.c.c. group and let $F\subset\Gamma$ be
a finite generating set. Then $C_\Gamma(F)=Z(\Gamma)=\{e\}$ and the hypotheses of
Proposition~\ref{prop:finite-action-fixed-point-rigidity} are satisfied.
\medskip
\noindent More generally, let $\Delta$ be a non-trivial finitely generated group with finite
center, let $\Lambda$ be a non-trivial countable group which is not finitely
generated, and put $\Gamma:=\Delta*\Lambda$. Then $\Gamma$ is countable, is not finitely generated, and is i.c.c. Let
$F\subset\Delta$ be a finite generating set. Since $\Delta$ is malnormal in $\Gamma$, $C_\Gamma(F)=C_\Gamma(\Delta)=Z(\Delta)$, which is finite. Hence, the hypotheses of
Proposition~\ref{prop:finite-action-fixed-point-rigidity} are satisfied.
\medskip
\noindent More specifically, one may take $\Gamma=\F_\infty=\langle a_1,a_2,a_3,\ldots\rangle$ and the order-two automorphism $\phi\in\Aut(\Gamma)$ determined by
$a_1\mapsto a_2$, $a_2\mapsto a_1$ and $a_n\mapsto a_n$ for all $n\geq3$. Indeed, $C_{\F_\infty}(a_1)=\langle a_1\rangle$ and $C_{\F_\infty}(a_2)=\langle a_2\rangle$ so $C_{\F_\infty}(\{a_1,a_2\})=\{e\}$. Proposition~\ref{prop:finite-action-fixed-point-rigidity} {\rm(2)} shows that $\mathcal U:=\{\lambda_s\LL(\mathbb{F}_\infty)^\phi\lambda_s^*:s\in\mathbb{F}_\infty\}$ is an infinite discrete URA.
\end{remark}

\noindent The hypothesis of
Proposition~\ref{prop:finite-action-fixed-point-rigidity} is a condition on the
group $\Gamma$. We now show that it is implied by a condition on the von Neumann
algebra $\LL(\Gamma)$ alone, namely the triviality of the relative commutant in
the tracial ultrapower. The passage from one to the other is a standard local
approximation.

\noindent For a free ultrafilter $\omega\in\beta\N\setminus\N$, let $M^\omega$
denote the tracial ultrapower of $M$, that is, the quotient of
$\ell^\infty(\N,M)$ by the ideal
$\{(x_n)_n:\lim_{n\to\omega}\|x_n\|_2=0\}$, with $M\subset M^\omega$ embedded as
constant sequences.
\begin{lemma}\label{lem:full-factor-unitary-gap}
Let $(M,\tau)$ be a finite von Neumann algebra with separable predual such
that $M'\cap M^\omega=Z(M)^\omega$. For every $\varepsilon>0$, there exist a finite
set $F\subset M$ and $\delta>0$ such that every $u\in\Ucal(M)$ satisfying
$\max_{x\in F}\Vert ux-xu\Vert_2<\delta$
satisfies $\dist_2(u,Z(M))<\varepsilon$.
\end{lemma}
\begin{proof}
Let $(x_n)_n$ be $\Vert\cdot\Vert_2$-dense in the unit
ball of $M$. If the statement is false, there exist $\varepsilon>0$ and $u_n\in\Ucal(M)$ such that $\max_{1\leq k\leq n}\Vert u_nx_k-x_ku_n\Vert_2<\frac1n$ and $\dist_2(u_n,Z(M))\geq\varepsilon$. Then $(u_n)_n\in M'\cap M^\omega=Z(M)^\omega$, so
$\lim_{n\to\omega}\dist_2(u_n,Z(M))=0$, contradicting the second inequality.
\end{proof}
\begin{corollary}\label{cor:fixed-point-ura}
Let $\Gamma$ be a countable group and assume that
$$
\LL(\Gamma)'\cap\LL(\Gamma)^\omega
=Z(\LL(\Gamma))^\omega.
$$
Let $G\leq\Aut(\Gamma)$ be a non-trivial finite subgroup. The following
holds.
\begin{enumerate}
\item $\mathcal U
:=\{\lambda_s\LL(\Gamma)^G\lambda_s^*:s\in\Gamma\}$ is a discrete exotic URA.
\item If $\Gamma$ is i.c.c. then $\mathcal U$ has
infinite cardinality.
\end{enumerate}
\end{corollary}
\begin{proof}
\noindent We first claim that there exists a finite set
$F\subset\Gamma$ such that $C_\Gamma(F)=Z(\Gamma)$. Fix $0<\varepsilon<1/\sqrt2$, and choose a finite set
$\mathcal F\subset\LL(\Gamma)$ and $\delta>0$ as in
Lemma~\ref{lem:full-factor-unitary-gap}. For every $x\in\mathcal F$, choose
a finite Fourier sum $y_x\in\mathbb C[\Gamma]$ such that $\Vert x-y_x\Vert_2<\frac{\delta}{2}$. Let $F\subset\Gamma$ be the union of the Fourier supports of the elements
$y_x$, $x\in\mathcal F$. If $g\in C_\Gamma(F)$, then
$\lambda_gy_x=y_x\lambda_g$ for every $x\in\mathcal F$, and hence $\Vert\lambda_gx-x\lambda_g\Vert_2
\leq2\Vert x-y_x\Vert_2
<\delta$. Lemma~\ref{lem:full-factor-unitary-gap} gives $\dist_2(\lambda_g,Z(\LL(\Gamma)))<\varepsilon$.
If $g\notin Z(\Gamma)$, then $\dist_2(\lambda_g,Z(\LL(\Gamma)))^2
=
1-\Vert E_{Z(\LL(\Gamma))}(\lambda_g)\Vert_2^2
\geq\frac12$,
a contradiction, since the orthogonal projection onto
$\LL^2(Z(\LL(\Gamma)))$ satisfies
$$
\Vert E_{Z(\LL(\Gamma))}(\lambda_g)\Vert_2^2
=
\begin{cases}
|\operatorname{Cl}_\Gamma(g)|^{-1},
&|\operatorname{Cl}_\Gamma(g)|<\infty,\\
0,
&|\operatorname{Cl}_\Gamma(g)|=\infty,
\end{cases}
$$
and $|\operatorname{Cl}_\Gamma(g)|\geq2$. Thus
$C_\Gamma(F)\subseteq Z(\Gamma)$, and the reverse
inclusion is immediate.
\medskip
\noindent Since
$C_\Gamma(F)/Z(\Gamma)=\{Z(\Gamma)\}$,
Proposition~\ref{prop:finite-action-fixed-point-rigidity} gives both
assertions.\end{proof}
\noindent The fixed-point URAs above are built from averaging finite groups
of automorphisms. We now give a genuinely different mechanism, in which
the subalgebra is a \emph{corner} $pMp\oplus(1-p)M(1-p)$ cut out by a single
projection rather than a fixed-point algebra. For the paired
constructions associated with a finite-order element of $\Gamma$, the
Pimsner--Popa index comparison is given in
Remark~\ref{rmk:fixed-point-corner-index-comparison}.
\begin{definition}\label{def:corner-ura}
Let $\Gamma$ be a countable i.c.c.\ group and let $p\in\LL(\Gamma)$ be a
projection with $0<\tau(p)<1$. Writing
$N=p\LL(\Gamma)p\oplus(1-p)\LL(\Gamma)(1-p)$, we call
$\overline{\Gamma\cdot N}^{\,\mathrm{EM}}\subset\Sub(\LL(\Gamma))$ the
\textbf{corner URA} associated with $p$, when it is a URA.
\end{definition}

\noindent A corner is the fixed-point algebra of the inner symmetry
$\Ad(2p-1)$, so the two families are formally parallel; the difference is that
the symmetry is now implemented inside $M$ rather than coming from
$\Aut(\Gamma)$. The next lemma records the corresponding conditional
expectation, shows that a corner determines its projection up to
complementation, and provides the $\LL^2$-estimate which makes the orbit map
continuous.
\medskip

\begin{lemma}\label{lem:block-diagonal-comparison}
Let $(M,\tau)$ be a $\mathrm{II}_1$ factor and let $p,q\in M$ be non-trivial
projections. Put $N_p:=pMp\oplus(1-p)M(1-p)$ and $u_p:=2p-1$, and define $N_q$ and $u_q$ similarly. Then
$E_{N_p}(x)=\frac12(x+u_pxu_p)$ for all $x\in M$ and the following are equivalent:
\begin{enumerate}
\item $N_p=N_q$,
\item $q\in\{p,1-p\}$,
\item $u_q\in\{u_p,-u_p\}$.
\end{enumerate}
Moreover, for every $x\in M$ with $\Vert x\Vert_\infty\leq1$, $\Vert E_{N_p}(x)-E_{N_q}(x)\Vert_2
\leq \Vert u_p-u_q\Vert_2$.
\end{lemma}
\begin{proof}
The expectation formula follows by expansion. Since $M$ is a factor,
$$
Z(N_p)=\C p\oplus\C(1-p).
$$
Thus $N_p=N_q$ implies that the non-trivial central projection $q$ of $N_q$
belongs to $\{p,1-p\}$. The converse and the equivalence with
$u_q\in\{u_p,-u_p\}$ are immediate. Finally,
$$
\begin{aligned}
2\Vert E_{N_p}(x)-E_{N_q}(x)\Vert_2
&=\Vert u_pxu_p-u_qxu_q\Vert_2
\leq\Vert u_px(u_p-u_q)\Vert_2
+\Vert(u_p-u_q)xu_q\Vert_2\\
&\leq2\Vert u_p-u_q\Vert_2.
\end{aligned}
$$\end{proof}

\noindent The following lemma isolates the mechanism which produces discreteness
for families of fixed-point algebras of self-adjoint unitaries. It plays, in the
corner setting, the role played by the multiset rigidity of
Lemma~\ref{lem: close multisets are equal} in the fixed-point setting.

\begin{lemma}\label{lem:uniformly-discrete-symmetries}
Let $(M,\tau)$ be a $\mathrm{II}_1$ factor with separable predual, let
$\mathcal F\subset M$ be finite, and let
$\mathcal V\subset\Ucal(M)$ consist of self-adjoint unitaries.
Assume that, for every $x\in\mathcal F$, the set
$\{uxu:u\in\mathcal V\}$
is uniformly discrete in $\LL^2(M)$, that is, for every $x\in\mathcal F$ there
exists $\delta_x>0$ such that any two distinct elements of this set have
$\LL^2$-distance at least $\delta_x$.
Assume moreover that
$\{\Ad(vu):u,v\in\mathcal V,\ vu\in\mathcal F'\cap M\}$
is finite. Then
$\{M^{\Ad(u)}:u\in\mathcal V\}$
is closed and discrete in $\Sub(M)$.
\end{lemma}
\begin{proof}
Let $u_n\in\mathcal V$ and assume that $M^{\Ad(u_n)}$ converges in
$\Sub(M)$. It suffices to show that $(M^{\Ad(u_n)})_n$ has a constant subsequence. Since
$E_{M^{\Ad(u_n)}}(x)=\frac12(x+u_nxu_n),$
the sequence $(u_nxu_n)_n$ is $\Vert\cdot\Vert_2$-Cauchy for every
$x\in\mathcal F$. Uniform discreteness and finiteness of $\mathcal F$ imply
that, for all sufficiently large $i,j$, $u_ixu_i=u_jxu_j$ for all $x\in\mathcal F$.
Hence $u_ju_i\in\mathcal F'\cap M$. Fixing one sufficiently large index
$i$ and passing to a subsequence, the finiteness assumption makes
$\Ad(u_ju_i)$ independent of $j$. Since
$$
\Ad(u_j)=\Ad(u_ju_i)\circ\Ad(u_i),
$$
the corresponding fixed-point algebras are constant on this subsequence.\end{proof}

\noindent Infinitude of the orbit is a separate matter, and it requires no
rigidity hypothesis at all. It holds for every non-trivial projection in the
group von Neumann algebra of an i.c.c.\ group.
\medskip

\begin{lemma}\label{lem:corner-orbit-infinite}
Let $\Gamma$ be a countable i.c.c. group and $p\in \LL(\Gamma)$ be a non-trivial projection. The orbit of
$N_p:=p\LL(\Gamma)p\oplus(1-p)\LL(\Gamma)(1-p)$ in $\Sub(\LL(\Gamma))$ is infinite.
\end{lemma}
\begin{proof}
Suppose that the orbit is finite and put
$K:=\operatorname{Stab}_\Gamma(N_p)$.
Then $K$ has finite index in $\Gamma$. By
Lemma~\ref{lem:block-diagonal-comparison}, the action of $K$ on
$\{p,1-p\}$ defines a homomorphism $K\to\Z/2\Z$. Its kernel
$K_0$ has finite index in $\Gamma$ and fixes $p$. Let 
$\Lambda:=\bigcap_{r\in\Gamma}rK_0r^{-1}$ be the normal core of $K_0$: it is a normal finite-index subgroup of $\Gamma$ and $\Lambda\leq K_0$ so
$p\in\LL(\Lambda)'\cap\LL(\Gamma)$. For $g\neq e$, the $\Lambda$-conjugacy
class of $g$ is infinite. Indeed, otherwise $C_\Lambda(g)$, and hence
$C_\Gamma(g)$, would have finite index, contrary to the i.c.c. hypothesis.
Comparison of Fourier coefficients therefore gives $\LL(\Lambda)'\cap\LL(\Gamma)=\C1$. Thus $p\in\C1$, contradicting that $p$ is non-trivial.
\end{proof}

\noindent The uniform discreteness required by
Lemma~\ref{lem:uniformly-discrete-symmetries} becomes available as soon as the
projection has rational coefficients in the group algebra.

\begin{lemma}\label{lem:rational-corner-rigidity}
Let $\Gamma$ be a countable i.c.c. group,
$p\in \LL(\Gamma)$ be a non-trivial projection such that
$p\in d^{-1}\Z[\Gamma]$ for some integer $d\geq1$. Assume that there exists a finite set
$F\subset\Gamma$ such that
$D_F:=\{\lambda_t:t\in F\}'\cap \LL(\Gamma)$
is finite-dimensional. Then
$\mathcal U:=\{\lambda_rN_p\lambda_r^*:r\in\Gamma\}$, where 
$N_p:=p\LL(\Gamma)p\oplus(1-p)\LL(\Gamma)(1-p)$,
is a discrete corner URA of infinite cardinality.
\end{lemma}
\begin{proof}
For $r\in\Gamma$, put $p_r:=\lambda_rp\lambda_r^*$ and $u_r:=2p_r-1$.
Then $\lambda_rN_p\lambda_r^*=\LL(\Gamma)^{\Ad(u_r)}$. For every $t\in\Gamma$,
$d^2E_{\lambda_rN_p\lambda_r^*}(\lambda_t)
=
d^2\bigl(p_r\lambda_tp_r+(1-p_r)\lambda_t(1-p_r)\bigr)
\in\Z[\Gamma]$. Consequently, for every $t\in F$, the set
$\{u_r\lambda_tu_r:r\in\Gamma\}$
is uniformly discrete in $\LL^2(\LL(\Gamma))$: two distinct elements have distance at
least $2/d^2$. Moreover, $\{u_su_r:r,s\in\Gamma\}\subseteq d^{-2}\Z[\Gamma]$ and every element of this set has $\LL^2$-norm one. Since
$d^{-2}\Z[\Gamma]$ is $d^{-2}$-separated and the unit ball of $D_F$ is
compact, the set
$D_F\cap\{u_su_r:r,s\in\Gamma\}$
is finite. Lemma~\ref{lem:uniformly-discrete-symmetries} gives closedness and
discreteness. The orbit is minimal because the action on it is transitive.
Infinitude follows from Lemma~\ref{lem:corner-orbit-infinite}.
\end{proof}

\noindent The most natural source of projections with rational coefficients is a
finite subgroup of $\Gamma$, and this is the case in which the Pimsner--Popa
index can be computed exactly.
\medskip

\begin{proposition}\label{prop:finite-subgroup-corner-ura}
Let $\Gamma$ be a countable i.c.c. group and $H\leq\Gamma$ be a non-trivial
finite subgroup. Assume that there exists a finite set
$F\subset\Gamma$ such that the subalgebra $\{\lambda_t:t\in F\}'\cap\LL(\Gamma)$ is finite-dimensional.
Define
$$
N_H:=p_H\LL(\Gamma)p_H\oplus(1-p_H)\LL(\Gamma)(1-p_H)\in\Sub(\LL(\Gamma)),\quad\text{where }p_H:=\frac{1}{\vert H\vert}\sum_{h\in H}\lambda_h.
$$
Then $\mathcal U_H:=\{\lambda_rN_H\lambda_r^*:r\in\Gamma\}$
is a discrete corner URA of infinite cardinality and
$[\LL(\Gamma):\mathcal U_H]=2$.
In particular, the conclusion holds whenever $\Gamma$ is finitely generated.
\end{proposition}
\begin{proof}
Note that $p_H$ is a non-trivial orthogonal projection in $\LL(\Gamma)$ and since $p_H\in \vert H\vert^{-1}\Z[\Gamma]$,
Lemma~\ref{lem:rational-corner-rigidity} shows that $\mathcal U_H$ is a discrete URA of infinite cardinality. Finally, Lemma~\ref{lem:index} and
Definition~\ref{def:pimsner-popa-index-of-ura} give $[\LL(\Gamma):\mathcal U_H]=[\LL(\Gamma):N_H]=2$. If $\Gamma$ is finitely generated, one may take for $F$ a finite generating set.\end{proof}

The following remarks are in order.

\begin{remark}\label{rmk:fixed-point-corner-index-comparison}
Let $\Gamma$ be a finitely generated i.c.c.\ group, let $g\in\Gamma$ have exact finite order
$m\geq2$, put $H:=\langle g\rangle$ and set
$\phi:=\Ad(g)\in\Aut(\Gamma)$, whose extension to $\LL(\Gamma)$ is
$\Ad(\lambda_g)$. By Proposition~\ref{prop:finite-action-fixed-point-rigidity}, $\mathcal V_\phi:=\overline{\{\lambda_r\LL(\Gamma)^\phi\lambda_r^*:r\in\Gamma\}}^{\,\mathrm{EM}}$ is a discrete infinite URA and Proposition~\ref{prop:finite-subgroup-corner-ura} produces another discrete infinite URA $\mathcal U_H$. Lemma~\ref{lem:index} 
gives $[\LL(\Gamma):\LL(\Gamma)^\phi]=m$, hence $[\LL(\Gamma):\mathcal V_\phi]=m$. On the other hand, Proposition~\ref{prop:finite-subgroup-corner-ura} gives $[\LL(\Gamma):\mathcal U_H]=2$. Therefore, if $m\geq3$, then $\mathcal V_\phi\cap\mathcal U_H=\emptyset$. If $m=2$, then
$p_H=\frac{1+\lambda_g}{2}$ and $\lambda_g=2p_H-1$.
Hence $
N_H=\{\lambda_g\}'\cap\LL(\Gamma)
=\LL(\Gamma)^\phi$
and therefore $\mathcal U_H=\mathcal V_\phi$.
\end{remark}
\begin{remark}
For every $\mathrm{II}_1$ factor $M$ and every non-trivial
projection $p\in M$, the algebra $N:=pMp\oplus (1-p)M(1-p)$ is maximal among the proper von Neumann subalgebras of $M$ as shown in \cite[Proposition 3.4]{zhou_maximal_typeI}. 
Maximality does not imply that the orbit closure is a URA. Indeed, let
a countable group $\Gamma$ act faithfully and sharply
$3$-transitively on an
infinite set and let $H:=\operatorname{Stab}_\Gamma(x)$. Every non-trivial element
of $\Gamma$ fixes at most two points. Thus, for every finite
$F\subseteq \Gamma\setminus\{e\}$, there exists $y\in X$ fixed by no element of
$F$. Consequently, $\{e\}$ belongs to the Chabauty closure of the conjugacy
class of $H$. By the continuity of $K\mapsto\LL(K)$, the orbit closure of
$\LL(H)$ contains $\LL(\{e\})=\C1$. Since $\C1$ is fixed and
$\LL(H)\neq\C1$, this orbit closure is not minimal. On the other hand,
$\LL(H)$ is maximal in $\LL(\Gamma)$ by
\cite[Theorem A]{Zhou_maximal_triple-transitive}.
\end{remark}
\noindent The preceding results prove discreteness of both the corner
URAs and the fixed-point URAs under their respective hypotheses. Hence all
the discrete examples obtained above fall within the scope of
Theorem~\ref{thm:discrete-uras}. The next
example shows that the finite-centralizer-modulo-the-center hypothesis in
Proposition~\ref{prop:finite-action-fixed-point-rigidity} cannot be omitted.
\medskip
\noindent We denote by $S_\infty$ the group of finitely supported bijections of $\N$.
\begin{proposition}\label{prop:fixed-point-non-fg-not-discrete}
Let $A:=\bigoplus_{\N}\Z/3\Z$ and $\Gamma:=A\rtimes_\sigma S_\infty$, where $\sigma:S_\infty\curvearrowright A$ permutes the coordinates. Define
$\phi\in\Aut(\Gamma)$ by $\phi(a,v)=(-a,v)$ and
$$\mathcal U:=\overline{\{\lambda_g\LL(\Gamma)^\phi\lambda_g^*:g\in\Gamma\}}^{\,\mathrm{EM}}.$$
Then $\Gamma$ is a locally finite i.c.c. group and $\mathcal U$ is an exotic compact fixed-point URA
homeomorphic to the Cantor space. Moreover $[\LL(\Gamma):\mathcal U]=2$.
\end{proposition}
\begin{proof}
The group $\Gamma$ is locally finite. Indeed, if
$F\subseteq\Gamma$ is finite, choose a finite set $I\subseteq\N$ containing
$\operatorname{supp}(a)\cup\operatorname{supp}(v)$ for every $(a,v)\in F$.
Then
$$
\langle F\rangle
\leq
\left(\bigoplus_I\Z/3\Z\right)\rtimes\operatorname{Sym}(I),
$$
and the group on the right is finite. $\Gamma$ is i.c.c.: if $(a,v)\neq e$ and $v\neq e$, its
conjugates by $(0,w)$, $w\in S_\infty$, attain infinitely many values because
the support of $wvw^{-1}$ can be moved to infinitely many distinct finite
subsets of $\N$. If $v=e$ and $a\neq0$, the elements $\sigma_w(a)$,
$w\in S_\infty$, attain infinitely many values. Moreover, if
$F\subseteq\Gamma$ is finite, choose $n\in\N$ outside the supports of all
components of the elements of $F$. Every element of $A$ supported at $n$
commutes with $F$. Repeating this for infinitely many $n$ shows that
$C_\Gamma(F)$ is infinite.
\medskip
\noindent Consider the compact group $B:=\prod_{\N}\Z/3\Z$ (with product topology). For $b\in B$, define $\psi_b\in\Aut(\Gamma)$ by $\psi_b(a,v)=(b-\sigma_v(b)+a,v)$. This is well defined because $b-\sigma_v(b)\in A$ and a direct computation shows that $\psi_b$ is indeed an automorphism. Another way to see that $\psi_b\in\Aut(\Gamma)$ is to remark that, in the larger group $B\rtimes_\sigma S_\infty$, one has $\psi_b(a,v)=(b,e)(a,v)(b,e)^{-1}$. Moreover, $b\mapsto\psi_b$ is a group homomorphism $B\to\Aut(\Gamma)$.
\medskip
\noindent For $g=(a,v)\in\Gamma$, one has $(\psi_b\phi\psi_b^{-1})(g)
=
(2b-2\sigma_v(b)-a,v)$. Define $N:=\LL(\Gamma)^\phi$ and  $\Theta:B\longrightarrow\Sub(\LL(\Gamma))$, $\Theta(b):=\psi_b(N)=\LL(\Gamma)^{\psi_b\phi\psi_b^{-1}}$. If $b_i\to b$ in $B$, then for every $g=(a,v)\in\Gamma$ the
automorphisms $\psi_{b_i}\phi\psi_{b_i}^{-1}$ and
$\psi_b\phi\psi_b^{-1}$ agree on $g$ for all sufficiently large $i$. Since $\phi$ has order $2$, $\psi_b\phi\psi_b^{-1}$ has order $2$ for all $b\in B$. Hence, $E_{\Theta(b)}=\frac12(\id+\psi_b\phi\psi_b^{-1})$
. It follows that $E_{\Theta(b_i)}(x)\longrightarrow E_{\Theta(b)}(x)$
in $\LL^2(\LL(\Gamma))$ for every $x\in\mathbb C[\Gamma]$. Since the
conditional expectations are $\LL^2$-contractive and
$\mathbb C[\Gamma]$ is $\LL^2$-dense in $\LL(\Gamma)$, the same convergence
holds for every $x\in\LL(\Gamma)$. Thus $\Theta$ is continuous.
\medskip
\noindent If $b,c\in B$ are such that $\Theta(b)=\Theta(c)$ then $E_{\Theta(b)}=E_{\Theta(c)}$ implies that $\psi_b\phi\psi_b^{-1}=\psi_c\phi\psi_c^{-1}$. Hence, $b-c=\sigma_v(b-c)$ for all $v\in S_\infty$. Thus $b-c$ is constant i.e. $b-c\in\langle\mathbf1\rangle$, where $\mathbf1:=(1,1,\dots)\in B$. The converse is immediate: if $b-c\in\langle\mathbf1\rangle$ then $\Theta(b)=\Theta(c)$. Therefore $\Theta$
induces a continuous injection $\overline\Theta:B/\langle\mathbf1\rangle\to\Sub(\LL(\Gamma))$. It is a homeomorphism onto its image because its domain is compact and its
codomain is Hausdorff. For $s=(a_s,\rho_s)\in\Gamma$, a direct computation gives $\lambda_sN\lambda_s^*=\Theta(a_s)$.
Since $A$ is dense in $B$, it follows that $\mathcal U=\Theta(B)\cong B/\langle\mathbf1\rangle$. Note that, for all $a\in A$ and $b\in B$,
$\lambda_a\Theta(b)\lambda_a^*=\Theta(b+a)$. Hence, since the image of $A$ is dense in $B/\langle\mathbf1\rangle$, every
$\Gamma$-orbit in $\mathcal U$ is dense. Hence $\mathcal U$ is a compact URA. Moreover, every coset in $B/\langle\mathbf1\rangle$ has a unique representative whose
first coordinate is zero. Consequently, $\mathcal U\cong\prod_{k\geq2}\Z/3\Z$ so $\mathcal U$ is a Cantor space. Since $\Theta(b)=\LL(\Gamma)^{\psi_b\phi\psi_b^{-1}}$ for every $b\in B$,
where $\psi_b\phi\psi_b^{-1}$ is a non-trivial automorphism of order two, Lemma~\ref{lem:finite-fixed-point-algebra} {\rm(1)} shows that $\mathcal U$ is exotic.
\medskip
\noindent Under the Fourier identification $\LL(A)=\LL^\infty(B)$, the
restriction of $\phi$ to $\LL(A)$ is induced by the involution
$b\mapsto-b$ of $B$. Let
$$
n(b):=\min\{n\in\N:b_n\neq0\},
\quad
C:=\{b\in B\setminus\{0\}:b_{n(b)}=1\},
$$
and let $p:=1_C\in\LL(A)$. Since the Haar measure of $\{0\}$ is zero,
$\phi(p)=1-p$. Writing $N:=\LL(\Gamma)^\phi$, one has
$E_N(x)=\frac12(x+\phi(x))\geq\frac12x$ for all $x\in\LL(\Gamma)^+)$ and hence $[\LL(\Gamma):N]\leq2$. On the other hand,
$E_N(p)=\frac12 1$. If $E_N(p)\geq cp$ for some $c>0$, compression by $p$
gives $\frac12p=pE_N(p)p\geq cp$. Thus $c\leq\frac12$, so $[\LL(\Gamma):N]=2$. Lemma~\ref{lem:pimsner-popa-index-of-ura}
and Definition~\ref{def:pimsner-popa-index-of-ura} give
$[\LL(\Gamma):\mathcal U]=2$.
\end{proof}
\noindent The discrete fixed-point and corner URAs constructed
above consist of subalgebras with finite Pimsner--Popa index. In the
classical setting of $\Sub(\Gamma)$, every finite-index subgroup of a
finitely generated group $\Gamma$ is isolated in the Chabauty topology.
Proposition~\ref{prop:fixed-point-non-fg-not-discrete} already shows that
finite Pimsner--Popa index does not imply isolation in an orbit closure. However the
ambient group in that example, being locally finite and infinite, is not finitely generated. The following
proposition strengthens this phenomenon by showing that, even for a
two-generated amenable i.c.c.
group $\Gamma$, a maximal finite-index subalgebra of $\LL(\Gamma)$ need not
be isolated in its orbit closure. The orbit closure below is not asserted to
be a URA.
\begin{proposition}\label{ex:finitely-generated-nonisolation}
Let $\operatorname{Sym}_{\mathrm{fin}}(\Z)$ be the group of finitely supported bijections of $\Z$,
$\theta\in\Aut(\operatorname{Sym}_{\mathrm{fin}}(\Z))$ be the automorphism induced by the translation $n\mapsto n+1$ and $\Gamma:=\operatorname{Sym}_{\mathrm{fin}}(\Z)\rtimes_\theta\Z$.
Then $\Gamma$ is two-generated, amenable and i.c.c. Moreover, there exists
$N\in\Sub(\LL(\Gamma))$ such that
\begin{enumerate}
\item $[\LL(\Gamma):N]=2$;
\item $N$ is maximal among the proper von Neumann subalgebras of
$\LL(\Gamma)$;
\item $N$ is not isolated in
$\overline{\{\lambda_gN\lambda_g^*:g\in\Gamma\}}^{\,\mathrm{EM}}$.
\end{enumerate}
\end{proposition}
\begin{proof}
Put $\Lambda:=\operatorname{Sym}_{\mathrm{fin}}(\Z)$ and write $u$ the canonical generator of $\Z$. If $t:=(0\, \,1)$, then $u^ktu^{-k}=(k\, \, k+1)$ for every $k\in\Z$.
Every element of $\Lambda$ is supported in a finite interval of $\Z$, and
the symmetric group of a finite interval is generated by its adjacent
transpositions. Hence $\Lambda=\langle u^ktu^{-k}:k\in\Z\rangle$ and $
\Gamma=\langle u,t\rangle$. In particular, $\Gamma$ is two-generated. Moreover, $\Gamma$ is amenable because both
$\Lambda$ and $\Z$ are.
\medskip
\noindent Every element of $\Gamma$ has a unique expression $\sigma u^m$, with
$\sigma\in\Lambda$ and $m\in\Z$. If $\sigma\neq e$, the elements
$u^k(\sigma u^m)u^{-k}=\theta^k(\sigma)u^m$ take infinitely many distinct
values because a non-empty finite support cannot be invariant
under infinitely many translations. If $\sigma=e$ and $m\neq0$, put $t_k:=u^ktu^{-k}$. Then $t_ku^mt_k^{-1}=t_k\theta^m(t_k)u^m$ and the non-trivial finitary permutations
$t_k\theta^m(t_k)=\theta^k(t\theta^m(t))$ are pairwise distinct for
infinitely many $k$. Thus $\Gamma$ is i.c.c.
\medskip
\noindent Put $M:=\LL(\Gamma)$. For $j\geq1$, let $s_j:=(2j\, \, 2j+1)$ and $H:=\langle s_j:j\geq1\rangle$. Then $H\simeq\bigoplus_{\N}\Z/2\Z$. Under the Fourier identification
$\LL(H)\simeq \LL^\infty(\Omega,\mu)$, where $\Omega:=\{0,1\}^{\N}$ and $\mu$ is the Bernoulli product measure, write every $x\in\Omega$ as
$x=(x_j)_{j\geq1}$, where $x_j\in\{0,1\}$. Under the identification
$\widehat H\simeq\Omega$, the point $x$ corresponds to the character
$\chi_x\in\widehat H$ determined by $\chi_x(s_j)=(-1)^{x_j}$ for all $j\geq1$. Define the binary-expansion map $T:\Omega\to[0,1]$ and the projection
$p\in\LL^\infty(\Omega,\mu)$ by
$$
T(x):=\sum_{j\geq1}2^{-j}x_j,
\quad
C:=\{x\in\Omega:T(x)<1/3\},
\quad
p:=1_C.
$$
For every $n\geq1$ and $\varepsilon_1,\ldots,\varepsilon_n\in\{0,1\}$ the cylinder $\{x\in\Omega:x_j=\varepsilon_j,\ 1\leq j\leq n\}$ has $\mu$-measure $2^{-n}$, and its image under $T$, modulo endpoints, is
the dyadic interval
$$
\left[
\sum_{j=1}^n2^{-j}\varepsilon_j,\
\sum_{j=1}^n2^{-j}\varepsilon_j+2^{-n}
\right].
$$
Hence the push-forward of $\mu$ by $T$ is Lebesgue probability $\lambda$ on $[0,1]$, so
$\tau(p)=1/3$.
\medskip
\noindent For $k\geq1$, define
$g_k:=(2k\, \, 2k+2)(2k+1\, \, 2k+3)
\in\operatorname{Sym}_{\mathrm{fin}}(\Z)$. Conjugation by $g_k$ exchanges $s_k$ and $s_{k+1}$ and fixes every other
$s_j$. Let $\vartheta_k:\Omega\to\Omega$ exchange the $k$-th and
$(k+1)$-st coordinates, and put $p_k:=\lambda_{g_k}p\lambda_{g_k}^*=1_{\vartheta_k(C)}$. Write $1/3=\sum_{j\geq1}2^{-j}d_j$, where $d_j=0$ for $j$ odd and
$d_j=1$ for $j$ even, and put 
$$D_k:=\{x\in\Omega:x_j=d_j,\ 1\leq j\leq k+1\}.$$
Outside the countable set of points having two binary
expansions, comparison with $\frac13=0.010101\ldots$
is lexicographic. Since $\vartheta_k$ changes only the $k$-th and
$(k+1)$-st digits, $T(x)$ and $T(\vartheta_k(x))$ lie on opposite sides of
$1/3$ if and only if the first $k+1$ digits of one of $x$ and
$\vartheta_k(x)$ are $d_1,\ldots,d_{k+1}$. Consequently, modulo null sets,
$$
C\mathbin{\triangle}\vartheta_k(C)
=A_k\sqcup\vartheta_k(A_k),
\quad\text{where }A_k:=
\begin{cases}
D_k\cap C,&k\text{ odd},\\
D_k\cap C^c,&k\text{ even}.
\end{cases}
$$
\noindent For $x\in\Omega$, put
$R_k(x):=\sum_{\ell\geq1}2^{-\ell}x_{k+1+\ell}$ and 
$a_k:=\sum_{j=1}^{k+1}2^{-j}d_j$. Then
$$
T(x)=a_k+2^{-(k+1)}R_k(x)
\quad\text{for all }x\in D_k,
$$
and
$\frac13
=a_k+2^{-(k+1)}r_k$ and $r_k:=
\sum_{\ell\geq1}2^{-\ell}d_{k+1+\ell}
=
\begin{cases}
1/3,&k\text{ odd},\\
2/3,&k\text{ even}.
\end{cases}
$
\noindent Note that, for every Borel subset $B\subseteq[0,1]$, $\mu\bigl(D_k\cap R_k^{-1}(B)\bigr)=2^{-(k+1)}\lambda(B)$ and,
$$
A_k=
\begin{cases}
D_k\cap R_k^{-1}([0,1/3)),&k\text{ odd},\\
D_k\cap R_k^{-1}([2/3,1]),&k\text{ even}.
\end{cases}
$$
Hence $\mu(A_k)=\frac{2^{-(k+1)}}{3}$ and therefore  $\|p_k-p\|_2^2
=
\mu(C\mathbin{\triangle}\vartheta_k(C))
=
\frac{1}{3\cdot2^k}>0
$. Thus $p_k\neq p$ for every $k$ and $p_k\to p$ in $\Vert\cdot\Vert_2$.
\medskip
\noindent Define $
N:=pMp\oplus(1-p)M(1-p)$ and $N_k:=\lambda_{g_k}N\lambda_{g_k}^*$. Lemma~\ref{lem:index} and
\cite[Proposition~3.4]{zhou_maximal_typeI} show that $N$ has
Pimsner--Popa index $2$ and is maximal among the proper von Neumann
subalgebras of $M$. Lemma~\ref{lem:block-diagonal-comparison} gives $N_k\to N$ in the Effros--Mar\'echal topology. Since $\tau(p_k)=1/3$, the equality
$N_k=N$ would imply $p_k=p$ by the same lemma. Hence $N_k\neq N$ for all $k\geq1$ so $N$ is not isolated in its orbit closure.\end{proof}
%%%%%%%%%%%%%%%%%%%%%%%%%%%%%%%%%%%%%%%%%%%%%%%%%%%%%%%%%%%%%%%%%%%%%%%%%%%%%%%%%%%%%%%%%%%%%%%%%%%%%%%%%%%%%%%%%%%%%%%%%%%%%%%%%%%%%%%%%%%%%%%%%%%%%%%%%%
\subsection{Geometric realization via continuous fields}
\label{subsec:continuousfields}
%%%%%%%%%%%%%%%%%%%%%%%%%%%%%%%%%%%%%%%%%%%%%%%%%%%%%%%%%%%%%%%%%%%%%%%%%%%%%%%%%%%%%%%%%%%%%%%%%%%%%%%%%%%%%%%%%%%%%%%%%%%%%%%%%%%%%%%%%%%%%%%%%%%%%%%%%%%%%%%%%%%%%%%%%%%%%%%%%%%%%%

All the constructions above start from a state and read off its centralizer. We
close with a construction of a different, geometric nature.  The interest of this variant is that the
URAs it produces need be neither compact nor discrete, so they escape both
Theorem~\ref{ThmCompactURA} and Theorem~\ref{thm:discrete-uras}; an explicit
instance is given in Example~\ref{ex:discontinuous-fiber-map}.

Let $\Gamma$ be a countable discrete group acting minimally on a
compact Hausdorff space $X$. Let $B$ be a separable unital $C^*$-algebra
with a faithful tracial state $\tau$, let
$\alpha:\Gamma\curvearrowright B$ be trace preserving, and put
$M:=B''$ in the GNS representation associated with $\tau$. Then $M$ has separable predual. We still denote by $\alpha:\Gamma\curvearrowright M$ the canonical trace preserving action. Identify
$C(X)\otimes_{\min}B$ with $C(X,B)$ and equip it with the diagonal action
$$
(g\cdot F)(x):=\alpha_g(F(g^{-1}x))
\quad\text{for all }g\in\Gamma,\ F\in C(X,B),\ x\in X.
$$
Let $C(X)\subseteq A\subseteq C(X,B)$ be a $\Gamma$-invariant
intermediate $C^*$-algebra. For $x\in X$, put
$A_x:=\{F(x):F\in A\}\subseteq B$ and define $\pi:X\to\Sub(M)$ by $\pi(x):=A_x''\in\Sub(M)$. The map $\pi$ is called
\textbf{the fiber map}.
\begin{remark}\label{rmk:fiber-map-equivariance}
The fiber map is $\Gamma$-equivariant. Indeed, since $g\cdot A=A$, one has
$A_{gx}
=\{(g\cdot F)(gx):F\in A\}
=\alpha_g(A_x)
\qquad(g\in\Gamma,\ x\in X)$. Consequently,
$$
\pi(gx)=A_{gx}''=\alpha_g(A_x)''=\alpha_g(A_x'')=\alpha_g(\pi(x)).
$$
\end{remark}

\noindent Unlike the centralizer map, the fibre map is lower semicontinuous.

\begin{lemma}\label{lem:fiber-map-continuity-points}
The set $\operatorname{Cont}(\pi)$ is a dense $G_\delta$ subset of $X$.
\end{lemma}
\begin{proof}
For $b\in M$, put $h_b(x):=\Vert E_{A_x''}(b)\Vert_2$ for all $x\in X$. The projection formula and Kaplansky density give
$h_b(x)=\sup\bigl\{|\tau(a^*b)|:a\in A_x,\ \Vert a\Vert_2\leq1\bigr\}$. Fix $x_0\in X$ and $0\leq r<h_b(x_0)$. There exists $a\in A_{x_0}$ such
that $\Vert a\Vert_2\leq1$ and $|\tau(a^*b)|>r$. Multiplying $a$ by a
scalar sufficiently close to $1$, we may assume that
$\Vert a\Vert_2<1$ while retaining $|\tau(a^*b)|>r$. Choose $f\in A$ satisfying
$f(x_0)=a$. The maps $x\to\Vert f(x)\Vert_2$ and $x\to\tau(f(x)^*b)$
are continuous. Hence, on a neighborhood $V$ of $x_0$,
$\Vert f(x)\Vert_2<1$ and $|\tau(f(x)^*b)|>r$. Since $f(x)\in A_x$, one
has $h_b(x)>r$ for every $x\in V$. Thus $h_b$ is lower semicontinuous and lemma~\ref{lem:semicontinuous-coordinate-map}
applies to $\pi$.
\end{proof}
\noindent The fiberwise picture connects to the state-space construction as follows.
\begin{example}
\label{ex:continuous-field-ura}
By Lemma~\ref{lem:fiber-map-continuity-points},
$\operatorname{Cont}(\pi)$ is a dense $G_\delta$ subset of $X$. For every
$x_0\in\operatorname{Cont}(\pi)$,
Lemma~\ref{lem:Minimality from a continuity point} gives that $\mathcal{U} = \overline{\{ \alpha_g(A_{x_0}'') \mid g \in \Gamma \} }^{\,\mathrm{EM}} \subset \mathrm{Sub}(M)$ is a URA.
\end{example}
\noindent The following example shows simultaneously that the fiber map
need not be continuous and that the construction of
Example~\ref{ex:continuous-field-ura} produces URAs which are neither
compact nor discrete.
\begin{example}
\label{ex:discontinuous-fiber-map}
\noindent Let $Y:=\{0,1\}^{\N}$, let
$\mu:=(\frac12\delta_0+\frac12\delta_1)^{\otimes\N}$, and put
$B_0:=C(Y)$ and $M_0:=\LL^\infty(Y,\mu)$. Denote by $\tau$ the
corresponding trace. For $n\in\N$, consider
$q_n:=1_{\{\omega\in Y:\omega_n=0\}}$. The defining cylinder set is
clopen, so $q_n\in B_0$ is a projection, and the family $(q_n)_{n\in\N}$
is independent with $\tau(q_n)=1/2$. Define
$p_n:=q_0+(1-q_0)q_n$ and
$D_n:=C^*(1,p_n)=W^*(p_n)$ for $n\geq 1$. The two summands defining $p_n$ are orthogonal
projections, so $p_n\in B_0$ is a projection. Independence gives
$\tau(p_n)=\tau(q_0)+\tau((1-q_0)q_n)=3/4$ and, for distinct
$n,m\geq1$,
$$\Vert p_n-p_m\Vert_2^2=\tau((1-q_0)(q_n-q_m)^2)=1/4,$$
because
$\tau(1-q_0)=\tau((q_n-q_m)^2)=1/2$. Since
$q_0\leq p_n$ and
$D_n=\mathbb Cp_n\oplus\mathbb C(1-p_n)$, for distinct
$n,m\geq1$ one has
\begin{equation}\label{eq:continuous-field-separated-fibers}
E_{D_n}(q_0)
=\frac{\tau(q_0)}{\tau(p_n)}p_n
=\frac23p_n,
\quad
\Vert E_{D_n}(q_0)-E_{D_m}(q_0)\Vert_2
=\frac13.
\end{equation}
\medskip
\noindent Let $X:=\mathbb R/\mathbb Z$, fix
$\theta\in\mathbb R\setminus\mathbb Q$, and let $\mathbb Z\curvearrowright X$
be the minimal action generated by
$R(t+\mathbb Z):=t+\theta+\mathbb Z$. For
$n\geq 1$, choose non-zero functions $f_n\in C(X)_+$ with
pairwise disjoint supports and put $V_n:=\{x\in X:f_n(x)>0\}$.
Put
$$
B:=\bigotimes_{k\in\mathbb Z}B_0,
\quad
M:=B''=\overline{\bigotimes_{k\in\mathbb Z}(M_0,\tau)}^{\,\mathrm w},
$$
and denote by $\tau$ the product trace on $M$. Following the notation
of the proof of Lemma~\ref{LemInfiniteTens}, for $c\in M_0$ and
$k\in\mathbb Z$, denote by $c^{(k)}$ the elementary tensor whose
$k$-th coordinate is $c$ and whose other coordinates are equal to $1$.
For a unital subalgebra $Q\subseteq M_0$, put
$Q^{(k)}:=\{c^{(k)}:c\in Q\}$. Let
$\alpha:\mathbb Z\curvearrowright M$ be the Bernoulli shift preserving
$B$ and satisfying $\alpha_m(c^{(k)})=c^{(k+m)}$. For $k\in\mathbb Z$,
$n\geq 1$ and $x\in X$, put
$s_{k,n}(x):=f_n(R^{-k}x)p_n^{(k)}$, and define
$$
A:=C^*\bigl(C(X)1,\ s_{k,n}:k\in\mathbb Z,\
n\geq1\bigr)
\subseteq C(X,B).
$$
One has $m\cdot s_{k,n}=s_{k+m,n}$, so $A$ is
$\mathbb Z$-invariant for the diagonal action defined above.
\medskip
\noindent Define
$$
D(x):=
\begin{cases}
D_n,&x\in V_n,\\
\mathbb C1,&x\notin\displaystyle\bigcup_{n\geq1}V_n.
\end{cases}
$$
Fix $x\in X$. Since the sets $V_n$ are pairwise disjoint, for every
$k\in\mathbb Z$ there is at most one $n\geq1$ such that
$R^{-k}x\in V_n$. Put
$$
C_x:=C^*(1,\ p_n^{(k)}:k\in\mathbb Z,\ n\geq1,\
R^{-k}x\in V_n),
$$
viewed as a $C^*$-subalgebra of $B$. The evaluation map
$\operatorname{ev}_x:A\to B$, $\operatorname{ev}_x(F):=F(x)$, is a
$*$-homomorphism. Hence $A_x=\operatorname{ev}_x(A)
=C^*(1,\ s_{k,n}(x):k\in\mathbb Z,\ n\geq1)$.
If $R^{-k}x\notin V_n$, then $s_{k,n}(x)=0$. If
$R^{-k}x\in V_n$, then
$s_{k,n}(x)=f_n(R^{-k}x)p_n^{(k)}\in C_x$. Thus $A_x\subseteq C_x$.
Conversely, if $R^{-k}x\in V_n$, then $f_n(R^{-k}x)>0$, and hence
$$p_n^{(k)}=f_n(R^{-k}x)^{-1}s_{k,n}(x)\in A_x.$$
Therefore
$C_x\subseteq A_x$, so $A_x=C_x$.
\medskip
\noindent For $k\in\mathbb Z$, the copy of $D(R^{-k}x)$ in the $k$-th
tensor factor is $D(R^{-k}x)^{(k)}$. If $R^{-k}x\in V_n$, then
$D(R^{-k}x)^{(k)}=C^*(1,p_n^{(k)})$; if
$R^{-k}x\notin\bigcup_{n\geq1}V_n$, then
$D(R^{-k}x)^{(k)}=\mathbb C1$. Hence
$$
A_x=C_x=C^*(D(R^{-k}x)^{(k)}:k\in\mathbb Z)
=\bigotimes_{k\in\mathbb Z}D(R^{-k}x)
\subseteq B.
$$
Taking bicommutants gives
\begin{equation}\label{eq:continuous-field-tensor-fiber}
\pi(x):=A_x''
=
\left(\bigotimes_{k\in\mathbb Z}D(R^{-k}x)\right)''
=\overline{\bigotimes_{k\in\mathbb Z}
D(R^{-k}x)}^{\,\mathrm w}.
\end{equation}
\medskip
\noindent
By Remark~\ref{rmk:fiber-map-equivariance}, $\pi$ is
$\mathbb Z$-equivariant. Choose a continuity point $x_0$ of $\pi$,
whose existence follows from
Lemma~\ref{lem:fiber-map-continuity-points}. Lemma~\ref{lem:Minimality from a continuity point}
then implies that
$$
\mathcal U
:=\overline{\{\alpha_m(\pi(x_0)):m\in\mathbb Z\}}^{\,\mathrm{EM}}
=\overline{\{\pi(R^mx_0):m\in\mathbb Z\}}^{\,\mathrm{EM}}
$$
is a URA.
\medskip
\noindent For every $n\geq1$, minimality of $R$ gives
$m_n\in\mathbb Z$ such that $x_n:=R^{m_n}x_0\in V_n$. Thus
$D(x_n)=D_n$. The conditional expectation formula used in the proof of Lemma~\ref{LemInfiniteTens} gives, for every $x\in X$
and $a\in M_0$, $E_{\pi(x)}(a^{(0)})=\bigl(E_{D(x)}(a)\bigr)^{(0)}$. Taking $x=x_n$ and $a=q_0$, and using
\eqref{eq:continuous-field-separated-fibers}, gives
$E_{\pi(x_n)}(q_0^{(0)})=\frac23p_n^{(0)}$.
Since the embedding $M_0\to M$, $a\mapsto a^{(0)}$, preserves the
$L^2$-norm, Equation~\eqref{eq:continuous-field-separated-fibers} also shows that
$$
\Vert
E_{\pi(x_n)}(q_0^{(0)})
-E_{\pi(x_m)}(q_0^{(0)})
\Vert_2
=\frac13
\quad\text{for all }n,m\geq1,\ n\neq m.
$$
Thus
$(\pi(x_n))_{n\geq1}\subseteq\mathcal U$ has no convergent
subsequence. Since
$\Sub(M)$ is metrizable, $\mathcal U$ is not compact. If $\mathcal U$
were discrete,
Proposition~\ref{prop:discrete-minimal-state-space-characterization} $(1)\Rightarrow(3)$,
applied with $\mathcal Z=X$, $\Phi=\pi$ and $z_0=x_0$,
would imply that it is compact. Therefore $\mathcal U$ is
not discrete. Finally, if $\pi$ were continuous, then $\pi(X)$ would be
compact and hence closed in $\Sub(M)$. Equivariance would give
$$
\mathcal U
=\overline{\{\pi(R^mx_0):m\in\mathbb Z\}}^{\,\mathrm{EM}}
\subseteq\pi(X),
$$
so $\mathcal U$ would be a closed subset of the compact space $\pi(X)$.
This contradiction shows that $\pi$ is not continuous.
\end{example}
\bibliographystyle{alphaurl}
\bibliography{biblio}
\end{document}